\documentclass[hidelinks,onefignum,onetabnum]{siamonline220329}

\makeatletter

\renewcommand{\tableofcontents}{%
    \section*{\contentsname}
     \par\noindent
     \vskip 1em
    \@starttoc{toc}%
    \vskip 2em
}

\newcommand*\l@section[2]{%
  \ifnum \c@tocdepth >\z@
    \addpenalty\@secpenalty
    \addvspace{1.0em \@plus\p@}%
    \setlength\@tempdima{1.5em}%
    \begingroup
      \parindent \z@
      \rightskip \@pnumwidth
      \parfillskip -\@pnumwidth
      \leavevmode \bfseries
      \advance\leftskip\@tempdima
      \hskip -\leftskip
      #1\nobreak\hfil
      \nobreak\hb@xt@\@pnumwidth{\hss #2%
                                 \kern-\p@\kern\p@}\par
    \endgroup
  \fi
}

\newcommand*\l@subsection{%
  \@dottedtocline{2}{1.5em}{2.3em}
}

\newcommand*\l@subsubsection{%
  \@dottedtocline{3}{3.8em}{3.2em}
}

\newcommand*\l@paragraph{%
  \@dottedtocline{4}{7.0em}{4.1em}
}

\newcommand*\l@subparagraph{%
  \@dottedtocline{5}{10em}{5em}
}

\makeatother

\usepackage{amsmath}
\usepackage{amssymb}
\usepackage{mathtools}
\usepackage{cases}
\usepackage{mathrsfs}
\usepackage{enumerate}
\usepackage{tikz}
\usepackage{graphicx}
\usepackage{booktabs}

\usepackage{multirow}
\usepackage{makecell}
\usepackage{longtable} % 支持跨页表格
\usepackage{array}
\usepackage{color}
\usepackage{xcolor}

\usepackage{supertabular}
\usepackage{placeins}

\usepackage{siunitx}
\usepackage{cleveref}

\definecolor{darkblue}{rgb}{0.0, 0.0, 0.55}
\definecolor{bordeaux}{rgb}{0.34, 0.01, 0.1}
\newsiamthm{example}{Example}

\usepackage{lipsum}
\usepackage{amsfonts}
\usepackage{graphicx}
\usepackage{epstopdf}
\usepackage{algorithmic}
\ifpdf
  \DeclareGraphicsExtensions{.eps,.pdf,.png,.jpg}
\else
  \DeclareGraphicsExtensions{.eps}
\fi

\newsiamremark{remark}{Remark}
\newsiamremark{hypothesis}{Hypothesis}
\crefname{hypothesis}{Hypothesis}{Hypotheses}
\newsiamthm{claim}{Claim}

\headers{Quaternionic SOS, Moments, and Quaternionic Polynomial Optimization}{Yanqing Liu and Jie Wang}

\title{Quaternionic Sums of Squares, Moments, and Quaternionic Polynomial Optimization\thanks{Submitted to the editors DATE.
}}

\author{Yanqing Liu\thanks{State Key Laboratory of Mathematical Sciences, Academy of Mathematics and Systems Science, Chinese Academy of Sciences, Beijing, China (\email{liuyanqing@amss.ac.cn})} \and Jie Wang\thanks{State Key Laboratory of Mathematical Sciences, Academy of Mathematics and Systems Science, Chinese Academy of Sciences, Beijing, China
  (\email{wangjie212@amss.ac.cn}, \url{https://wangjie212.github.io/jiewang/})}}

\usepackage{amsopn}

\def\R{{\mathbb{R}}}
\def\C{{\mathbb{C}}}
\def\H{{\mathbb{H}}}
\def\N{{\mathbb{N}}}

\def\x{{\mathbf{x}}}

\def\q{{\mathbf{q}}}

\def\y{{\mathbf{y}}}

\def\v{{\mathbf{v}}}

\def\ba{{\mathbf{a}}}

\def\bd{{\mathbf{d}}}
\def\be{{\mathbf{e}}}
\def\bg{{\mathbf{g}}}
\def\bh{{\mathbf{h}}}

\def\bv{{\boldsymbol{v}}}

\newcommand{\bs}{\mathbf{s}}
\def\A{{\mathscr{A}}}

\def\M{{\mathbf{M}}}
\def\i{\hbox{\bf{i}}}
\def\j{\hbox{\bf{j}}}
\def\k{\hbox{\bf{k}}}
\def\supp{\hbox{\rm{supp}}}

\def\Tr{\hbox{\rm{Tr}}}

\def\rank{\hbox{\rm{rank}}}

\def\RR{{\mathcal{R}}}
\def\II{{\mathcal{I}}}
\def\JJ{{\mathcal{J}}}
\def\KK{{\mathcal{K}}}

\def\P{{\mathscr{P}}}
\def\cH{{\mathcal{H}}}
\def\cW{{\mathcal{W}}}

\newcommand{\sym}{\operatorname{Sym}}

\begin{document}

\maketitle

\begin{abstract}
We develop a moment-quaternionic-sum-of-squares (Moment--QSOS) framework for global polynomial optimization over the standard Hamilton quaternions.  The main difficulty is the hybrid algebraic structure of scalar quaternionic variables: multiplication is noncommutative, while scalar evaluations satisfy additional polynomial identities inherited from the four commuting real coordinates of each quaternion.  To capture this structure, we introduce quaternionic sums of squares together with real and quaternionic scalar-identity ideals.  We establish Archimedean Positivstellens\"atze for real-coefficient quaternionic polynomials and for objectives with quaternionic middle coefficients and real-coefficient constraints; a counterexample shows that the latter constraint hypothesis cannot, in general, be removed for the proposed pure-word certificate cone.  On the dual side, we solve the corresponding full real-valued and quaternion-valued moment problems through Riesz--Haviland-type characterizations, and we prove a flat-extension theorem for truncated quaternion-valued moments under an explicit scalar-consistency condition.  These results yield convergent Moment--QSOS hierarchies and finite-exactness criteria based on flatness.  For quaternionic quadratically constrained quadratic programs (QCQPs), we further prove that the first-order quaternionic moment relaxation has the same optimal value as the first-order relaxation of the equivalent real QCQP.  Numerical experiments on quadratic, quartic, and sparse problems, together with applications to quaternion-based maximum margin feature extraction and rotation synchronization, show that the quaternionic formulation attains the same relaxation values as the tested real SOS formulations while requiring substantially shorter computation times in the reported instances.
\end{abstract}

% REQUIRED

% REQUIRED

\tableofcontents

\section{Introduction}
Introduced by Hamilton in 1843, quaternions provide a compact algebraic representation of three-dimensional rotations and coupled multichannel data. Quaternionic models arise naturally in signal processing~\cite{sun2023convex}, image processing~\cite{qi2022quaternion}, computer vision~\cite{yang2019quaternion}, robotics~\cite{chen2024regularization,horn1987closed}, and attitude or rotation estimation~\cite{schmidt2001using,shuster1993survey}.  Optimization problems in these settings often have quaternionic decision variables and objectives or constraints built from quadratic and higher-degree quaternionic polynomial expressions.  Developing optimization methods that preserve this algebraic structure is therefore both theoretically and computationally relevant.

A standard approach to quaternionic optimization is to replace each variable by four real variables ($q_i=x_{i1}+x_{i2}\mathbf{i}+x_{i3}\mathbf{j}+x_{i4}\mathbf{k}$) and solve the resulting real problem.  This reduction makes the extensive toolbox of real optimization immediately available.  A complementary line of work develops quaternionic convex analysis, matrix optimization, duality, and algorithms directly over the quaternion skew field; see, for example,~\cite{flamant2021general}. For polynomial optimization, however, a comparably systematic theory of positivity certificates, moment representations, and convergent semidefinite hierarchies over the standard quaternion algebra remains unavailable.

Over the real numbers, polynomial optimization is closely connected with sums of squares (SOS), moment problems, and semidefinite programming.  The Moment--SOS hierarchy~\cite{lasserre2001global} replaces a generally nonconvex polynomial optimization problem by a sequence of moment and SOS semidefinite relaxations whose optimal values converge to the global optimum under an Archimedean assumption.  Related hierarchies have been developed over the complex field~\cite{josz2018lasserre} and extended to noncommutative polynomial optimization~\cite{klep2024state,pironio2010convergent}. These methods have become effective tools in structured applications, including power systems~\cite{josz2014application,molzahn2014moment}, neural networks~\cite{latorre2020lipschitz}, dynamical systems~\cite{henrion2013convex}, quantum information~\cite{navascues2008convergent}, quantum many-body computation~\cite{wang2026scalable}, and trajectory optimization~\cite{kang2024fast}. See also the monographs~\cite{magron2023sparse,nie2023moment}.

The quaternionic setting does not fit directly into either the commutative or the free noncommutative framework.  Quaternion multiplication is noncommutative, but scalar quaternions are determined by four commuting real coordinates.  Consequently, the free $*$-algebra in formal quaternionic variables contains distinctions that disappear under scalar quaternionic evaluation.  For instance, identities such as $\bar q_iq_i=q_i\bar q_i=|q_i|^2$ hold for scalar quaternions but not in a free $*$-algebra model.  A positivity certificate that ignores these scalar identities can therefore describe a strictly larger matrix-representation problem rather than the intended scalar quaternionic problem. This hybrid structure is one of the main algebraic obstacles to extending the Moment--SOS methodology to quaternions. A second obstacle is the more intricate form of a general quaternionic polynomial: a monomial term may contain multiple coefficients interspersed among the variables, in sharp contrast to the usual real and complex settings.

The purpose of this paper is to develop a Moment--SOS framework for quaternionic polynomial optimization problems (QPOPs) that preserves the standard Hamilton multiplication rules. To overcome the first obstacle, we introduce scalar-identity ideals that encode exactly the relations valid under scalar quaternionic evaluations and combine them with quaternionic sums of squares (QSOS). To address the second obstacle, we restrict attention to real-coefficient quaternionic polynomials and quaternionic middle-coefficient polynomials. These considerations lead to a Moment--QSOS framework in which both the primal moment conditions and the dual positivity certificates distinguish scalar quaternionic tuples from arbitrary matrix representations. In contrast to direct realification, this approach avoids quadrupling the number of variables; the resulting semidefinite models preserve the quaternionic monomial structure and are substantially smaller.
The main contributions of this paper are summarized as follows.
\begin{enumerate}[(i)]
\item \emph{Quaternionic semidefinite programming.}
As a foundation for solving Moment--QSOS relaxations, we derive an economical real reformulation of a quaternionic semidefinite program (SDP). The reformulation avoids explicitly imposing the full block-structure constraints of the standard realification, while preserving the optimal value and allowing recovery of an optimal quaternionic solution.

\item \emph{QSOS certificates and scalar-identity quotients.}
We introduce QSOS polynomials based on left-coefficient quaternionic polynomials and establish quaternionic Gram-matrix characterizations.  We identify the real and quaternionic scalar-identity ideals associated with scalar evaluation and describe generators for them.  For real coefficients, quotienting by the scalar-identity ideal makes the symmetric part central, which permits a reduction to a commutative Archimedean quadratic-module problem.

\item \emph{Archimedean Positivstellens\"atze.}
For real-coefficient data, we prove a Putinar-type Positivstellensatz modulo the real scalar-identity ideal.  For a symmetric objective with quaternionic middle coefficients and real-coefficient constraints, we develop an operator-valued lift and apply an Archimedean $*$-Positivstellensatz to obtain a quaternionic certificate modulo the full scalar-identity ideal.  We also construct a univariate counterexample showing that, for the present pure-word certificate cone, the real-coefficient assumption on the constraints cannot generally be dropped, even after imposing the full scalar-identity quotient and a sphere relation.

\item \emph{Quaternionic moment problems.}
Quaternionic moment theory remains largely underdeveloped; the univariate case was addressed recently in~\cite{taher2026quaternionic}. We initiate the study of a general quaternionic moment theory and establish Riesz--Haviland-type characterizations for full real-valued and quaternion-valued moment sequences. Under an Archimedean hypothesis, these results yield complete moment--localizing matrix criteria. For truncated quaternion-valued moments, we prove a flat-extension theorem: positivity and flatness yield a finite-dimensional representation, and an additional scalar-consistency condition is precisely what is needed to obtain a finitely atomic measure supported on scalar quaternionic tuples. We also provide finite tests for scalar consistency.

\item \emph{Moment--QSOS hierarchies and finite exactness.}
Using the Positivstellens\"atze and moment characterizations, we construct convergent moment and QSOS hierarchies for real-coefficient problems and for objectives with quaternionic middle coefficients and real-coefficient constraints.  We give a flatness criterion for detecting finite convergence.  When the moment matrix has rank one, multiplicativity is automatic and the global optimizer can be extracted directly, without a separate scalar-consistency test. For quaternionic quadratically constrained quadratic programs (QCQPs), we prove that the first-order quaternionic moment relaxation has exactly the same optimal value as the first-order Shor relaxation of the equivalent real QCQP.

\item \emph{Numerical performance.}
Numerical experiments on random quadratic, quartic, and sparse problems show that the quaternionic SOS approach attains the same relaxation values, to the displayed precision, as the corresponding real SOS approach, while the quaternionic SOS approach is substantially faster in the reported instances.  We further illustrate the approach on quaternion-based maximum margin feature extraction and quaternionic rotation synchronization.
\end{enumerate}

The scope of the convergence theory warrants emphasis. The general Archimedean convergence results proved here cover real-coefficient QPOPs and objectives with quaternionic middle coefficients subject to real-coefficient constraints.  The counterexample for general constraints with quaternionic middle coefficients shows that a direct extension using the same pure-word QSOS cone is false; stronger certificate spaces would be required for a fully quaternionic constraint theory.

The remainder of the paper is organized as follows. 
Section~\ref{sec:nota} introduces notation and basic concepts.
Section~\ref{sec:qsdp} develops quaternionic semidefinite programming and its real reformulations.  Section~\ref{sec:qsos} introduces QSOS polynomials and scalar-identity ideals, and proves the Archimedean Positivstellens\"atze, together with the counterexample for general constraints with quaternionic middle coefficients.  The following section treats the full and truncated quaternionic moment problems.  Section~\ref{sec:qpop} constructs the Moment--QSOS hierarchies, proves convergence and finite-exactness results, and compares first-order quaternionic and real relaxations of QCQPs.  Section~\ref{sec:numerical-experiments} reports numerical experiments.  We then present two applications in Section~\ref{sec:app} and conclude with directions for further work in Section~\ref{sec:con}.

\section{Notation and preliminaries}\label{sec:nota}
Let $\N$ and $\N^*$ denote the sets of nonnegative and positive integers, respectively. For a vector $\v$, let $|\v|$ denote its dimension. Let $\mathcal{S}^n$ denote the space of real symmetric $n\times n$ matrices. We write $[a,b]=ab-ba$ for the commutator of $a$ and $b$.

\subsection{Quaternions and quaternionic matrices}
Let $\H$ denote the skew field of quaternions. Every $q\in\H$ has the unique representation
\[
q=q_R+q_I\i+q_J\j+q_K\k,
\]
where $q_R,q_I,q_J,q_K\in\R$ and $\{\mathbf{i},\mathbf{j},\mathbf{k}\}$ are the imaginary units satisfying
\[
\i^2 = \j^2 = \k^2 = \i\j\k = -1, \quad\i\j = \k, \quad\j\k = \i, \quad\k\i = \j,
\]
and
\[
\quad\i\j=-\j\i,\quad\j\k=-\k\j,\quad\k\i=-\i\k.
\]
The conjugate and modulus of $q$ are defined by
$\bar q=q_R-q_I\i-q_J\j-q_K\k$ and $|q|=\sqrt{q\bar q}$, respectively. For a quaternionic vector $\mathbf{v}$, let $\mathbf{v}^*$ denote its conjugate transpose. Likewise, every quaternionic matrix $A\in\H^{m\times n}$ can be written as
\[
A=A_R+A_I\i+A_J\j+A_K\k,
\]
where $A_R,A_I,A_J,A_K\in\R^{m\times n}$. The conjugate transpose of $A$ is denoted by $A^*$. A matrix $A\in\H^{n\times n}$ is Hermitian if $A=A^*$, and the space of $n\times n$ Hermitian quaternionic matrices is denoted by $\cH^n$. A matrix $A\in\cH^n$ is positive semidefinite (PSD), written as $A\succeq0$, if $\mathbf{v}^*A\mathbf{v}\geq0$ for every $\mathbf{v}\in\H^n$. For $A,B\in\H^{n\times n}$, define the trace pairing by $\langle A,B\rangle=\Tr(A^*B)$.
For \(q=q_R+q_I\i+q_J\j+q_K\k\in\H\), define
\(
\RR(q)\coloneqq q_R,
\II(q)\coloneqq q_I,
\JJ(q)\coloneqq q_J,
\KK(q)\coloneqq q_K.
\)
We extend these maps componentwise to quaternionic vectors and matrices. The real part is invariant under cyclic permutation; in particular, $\RR(ab)=\RR(ba)$ for all $a,b\in\H$.
% Viewing $\H^n$ as a real vector space, we equip it with the real inner product
% \begin{equation*}
%     \langle \u, \v\rangle_{\R}=\RR(\u^{*}\v)=\RR(\u)^{\intercal}\RR(\v)+\II(\u)^{\intercal}\II(\v)+\JJ(\u)^{\intercal}\JJ(\v)+\KK(\u)^{\intercal}\KK(\v), \quad \u,\v\in\H^n.
% \end{equation*}

\subsection{Quaternionic polynomials}
Let $\q=(q_1,\ldots,q_n)$ be a tuple of quaternionic variables, and let $\cW$ denote the set of quaternionic words (also called monomials) in $\q$ and $\bar\q$. Suppose that $w=w_1w_2\cdots w_t\in\cW$, where $w_i\in\{q_1,\ldots,q_n,\bar q_1,\ldots,\bar q_n\}$. Whenever $w_iw_{i+1}=|q_j|^2$, we may write $w=|q_j|^2w_1\cdots w_{i-1}w_{i+2}\cdots w_t$. More generally, every quaternionic word can be written uniquely in the form
\begin{equation}\label{normform}
    w=\prod_{i}|q_i|^{2\alpha_i} \tilde{w}
\end{equation}
where $\alpha_i\in\N$ and $\tilde w$ contains no squared-norm factor; we call this expression the \emph{normal form} of $w$. Monomials with the same normal form are identified. The involution of $w=w_1w_2\cdots w_t$ is defined by $w^*=\bar w_t\cdots\bar w_2\bar w_1$. The length of $w$ is denoted by $|w|$. Let $\cW_d$ denote the column vector consisting of all quaternionic monomials of length at most $d$.

Let $\R\langle\q,\bar\q\rangle$ be the real $*$-algebra of quaternionic polynomials with real coefficients. The involution of $p=\sum_{w\in\cW}p_ww\in\R\langle\q,\bar\q\rangle$ is defined by
\(p^*\coloneqq \sum_{w\in\cW}p_{w}w^*\). The degree of $p$ is defined by
$\deg(p)\coloneqq 
\max\left\{
|w|\,:\,p_{w}\neq 0\right\}$.
We call $p$ \emph{symmetric} if $p=p^*$.
Let $\sym\R\langle\q,\bar\q\rangle$ denote the symmetric part of $\R\langle\q,\bar\q\rangle$. 

As quaternion multiplication is noncommutative, quaternionic polynomials with quaternionic coefficients may have more complicated forms. We restrict attention to quaternionic polynomials
of the middle-coefficient form $p=\sum_{a,b\in\cW}a p_{a,b}b$ with $p_{a,b}\in\H$, and denote their space by $\P\langle\q,\bar\q\rangle$.
The degree of $p$ is defined by
$\deg(p)\coloneqq 
\max\left\{
|a|+|b|\,:\,p_{a,b}\neq 0\right\}$.
The involution of $p$ is defined by
\(p^*\coloneqq \sum_{a,b\in\cW}b^*\bar p_{a,b}a^*\).
We call $p$ \emph{symmetric} if $p=p^*$. Every symmetric quaternionic polynomial is real-valued under scalar quaternionic evaluations, and the set of such polynomials is denoted by $\sym\P\langle\q,\bar\q\rangle$.

\section{Quaternionic semidefinite programming}\label{sec:qsdp}
We begin by introducing quaternionic semidefinite programming, which provides a fundamental tool for the Moment--QSOS framework. A quaternionic SDP has the form
\begin{equation}\label{qsdpp}
    \begin{cases}
     \sup\limits_{H\in\cH^n} &\RR(\langle C,H\rangle)\\
      \,\,\,\,\mathrm{s.t.}&\A(H)= b,\\
      &H\succeq0,
    \end{cases}\tag{QSDP-\mbox{$\H$}}
\end{equation}
where $C\in\cH^n$, $b\in\H^m$, and $\A\colon\cH^n\to\H^m$ is the linear operator
$\A(H)\coloneqq(\langle A_i,H\rangle)_{i=1}^m$ associated with matrices $A_1,\ldots,A_m\in\H^{n\times n}$.

\begin{lemma}\label{lm1}
Let $H=H_R+H_I\mathbf{i}+H_J\mathbf{j}+H_K\mathbf{k}\in\cH^n$ with $H_R,H_I,H_J,H_K\in\R^{n\times n}$. Then
\[H\succeq0 \iff
\Lambda(H)\;\coloneqq\;\begin{bmatrix}H_R&-H_I&-H_J&-H_K\\
      H_I&H_R&-H_K&H_J\\
      H_J&H_K&H_R&-H_I\\
      H_K&-H_J&H_I&H_R\end{bmatrix}\succeq0.\]
\end{lemma}

\begin{proof}
For any $\q=\q_R+\q_I\mathbf{i}+\q_J\mathbf{j}+\q_K\mathbf{k}\in \H^n$ with $\q_R,\q_I,\q_J,\q_K\in\R^n$, we have
\begin{align*}
        \q^*H\q&=(\q_R-\q_I\mathbf{i}-\q_J\mathbf{j}-\q_K\mathbf{k})^{\intercal}(H_R+H_I\mathbf{i}+H_J\mathbf{j}+H_K\mathbf{k})(\q_R+\q_I\mathbf{i}+\q_J\mathbf{j}+\q_K\mathbf{k})\\
&=\begin{bmatrix}\q_R;&\q_I;&\q_J;&\q_K\end{bmatrix}^{\intercal}\Lambda(H)\begin{bmatrix}\q_R;&\q_I;&\q_J;&\q_K\end{bmatrix}.
    \end{align*}
Therefore, $H\succeq0$ if and only if $\Lambda(H)\succeq0$.
\end{proof}

Let
\(
H=H_R+H_I\i+H_J\j+H_K\k,\, C=C_R+C_I\i+C_J\j+C_K\k,\, b=b_R+b_I\i+b_J\j+b_K\k
\)
with $H_R,H_I,H_J,H_K,C_R,C_I,C_J,C_K\in\R^{n\times n}$, $b_R,b_I,b_J,b_K\in\R^m$, and introduce four real linear operators $\A_R,\A_I,\A_J,\A_K\colon\R^{n\times n}\rightarrow\R^m$ associated with $\A$ by
\[
\begin{aligned}
\A_R(S) &\coloneqq \big(\langle \RR(A_i),S\rangle\big)_{i=1}^m,\,\,\
\A_I(S) \coloneqq \big(\langle \II(A_i),S\rangle\big)_{i=1}^m,\\
\A_J(S) &\coloneqq \big(\langle \JJ(A_i),S\rangle\big)_{i=1}^m,\,\
\A_K(S) \coloneqq \big(\langle \KK(A_i),S\rangle\big)_{i=1}^m,
\end{aligned}
\qquad \forall S\in\R^{n\times n}.
\]
Separating the four real components of $\A(H)=b$ and applying Lemma~\ref{lm1} yield the following real SDP reformulation of \eqref{qsdpp}:
\begin{equation}\label{qsdpp1}
    \begin{cases}
    \sup\limits_{Y\in\mathcal{S}^{4n}} &\langle C_R,H_R\rangle+\langle C_I,H_I\rangle+\langle C_J,H_J\rangle+\langle C_K,H_K\rangle\\
\,\,\,\,\mathrm{s.t.}&\A_R(H_R)+\A_I(H_I)+\A_J(H_J)+\A_K(H_K)= b_R,\\
      &\A_R(H_I)-\A_I(H_R)-\A_J(H_K)+\A_K(H_J)= b_I,\\
      &\A_R(H_J)+\A_I(H_K)-\A_J(H_R)-\A_K(H_I)= b_J,\\
      &\A_R(H_K)-\A_I(H_J)+\A_J(H_I)-\A_K(H_R)= b_K,\\
      &Y=\Lambda(H)\succeq0.
    \end{cases}\tag{QSDP-\mbox{$\R$}}
\end{equation}

Formulating \eqref{qsdpp1} for a standard SDP solver requires $3n(2n+1)$ additional linear constraints to impose the prescribed block structure of $Y$, which may substantially increase the size of the problem. Following the approach of~\cite{wang2026more}, we next derive a more economical real reformulation.

Let
\begin{equation*}
X=\left[\begin{smallmatrix}X_1&X_5^{\intercal}&X_6^{\intercal}&X_7^{\intercal}\\X_5&X_2&X_8^{\intercal}&X_9^{\intercal}\\X_6&X_8&X_3&X_{10}^{\intercal}\\
X_7&X_9&X_{10}&X_4
\end{smallmatrix}\right]\in\mathcal{S}^{4n}
\end{equation*}
be a block matrix with $n\times n$ blocks $X_i$. Define
\begin{equation*}
% \boxed{\quad
\begin{aligned}
R_X &\coloneqq X_1+X_2+X_3+X_4,\\
I_X &\coloneqq X_5-X_5^{\intercal}+X_{10}-X_{10}^{\intercal},\\
J_X &\coloneqq X_6-X_6^{\intercal}-X_9+X_9^{\intercal},\\
K_X &\coloneqq X_7-X_7^{\intercal}+X_8-X_8^{\intercal}.
\end{aligned}
\end{equation*}
Consider the following real SDP:
\begin{equation}\label{qsdpp2}
    \begin{cases}  \sup\limits_{X\in\mathcal{S}^{4n}} &\langle C_R,R_X \rangle+\langle C_I,I_X\rangle+\langle C_J,J_X\rangle+\langle C_K,K_X\rangle\\
\,\,\,\,\mathrm{s.t.}&\A_R(R_X)+\A_I(I_X)+\A_J(J_X)+\A_K(K_X)=b_R,\\
      &\A_R(I_X)-\A_I(R_X)-\A_J(K_X)+\A_K(J_X)=b_I,\\
      &\A_R(J_X)+\A_I(K_X)-\A_J(R_X)-\A_K(I_X)=b_J,\\      
      &\A_R(K_X)-\A_I(J_X)+\A_J(I_X)-\A_K(R_X)=b_K,\\     
      &X\succeq0.
    \end{cases}\tag{QSDP-\mbox{$\R$}'}
\end{equation}
The following theorem establishes the equivalence of \eqref{qsdpp2} and \eqref{qsdpp}.
\begin{theorem}\label{thm1}
Programs~\eqref{qsdpp2} and~\eqref{qsdpp1} have the same optimal value. Moreover, if $X^{\star}$ is an optimal solution of \eqref{qsdpp2}, then
$H^{\star}=R_{X^{\star}}+I_{X^{\star}}\i+J_{X^{\star}}\j+K_{X^{\star}}\k$ 
is an optimal solution of \eqref{qsdpp}.
\end{theorem}
\begin{proof}
Let $v$ and $v'$ denote the optimal values of \eqref{qsdpp1} and \eqref{qsdpp2}, respectively.
Let $Y$ be feasible for \eqref{qsdpp1}. Direct substitution shows that $X\coloneqq\frac14Y$ is feasible for \eqref{qsdpp2} and has the same objective value. Hence $v\leq v'$. Conversely, let $X$ be any feasible solution of \eqref{qsdpp2}.
Define 
\[ P_1\coloneqq \begin{bmatrix} 0&-1&0&0\\ 1&0&0&0\\ 0&0&0&1\\ 0&0&-1&0 \end{bmatrix}, \, P_2\coloneqq \begin{bmatrix} 0&0&-1&0\\ 0&0&0&-1\\ 1&0&0&0\\ 0&1&0&0 \end{bmatrix}, \,P_3\coloneqq \begin{bmatrix} 0&0&0&-1\\ 0&0&1&0\\ 0&-1&0&0\\ 1&0&0&0 \end{bmatrix}. \]
For \(i=1,2,3\), let 
\(\widehat P_i\coloneqq P_i\otimes I_n. \)
Because each $\widehat P_i$ is orthogonal, the matrices
$X^{(i)}\coloneqq\widehat P_i^{-1}X\widehat P_i$ are PSD. Therefore,
\begin{equation*}
Y\coloneqq X+X^{(1)}+X^{(2)}+X^{(3)}=\begin{bmatrix}R_X&-I_X&-J_X&-K_X\\I_X&R_X&-K_X&J_X\\
        J_X&K_X&R_X&-I_X\\K_X&-J_X&I_X&R_X
        \end{bmatrix}
\end{equation*}
is feasible for \eqref{qsdpp1}.
The corresponding objective values coincide, so $v\geq v'$. Consequently, $v=v'$. The final assertion follows directly from the construction.
\end{proof}

\section{Quaternionic sums of squares}\label{sec:qsos}
We now develop quaternionic sums of squares and the Positivstellensatz machinery needed for the Moment--QSOS hierarchy. The main algebraic issue is that the free word algebra contains distinctions that disappear under evaluation at scalar quaternions. We therefore begin by introducing QSOS polynomials and the scalar-identity ideals that encode these evaluation identities. We then prove Archimedean Positivstellens\"atze, first for real-coefficient polynomials and subsequently for quaternionic-coefficient objectives with real-coefficient constraints. The section concludes with a counterexample showing that the real-coefficient hypothesis on the constraints in the latter result cannot, in general, be removed.

For the remainder of the paper, let $\q=(q_1,\ldots,q_n)$ be a fixed tuple of quaternionic variables.
Let $\mathcal L$ denote the space of left-coefficient quaternionic polynomials
$p=\sum_{w\in\cW}p_w w$, where $p_w\in\H$. Note that $\R\langle\q,\bar\q\rangle\subset\mathcal L$ as a subspace.
\begin{definition}[Quaternionic sum of squares]
Let $p\in\sym\P\langle\q,\bar\q\rangle$.
We call $p$ a QSOS if there exist left-coefficient quaternionic polynomials $p_1,\ldots,p_N\in\mathcal L$ such that $p=\sum_{i=1}^N p_i^*p_i$. When $p\in\sym\R\langle\q,\bar\q\rangle$, we additionally require each $p_i$ to have real coefficients.
\end{definition}
Let $\Sigma^{\H}$ and $\Sigma^{\R}$ denote the cones of QSOS polynomials in $\P\langle\q,\bar\q\rangle$ and $\R\langle\q,\bar\q\rangle$, respectively.

The first basic observation connects the algebraic definition of QSOS with a semidefinite Gram representation, thereby making the QSOS certificates amenable to semidefinite programming.
\begin{lemma}
Let $p\in\sym\P\langle\q,\bar\q\rangle$ satisfy $\deg(p)=2d$. Then $p$ is a QSOS if and only if there exists a PSD quaternionic matrix $G$ such that
$p=\cW_d^*G\cW_d$. If $p\in\sym\R\langle\q,\bar\q\rangle$, then $G$ may be chosen real.
\end{lemma}
\begin{proof}
The result follows directly from the Gram-factorization characterization of positive semidefinite quaternionic matrices.
\end{proof}

A Gram representation alone does not account for identities that hold for scalar quaternionic variables but not in the free $*$-algebra. To distinguish genuine scalar evaluations from purely formal word relations, we now introduce the corresponding scalar-identity quotient. Let $\H\langle\q,\bar\q\rangle$ be the free quaternionic $*$-algebra, and let $\H[x_{11},\ldots,x_{n4}]$ be the polynomial ring in the commuting real variables $x_{11},\ldots,x_{n4}$ with quaternionic coefficients.
Define a unital $*$-homomorphism
\begin{equation}
\rho:\H\langle\q,\bar\q\rangle\to\H[x_{11},\ldots,x_{n4}]
\end{equation}
by fixing $\H$ and setting
\begin{equation}
    \rho(q_i) = x_{i1} + x_{i2}\mathbf{i} + x_{i3}\mathbf{j} + x_{i4}\mathbf{k}, \quad i =1,\ldots,n.
\end{equation}
Necessarily,
\begin{equation}
    \rho(\bar q_i) = x_{i1} - x_{i2}\mathbf{i} - x_{i3}\mathbf{j} - x_{i4}\mathbf{k}, \quad i =1,\ldots,n.
\end{equation}

For $i=1,\ldots,n$, define
\begin{equation}
\chi_{i1} = \frac{q_i+\bar{q}_i}{2},\quad
\chi_{i2} = \frac{\mathbf{i}\bar{q}_i-q_i\mathbf{i}}{2},\quad
\chi_{i3} = \frac{\mathbf{j}\bar{q}_i-q_i\mathbf{j}}{2},\quad
\chi_{i4} = \frac{\mathbf{k}\bar{q}_i-q_i\mathbf{k}}{2}.
\end{equation}
For $i=1,\ldots,n$, $j=1,\ldots,4$, define
\begin{equation}
    C_{ij,1} \coloneqq [\chi_{ij}, \mathbf{i}],\quad C_{ij,2} \coloneqq [\chi_{ij}, \mathbf{j}], \quad C_{ij,3} \coloneqq [\chi_{ij}, \mathbf{k}].
\end{equation}
For $i,k=1,\ldots,n$ and $j,l=1,\ldots,4$, define
\begin{equation}
    D_{ij,kl} \coloneqq [\chi_{ij},\chi_{kl}].
\end{equation}
For $i=1,\ldots,n$, define
\begin{equation}
    R_{i} \coloneqq q_i-\chi_{i1}-\chi_{i2}\mathbf{i}-\chi_{i3}\mathbf{j}-\chi_{i4}\mathbf{k}.
\end{equation}

The kernel $\mathcal I_n^\H\coloneqq\ker(\rho)$ is called the \emph{quaternionic scalar-identity ideal}.  The following lemma gives explicit generators and identifies the resulting quotient with the ordinary coordinate polynomial algebra.
\begin{lemma}
The quaternionic scalar-identity ideal is generated, as a two-sided $*$-ideal in $\H\langle\q,\bar\q\rangle$, by the $C_{ij,s}, D_{ij,kl}, R_i$'s, i.e.,
\begin{equation}
\mathcal I_{n}^\H=\left\langle
\{C_{ij,s}\}_{i,j,s}
\cup
\{D_{ij,kl}\}_{i,j,k,l}
\cup
\{R_i\}_i
\right\rangle_*.
\end{equation}
Moreover, $\H\langle\q,\bar\q\rangle/\mathcal I_{n}^\H\simeq\H[x_{11},\ldots,x_{n4}]$.
\end{lemma}
\begin{proof}
Let $\mathcal J_n^{\H}$ denote the two-sided $*$-ideal on the right-hand side of the displayed identity. A direct calculation using the definitions and the quaternion multiplication rules gives
\begin{equation}\label{eq:coordinate-polynomials-rho}
\rho(\chi_{ij})=x_{ij},
\qquad
 i=1,\ldots,n,\quad j=1,\ldots,4.
\end{equation}
Each $\chi_{ij}$ is self-adjoint. Since the coordinate variables $x_{ij}$ are central in $\H[x_{11},\ldots,x_{n4}]$, \eqref{eq:coordinate-polynomials-rho} implies
\[
\rho(C_{ij,k})=0,
\qquad
\rho(D_{ij,kl})=0.
\]
The definition of $R_i$, together with \eqref{eq:coordinate-polynomials-rho}, also gives
$\rho(R_i)=
\rho(q_i)-x_{i1}-x_{i2}\mathbf i-x_{i3}\mathbf j-x_{i4}\mathbf k
=0$.
Hence,
\begin{equation}\label{eq:JH-contained-kernel}
\mathcal J_n^{\H}\subseteq\ker(\rho)=\mathcal I_{n}^\H.
\end{equation}

Let $\mathcal B
\coloneqq
\H\langle\q,\bar\q\rangle/\mathcal J_n^{\H}$,
and write $\widetilde a$ for the class of $a$ in $\mathcal B$.  The relations $C_{ij,k}=0$ imply that every $\widetilde\chi_{ij}$ commutes with $\mathbf i,\mathbf j,\mathbf k$, and hence with every quaternionic constant.  The relations $D_{ij,kl}=0$ imply that the $\widetilde\chi_{ij}$ commute pairwise.  Therefore there is a unique unital $*$-homomorphism
\begin{equation}\label{eq:coordinate-inverse-map}
\Psi:
\H[x_{11},\ldots,x_{n4}]
\longrightarrow
\mathcal B
\end{equation}
which fixes $\H$ and satisfies
$\Psi(x_{ij})=\widetilde\chi_{ij}$.
Here the involution on the polynomial ring fixes the real variables and restricts to quaternionic conjugation on $\H$.

By \eqref{eq:JH-contained-kernel}, $\rho$ factors through a unital $*$-homomorphism
$\overline\rho:
\mathcal B
\longrightarrow
\H[x_{11},\ldots,x_{n4}]$.
Equation \eqref{eq:coordinate-polynomials-rho} yields
$(\overline\rho\circ\Psi)(x_{ij})=x_{ij}$
and $\overline\rho\circ\Psi$ fixes $\H$, so
$\overline\rho\circ\Psi=
\operatorname{id}_{\H[x_{11},\ldots,x_{n4}]}$.
Conversely, the relation $R_i=0$ in $\mathcal B$ gives
$\widetilde q_i=
\widetilde\chi_{i1}
+\widetilde\chi_{i2}\mathbf i
+\widetilde\chi_{i3}\mathbf j
+\widetilde\chi_{i4}\mathbf k$,
and so
$(\Psi\circ\overline\rho)(\widetilde q_i)
=\widetilde q_i$.
Taking adjoints gives the same identity for $\widetilde{\bar q_i}$, and the composition also fixes all quaternionic constants.  Since these elements generate $\mathcal B$,
$\Psi\circ\overline\rho=
\operatorname{id}_{\mathcal B}$.
Thus $\overline\rho$ is an isomorphism.  In particular, its kernel is zero; hence \eqref{eq:JH-contained-kernel} is an equality:
$\mathcal I_{n}^\H=\mathcal J_n^{\H}$.
The same argument gives the asserted isomorphism
$\H\langle\q,\bar\q\rangle/\mathcal I_{n}^\H
\simeq\H[x_{11},\ldots,x_{n4}]$.
\end{proof}

For the real-coefficient theory below, we need the corresponding identity ideal inside the real free $*$-algebra.  Let $\rho_{\R}$ be the restriction of $\rho$ to $\R\langle\q,\bar\q\rangle$. We call $\mathcal I_n^\R\coloneqq\ker(\rho_{\R})$ the \emph{real scalar-identity ideal}.  In this case the generators admit a particularly simple commutator description.
\begin{lemma}\label{real-scalar-identity}
The real scalar-identity ideal is generated, as a two-sided $*$-ideal in $\R\langle\q,\bar\q\rangle$, by
\begin{equation}
\mathcal I_{n}^\R=\langle[w+w^*,q_i]:w\in\cW,i=1,\ldots,n\rangle_*.
\end{equation}
\end{lemma}
\begin{proof}
Let
\[
\mathcal J_n^{\R}
\coloneqq
\langle[w+w^*,q_i]:w\in\cW,\ i=1,\ldots,n\rangle_*.
\]
For every scalar quaternionic tuple $\ba\in\H^n$,
\[
(w+w^*)(\ba,\bar\ba)
=
w(\ba,\bar\ba)+\overline{w(\ba,\bar\ba)}
=
2\RR\bigl(w(\ba,\bar\ba)\bigr)
\in\R.
\]
Thus it commutes with every $a_i$, so each generator $[w+w^*,q_i]$ vanishes under scalar evaluation. This proves
\begin{equation}\label{eq:JR-contained-IR}
\mathcal J_n^{\R}\subseteq \mathcal I_{n}^\R.
\end{equation}

For the reverse inclusion, we use the classical basis theorem for polynomial identities with involution of the second-order matrix algebra equipped with the symplectic involution: over a field of characteristic zero, its $*$-T-ideal of identities is generated by
\begin{equation}\label{eq:basic-symplectic-star-identity}
[X+X^*,Y].
\end{equation}
See, for example, \cite{levchenko1980identities}. This theorem applies to the Hamilton quaternions with standard conjugation. Indeed,
$\H\otimes_{\R}\C\simeq M_2(\C)$,
and the complexification of quaternionic conjugation is the symplectic involution on $M_2(\C)$.  Moreover, a real-coefficient $*$-polynomial is an identity of $\H$ if and only if its complexification is an identity of $M_2(\C)$: after choosing a real basis of $\H$, every matrix entry of an evaluation is a polynomial in the $4n$ real coordinates, and vanishing on all of $\R^{4n}$ implies vanishing identically after complexification.  Thus the real $*$-identities of $(\H,\overline{\phantom q})$ are exactly the consequences of \eqref{eq:basic-symplectic-star-identity}.

It remains to verify that all such consequences belong to $\mathcal J_n^{\R}$. Let
$a,b\in\R\langle\q,\bar\q\rangle$.  Since $a+a^*$ is a real linear combination of elements $w+w^*$, the generators of $\mathcal J_n^{\R}$ give
$[a+a^*,q_i]\in\mathcal J_n^{\R}$.
Because $\mathcal J_n^{\R}$ is $*$-closed,
$[w+w^*,q_i]^*
=-[w+w^*,\bar q_i]\in\mathcal J_n^{\R}$,
so also $[a+a^*,\bar q_i]\in\mathcal J_n^{\R}$.
Therefore, $a+a^*$ commutes modulo $\mathcal J_n^{\R}$ with every letter.  If
$b=v_1\cdots v_m$ is a word, the commutator identity
\[
[a+a^*,b]
=
\sum_{i=1}^m
v_1\cdots v_{i-1}
[a+a^*,v_i]
v_{i+1}\cdots v_m
\]
shows that
\begin{equation}\label{eq:substituted-basic-star-identity-in-JR}
[a+a^*,b]\in\mathcal J_n^{\R}.
\end{equation}
By real linearity, the same conclusion holds for arbitrary $b$. Consequently, every $*$-endomorphic image of \eqref{eq:basic-symplectic-star-identity}, and hence every element of the $*$-T-ideal that it generates, belongs to $\mathcal J_n^{\R}$.

Now let $p\in \mathcal I_{n}^\R$.  By definition, $p$ is a real $*$-polynomial identity of the Hamilton quaternions.  The identity-basis theorem and \eqref{eq:substituted-basic-star-identity-in-JR} imply $p\in\mathcal J_n^{\R}$.  Together with \eqref{eq:JR-contained-IR}, this proves
\[
\mathcal I_{n}^\R
=
\mathcal J_n^{\R}
=
\langle[w+w^*,q_i]:w\in\cW,\ i=1,\ldots,n\rangle_*.
\]
\end{proof}

With the scalar-identity quotients in place, we can formulate positivity relative to polynomial constraints.  For $\bg\coloneqq\{g_1,\ldots,g_m\}\subset\sym\P\langle\q,\bar\q\rangle$ and $\bh\coloneqq\{h_1,\ldots,h_\ell\}\subset\P\langle\q,\bar\q\rangle$, define the quaternionic basic semialgebraic set
\begin{equation}\label{eq:Kgh-real}
K_{\bg,\bh} \coloneqq  \{\q \in \H^n \mid g_i(\q,\bar\q) \ge 0,i=1,\ldots,m,\,h_j(\q,\bar\q)=0,j=1,\ldots,\ell\}.
\end{equation}

\subsection{The Archimedean Positivstellensatz for quaternionic polynomials with real coefficients}\label{sec:putinar-real}
We first establish a Putinar-type theorem in the real-coefficient setting.  Quotienting by $\mathcal I_n^\R$ has a decisive effect here: symmetric classes become central, so the positivity problem reduces to an Archimedean quadratic module in a commutative real algebra.  The proof proceeds by establishing this centrality, identifying the relevant positive characters with scalar quaternionic evaluations, and then applying Jacobi's representation theorem.  Throughout this subsection, all quaternionic polynomials are assumed to have real coefficients. We define
\begin{equation}\label{eq:QR-module}
\begin{split}
\mathcal Q^{\R}(\bg,\bh)\coloneqq\bigg\{&\sigma_0+\frac12\sum_{i=1}^m(\sigma_ig_i+g_i\sigma_i)
+\sum_{j=1}^\ell\sum_{\alpha}
\left(\tau_{j\alpha}h_j\nu_{j\alpha}
+\nu_{j\alpha}^*h_j^*\tau_{j\alpha}^*\right):\\
&\sigma_i\in\Sigma^{\R},\quad
\tau_{j\alpha},\nu_{j\alpha}\in\R\langle\q,\bar\q\rangle\bigg\}.
\end{split}
\end{equation}
We call $\mathcal Q^{\R}(\bg,\bh)$ \emph{Archimedean} if there exists $R>0$ such that
$R-\sum_{i=1}^n|q_i|^2\in \mathcal Q^{\R}(\bg,\bh)$.
Let
\[
\mathcal F_n\coloneqq\R\langle\q,\bar\q\rangle,
\qquad
\pi:\mathcal F_n\longrightarrow
\mathcal A\coloneqq\mathcal F_n/\mathcal I_{n}^\R
\]
be the canonical quotient map, and set
$\mathcal B\coloneqq\sym\mathcal A\coloneqq\{a\in\mathcal A:a^*=a\}$.
Let \(\mathcal J_\bh^{\R}\) be the two-sided \(*\)-ideal of \(\mathcal A\) generated by
\(\pi(h_1),\ldots,\pi(h_\ell)\), and define
\[
\mathcal A_\bh\coloneqq\mathcal A/\mathcal J_\bh^{\R},
\qquad
\mathcal B_\bh\coloneqq\sym\mathcal A_\bh.
\]
We denote the quotient map by
$\eta:\mathcal A\longrightarrow\mathcal A_\bh$
and, after establishing surjectivity below, denote its restriction to the symmetric part by
$\eta|_{\mathcal B}:\mathcal B\longrightarrow\mathcal B_\bh$.
For \(p\in\mathcal F_n\), write
\(\widehat p\coloneqq\eta(\pi(p))\).
Define
\begin{equation}\label{eq:Mh-real}
\mathcal Q_\bh(\bg)\coloneqq\bigg\{\sum_{\alpha}p_{0\alpha}^*p_{0\alpha}
+\sum_{i=1}^m\sum_{\alpha}p_{i\alpha}^*\widehat g_i p_{i\alpha}:p_{i\alpha}\in\mathcal A_\bh\bigg\}.
\end{equation}
We call $\mathcal Q_\bh(\bg)$ \emph{Archimedean} if there exists $R>0$ such that
$R-\sum_{i=1}^n|\widehat q_i|^2\in \mathcal Q_\bh(\bg)$.
For a ring $\mathcal D$, let $Z(\mathcal D)$ denote its center.  The first structural step is to show that the symmetric quotient is indeed commutative and that no symmetric class is lost when the equality constraints are imposed.
\begin{lemma}\label{lem:symmetric-central-quotient}
Every element of \(\mathcal B_\bh\) has a symmetric representative in \(\mathcal B\). Moreover,
\[
\mathcal B\subseteq Z(\mathcal A)
\qquad\text{and}\qquad
\mathcal B_\bh\subseteq Z(\mathcal A_\bh).
\]
In particular, \(\mathcal B_\bh\) is a commutative unital real algebra and
\(\eta|_{\mathcal B}:\mathcal B\to\mathcal B_\bh\) is surjective.
\end{lemma}
\begin{proof}
Let \(\beta\in\mathcal B_\bh\) and choose \(a\in\mathcal A\) such that
\(\eta(a)=\beta\). Since \(\beta=\beta^*\),
\(\eta(a-a^*)=0\). Hence
\[
b\coloneqq\frac{a+a^*}{2}\in\mathcal B
\qquad\text{and}\qquad
\eta(b)=\beta.
\]
Thus every symmetric class in $\mathcal A_\bh$ has a symmetric representative in $\mathcal A$, and $\eta|_{\mathcal B}$ is surjective.

We next establish centrality. Let $b\in\mathcal B$ and $a\in\mathcal A$, and choose representatives
\(\widetilde b,\widetilde a\in\mathcal F_n\). Since
\(\pi(\widetilde b)=b=b^*=\pi(\widetilde b^*)\), one has
\(\widetilde b-\widetilde b^*\in \mathcal I_{n}^\R\). Hence, for every
\(\q\in\H^n\),
\[
\widetilde b(\q,\bar\q)
=\widetilde b^*(\q,\bar\q)
=\overline{\widetilde b(\q,\bar\q)}\in\R.
\]
Consequently,
\[
[\widetilde b,\widetilde a](\q,\bar\q)=0
\qquad\text{for every }\q\in\H^n.
\]
By the definition of $\mathcal I_n^\R$, it follows that
\([\widetilde b,\widetilde a]\in \mathcal I_{n}^\R\), and therefore
\([b,a]=0\) in \(\mathcal A\). Thus \(\mathcal B\subseteq Z(\mathcal A)\).
If \(\beta\in\mathcal B_\bh\), choose the symmetric representative
\(b\in\mathcal B\) constructed above. Then \(b\) is central in \(\mathcal A\), so
\(\beta=\eta(b)\) is central in \(\mathcal A_\bh\). This proves
\(\mathcal B_\bh\subseteq Z(\mathcal A_\bh)\), and hence \(\mathcal B_\bh\) is commutative.
\end{proof}

Centrality makes commutative representation theory available, but we must still determine which positive characters of the symmetric quotient correspond to genuine scalar quaternionic points.  The following lemma provides precisely this identification; its proof also gives an intrinsic description of the symmetric scalar algebra.

\begin{lemma}\label{lem:scalar-character-real}
Let \(\phi:\mathcal B\to\R\) be a unital real-algebra homomorphism satisfying
\begin{equation}\label{eq:positive-character-condition}
\phi(a^*a)\geq0
\qquad\text{for every }a\in\mathcal A.
\end{equation}
Then there exists \(\ba=(a_1,\ldots,a_n)\in\H^n\) such that $\phi(b)=b(\ba,\bar\ba)$ for every $b\in\mathcal B$.
Consequently, every character of $\mathcal B$ satisfying \eqref{eq:positive-character-condition} arises from a scalar quaternionic evaluation.
\end{lemma}
\begin{proof}
We suppress the quotient map \(\pi\) in this proof and regard
\(q_i,\bar q_i\) as elements of \(\mathcal A\). Put
\[
x_i\coloneqq\frac{q_i+\bar q_i}{2},
\qquad
v_i\coloneqq\frac{q_i-\bar q_i}{2}.
\]
Then \(x_i^*=x_i\) and \(v_i^*=-v_i\). For skew-adjoint elements \(u,v\in\mathcal A\), define
\[
u\mathbin{\boldsymbol\cdot}v
\coloneqq-\frac12(uv+vu),
\qquad
u\mathbin{\boldsymbol\times}v
\coloneqq\frac12(uv-vu).
\]
Under scalar quaternionic evaluations, these operations coincide with the Euclidean dot and cross products, respectively, of the corresponding imaginary vectors. Define
\begin{equation}\label{sec4:lm3:eq1}
s_{ij}\coloneqq v_i\mathbin{\boldsymbol\cdot}v_j,
\qquad
c_{ij}\coloneqq v_i\mathbin{\boldsymbol\times}v_j,
\qquad
t_{ijk}\coloneqq v_i\mathbin{\boldsymbol\cdot}c_{jk}.
\end{equation}
The elements \(x_i,s_{ij},t_{ijk}\) are symmetric, whereas
\(v_i,c_{ij}\) are skew-adjoint.

We first prove that
\begin{equation}\label{eq:B-generators}
\mathcal B
=
\R[x_i,s_{ij},t_{ijk}:1\leq i,j,k\leq n],
\end{equation}
where the right-hand side denotes the subalgebra of \(\mathcal A\) generated by the displayed elements, rather than a freely generated polynomial ring. These generators satisfy relations; for example,
$t_{ijk}^2=
\det\bigl(s_{\alpha\beta}\bigr)_{\alpha,\beta\in\{i,j,k\}}$.
Let $\mathcal B'\coloneqq
\R[x_i,s_{ij},t_{ijk}:1\leq i,j,k\leq n]$ and let \(\mathcal C\) be the \(\mathcal B'\)-module generated by
\(v_i\) and \(c_{ij}\). By Lemma~\ref{lem:symmetric-central-quotient},
\(\mathcal B'\) is central in \(\mathcal A\).
The usual three-dimensional vector identities yield
\begin{align}
 c_{ij}\mathbin{\boldsymbol\cdot}c_{kr}
 &=s_{ik}s_{jr}-s_{ir}s_{jk},\label{eq:dot-cross-identities-a}\\
 v_i\mathbin{\boldsymbol\times}c_{jk}
 &=s_{ik}v_j-s_{ij}v_k,\label{eq:dot-cross-identities-b}\\
 c_{ij}\mathbin{\boldsymbol\times}v_k
 &=s_{ik}v_j-s_{jk}v_i,\label{eq:dot-cross-identities-c}\\
 c_{ij}\mathbin{\boldsymbol\times}c_{kr}
 &=t_{ijr}v_k-t_{ijk}v_r.\label{eq:dot-cross-identities-d}
\end{align}
Together with \eqref{sec4:lm3:eq1}
and the symmetry of the dot product, these identities imply
\[
\mathcal C\mathbin{\boldsymbol\cdot}\mathcal C\subseteq\mathcal B',
\qquad
\mathcal C\mathbin{\boldsymbol\times}\mathcal C\subseteq\mathcal C.
\]
For completeness, identities~\eqref{eq:dot-cross-identities-a}--\eqref{eq:dot-cross-identities-d} hold for vectors in $\R^3$ and therefore under every scalar quaternionic evaluation. The differences between left-hand sides and right-hand sides consequently belong to $\mathcal I_n^\R$, so the identities hold in $\mathcal A$.
For \(a,b\in\mathcal B'\) and \(u,v\in\mathcal C\), centrality of
\(\mathcal B'\) and the identity
\(uv=-u\mathbin{\boldsymbol\cdot}v+u\mathbin{\boldsymbol\times}v\) give
\[
(a+u)(b+v)
=
\bigl(ab-u\mathbin{\boldsymbol\cdot}v\bigr)
+
\bigl(av+bu+u\mathbin{\boldsymbol\times}v\bigr)
\in\mathcal B'+\mathcal C.
\]
Thus $\mathcal B'+\mathcal C$ is a subalgebra of $\mathcal A$. Since
$q_i=x_i+v_i$ and $\bar q_i=x_i-v_i$,
it contains all generators of \(\mathcal A\), and hence
\(\mathcal A=\mathcal B'+\mathcal C\). If \(z=a+u\in\mathcal A\), with
\(a\in\mathcal B'\) and \(u\in\mathcal C\), is symmetric, then
$a+u=z=z^*=a-u$. Hence $u=0$, and therefore $z\in\mathcal B'$, proving
\eqref{eq:B-generators}.

We now reconstruct a scalar quaternionic tuple from \(\phi\). Set
\[
X_i\coloneqq\phi(x_i),
\qquad
S_{ij}\coloneqq\phi(s_{ij}),
\qquad
T_{ijk}\coloneqq\phi(t_{ijk}),
\]
and let \(S=\left(S_{ij}\right)_{i,j=1}^n\). For every
\(\boldsymbol{\lambda}=(\lambda_1,\ldots,\lambda_n)\in\R^n\),
\begin{equation*}
\boldsymbol{\lambda}^{\intercal}S\boldsymbol{\lambda}
=\phi\!\left(\left(\sum_{i=1}^n\lambda_i v_i\right)^*\left(\sum_{j=1}^n\lambda_j v_j\right)\right)\geq0
\end{equation*}
by \eqref{eq:positive-character-condition}. Hence \(S\succeq0\).
For any indices \(i_1,\ldots,i_4\), the identity $\det\bigl(s_{i_a i_b}\bigr)_{a,b=1}^4=0$
holds under any scalar quaternionic evaluation because the matrix is the Gram matrix of four vectors in \(\R^3\). It therefore holds in $\mathcal B$, and applying $\phi$ shows that every $4\times4$ principal minor of $S$ vanishes. Since \(S\succeq0\), it follows that
$\rank(S)\leq3$.
Thus, there exists \(U\in\R^{3\times n}\) such that
\(S=U^{\intercal}U\). Denote the columns of \(U\) by
\(\mathbf u_1,\ldots,\mathbf u_n\in\R^3\).
For ordered triples \(I=(i_1,i_2,i_3)\) and
\(J=(j_1,j_2,j_3)\), write
\[
T_I\coloneqq T_{i_1i_2i_3},
\qquad
S_{I,J}\coloneqq(S_{i_a j_b})_{a,b=1}^3.
\]
The scalar triple-product identity gives the polynomial identity
\begin{equation}\label{eq:triple-Gram-identity}
t_{i_1i_2i_3}t_{j_1j_2j_3}
=
\det(s_{i_a j_b})_{a,b=1}^3
\end{equation}
in \(\mathcal B\). Applying \(\phi\) yields
\begin{equation}\label{eq:T-Gram-relation}
T_IT_J=\det S_{I,J}.
\end{equation}

If \(\rank(S)\leq2\), then
$T_I^2=\det S_{I,I}=0$
for every \(I\), while all determinants
\(\det(\mathbf u_{i_1},\mathbf u_{i_2},\mathbf u_{i_3})\) also vanish. Hence
$T_I=\det(\mathbf u_{i_1},\mathbf u_{i_2},\mathbf u_{i_3})$ for every $I$.
Suppose now that \(\rank(S)=3\). Choose a triple \(I_0\) such that
\(\det S_{I_0,I_0}>0\). Then
\(T_{I_0}\neq0\). If necessary, replace every vector $\mathbf u_i$ by $O\mathbf u_i$, where $O$ is an orthogonal matrix with $\det O=-1$. This transformation preserves the Gram matrix and allows us to arrange $\det(\mathbf u_{I_0})=T_{I_0}$,
where \(\mathbf u_I\) denotes the \(3\times3\) matrix whose columns are indexed by \(I\). For every ordered triple \(J\),
\begin{equation*}
T_{I_0}T_J=\det S_{I_0,J}=\det(\mathbf u_{I_0}^{\intercal}\mathbf u_J)=\det(\mathbf u_{I_0})\det(\mathbf u_J)=T_{I_0}\det(\mathbf u_J).
\end{equation*}
Since \(T_{I_0}\neq0\), it follows that
$T_J=\det(\mathbf u_J)$ for every $J$.
Finally, write
\(\mathbf u_i=((u_i)_1,(u_i)_2,(u_i)_3)\) and define
\[
a_i\coloneqq
X_i+(u_i)_1\i+(u_i)_2\j+(u_i)_3\k,
\qquad i=1,\ldots,n.
\]
At $\ba=(a_1,\ldots,a_n)$, the elements $x_i,s_{ij},t_{ijk}$ take the values $X_i,S_{ij},T_{ijk}$, respectively. Since these elements generate $\mathcal B$ by \eqref{eq:B-generators}, we have $\phi(b)=b(\ba,\bar\ba)$ for every $b\in\mathcal B$.
\end{proof}

The preceding two lemmas provide the bridge from scalar quaternionic positivity to the character space of a commutative Archimedean quadratic module.  We can therefore apply Jacobi's representation theorem to obtain the desired Positivstellensatz.
\begin{theorem}\label{thm:putinar-real-equalities}
Let $f,g_1,\ldots,g_m\in\sym\R\langle\q,\bar\q\rangle$, 
$h_1,\ldots,h_\ell\in\R\langle\q,\bar\q\rangle$.
Assume that \(\mathcal Q^{\R}(\bg,\bh)\) is Archimedean. If
$f>0$ on $K_{\bg,\bh}$, then $\widehat f\in \mathcal Q_\bh(\bg)$.
Equivalently, $f\in\mathcal Q^{\R}(\bg,\bh)+\sym\mathcal I_{n}^\R$.
\end{theorem}
\begin{proof}
By Lemma~\ref{lem:symmetric-central-quotient}, every summand in \eqref{eq:Mh-real} belongs to $\mathcal B_\bh$. The set $\mathcal Q_\bh(\bg)$ contains $1$ and is closed under addition and multiplication by squares from $\mathcal B_\bh$. Indeed, if \(b=b^*\in\mathcal B_\bh\), then centrality gives
\[
b^2a^*a=(ab)^*(ab),\qquad b^2a^*\widehat g_i a=(ab)^*\widehat g_i(ab).
\]
Thus \(\mathcal Q_\bh(\bg)\) is a quadratic module in the commutative algebra
\(\mathcal B_\bh\). 
% It is also closed under conjugation by arbitrary elements of
% \(\mathcal A_\bh\):
% \[
% c^*(a^*a)c=(ac)^*(ac), \qquad c^*(a^*\widehat g_i a)c=(ac)^*\widehat g_i(ac).
% \]
% This additional closure property is not needed for Jacobi's theorem, but will be useful for the ball criterion below.
Let
\[
X_{\mathcal Q_\bh(\bg)}\coloneqq
\left\{\chi\in\operatorname{Hom}_{\R\text{-alg}}(\mathcal B_\bh,\R)\,\middle|\,
\chi(\mathcal Q_\bh(\bg))\subseteq[0,\infty)\right\}.
\]
For $\ba\in K_{\bg,\bh}$, the corresponding scalar evaluation annihilates $\mathcal I_n^\R$ and \(\mathcal J_\bh^{\R}\), and therefore factors through $\mathcal A_\bh$. Its restriction to \(\mathcal B_\bh\), denoted by
\(\operatorname{ev}_{\ba}^{\bh}\), is a real character. Moreover,
\[
\operatorname{ev}_{\ba}^{\bh}\!\left(
\sum_{\alpha}p_{0\alpha}^*p_{0\alpha}
+\sum_{i,\alpha}p_{i\alpha}^*\widehat g_i p_{i\alpha}
\right)
=
\sum_{\alpha}|p_{0\alpha}(\ba)|^2
+\sum_{i,\alpha}g_i(\ba)|p_{i\alpha}(\ba)|^2
\geq0.
\]
Hence the map
\begin{equation}\label{eq:evaluation-surjection}
K_{\bg,\bh}\longrightarrow X_{\mathcal Q_\bh(\bg)},
\qquad
\ba\longmapsto\operatorname{ev}_{\ba}^{\bh},
\end{equation}
is well defined.

We next prove that \eqref{eq:evaluation-surjection} is surjective. Let
\(\chi\in X_{\mathcal Q_\bh(\bg)}\) and define
\[
\widetilde\chi\coloneqq\chi\circ\eta|_{\mathcal B}:\mathcal B\longrightarrow\R.
\]
For every \(a\in\mathcal A\), $\widetilde\chi(a^*a)
=\chi\bigl(\eta(a)^*\eta(a)\bigr)\geq0$,
because \(\eta(a)^*\eta(a)\in \mathcal Q_\bh(\bg)\). By
Lemma~\ref{lem:scalar-character-real}, there exists \(\ba\in\H^n\) such that
\begin{equation}\label{eq:chi-lift-evaluation}
\widetilde\chi(b)=b(\ba,\bar\ba)
\qquad\text{for every }b\in\mathcal B.
\end{equation}
For each equality constraint,
\[
0
=
\chi\!\left(\eta|_{\mathcal B}\bigl(\pi(h_j)^*\pi(h_j)\bigr)\right)
=
\widetilde\chi\bigl(\pi(h_j)^*\pi(h_j)\bigr)
=
|h_j(\ba,\bar\ba)|^2,
\]
so \(h_j(\ba,\bar\ba)=0\). Moreover, $\widehat g_i\in\mathcal Q_\bh(\bg)$, and hence
$g_i(\ba,\bar\ba)=
\widetilde\chi(\pi(g_i))
=\chi(\widehat g_i)
\geq0$.
Therefore \(\ba\in K_{\bg,\bh}\). Since
\(\eta|_{\mathcal B}\) is surjective, \eqref{eq:chi-lift-evaluation} implies
\(\chi=\operatorname{ev}_{\ba}^{\bh}\) on \(\mathcal B_\bh\). This proves surjectivity. The map need not be injective: for example, simultaneous conjugation of all quaternionic variables by a unit quaternion leaves every element of $\mathcal B$ invariant.

It follows from the strict positivity assumption that
\[
\chi(\widehat f)>0
\qquad\text{for every }\chi\in X_{\mathcal Q_\bh(\bg)}.
\]
If \(-1\in \mathcal Q_\bh(\bg)\), then \(\mathcal Q_\bh(\bg)=\mathcal B_\bh\). Indeed, for any
\(b\in\mathcal B_\bh\),
\[
b
=
\left(\frac{b+1}{2}\right)^2
+
(-1)\left(\frac{b-1}{2}\right)^2
\in \mathcal Q_\bh(\bg).
\]
Thus $\widehat f\in\mathcal Q_\bh(\bg)$. Now suppose that $-1\notin\mathcal Q_\bh(\bg)$. The Archimedean property of $\mathcal Q^{\R}(\bg,\bh)$ passes to $\mathcal Q_\bh(\bg)$; hence $\mathcal Q_\bh(\bg)$ is a proper Archimedean quadratic module. Jacobi's representation theorem \cite[Theorem 4]{jacobi2001representation} states that, for such a quadratic module in a commutative unital real algebra, every element that is strictly positive on all real characters nonnegative on the module belongs to the module. Applying it to \(\widehat f\) gives
$\widehat f\in \mathcal Q_\bh(\bg)$.
% Choose representatives \(p_{0\alpha},p_{j\alpha}\in\mathcal F_n\) of the finitely many factors occurring in such a representation of \(\widehat f\). Since
% \[
% \ker(\eta\circ\pi)
% =
% \mathcal I_{n}^\R+\langle h_1,\ldots,h_\ell\rangle_*,
% \]
% we obtain exactly \eqref{eq:putinar-congruence-real}.
It remains to relate this quotient formulation to
\(\mathcal Q^{\R}(\bg,\bh)\). Modulo \(\mathcal I_{n}^\R\), every symmetric element is central. Hence
\[
p^*g_ip
\equiv
 g_ip^*p
=
\frac12\bigl((p^*p)g_i+g_i(p^*p)\bigr)
\pmod{\mathcal I_{n}^\R}.
\]
Moreover, the symmetric part of \(\mathcal J_\bh^{\R}\) consists of finite sums of terms $\tau h_j\nu+\nu^*h_j^*\tau^*$.
Thus, $\widehat f\in \mathcal Q_\bh(\bg)$ is equivalent to \(f\in\mathcal Q^{\R}(\bg,\bh)+\sym\mathcal I_{n}^\R\).
\end{proof}

% \begin{corollary}[Ball criterion]\label{cor:putinar-ball-real}
% The Archimedean hypothesis in Theorem~\ref{thm:putinar-real-equalities} is satisfied if, for some \(R>0\),
% \begin{equation}\label{eq:ball-in-module-real}
% R-\sum_{i=1}^n\widehat q_i^*\widehat q_i\in \mathcal Q_\bh(\bg).
% \end{equation}
% \end{corollary}
% \begin{proof}
% Modulo \(\mathcal I_{n}^\R\), one has \(q_iq_i^*=q_i^*q_i\). It follows from
% \eqref{eq:ball-in-module-real} that, for every letter
% \(z\in\{\widehat q_1,\widehat q_1^*,\ldots,\widehat q_n,\widehat q_n^*\}\),
% \[
% R-z^*z\in \mathcal Q_\bh(\bg).
% \]
% If
% \(C-u^*u\in \mathcal Q_\bh(\bg)\) and
% \(D-v^*v\in \mathcal Q_\bh(\bg)\), then the closure under conjugation proved above gives
% \[
% CD-(uv)^*(uv)
% =
% v^*(C-u^*u)v+C(D-v^*v)
% \in \mathcal Q_\bh(\bg).
% \]
% By induction, every word \(w\) satisfies
% \(C_w-w^*w\in \mathcal Q_\bh(\bg)\) for some \(C_w>0\).
% Furthermore,
% \[
% 2C+2D-(u+v)^*(u+v)
% =
% 2(C-u^*u)+2(D-v^*v)+(u-v)^*(u-v)
% \in \mathcal Q_\bh(\bg).
% \]
% After absorbing real coefficients into the summands, another induction shows that, for every \(a\in\mathcal A_\bh\), there exists \(N_a>0\) such that
% \[
% N_a-a^*a\in \mathcal Q_\bh(\bg).
% \]
% In particular, if \(b=b^*\in\mathcal B_\bh\), then
% \(N_b-b^2\in \mathcal Q_\bh(\bg)\), and therefore
% \[
% \frac{N_b+1}{2}\pm b
% =
% \frac12\bigl(N_b-b^2+(b\pm1)^2\bigr)
% \in \mathcal Q_\bh(\bg).
% \]
% Hence \(\mathcal Q_\bh(\bg)\) is Archimedean.
% \end{proof}

\subsection{The Archimedean Positivstellensatz for quaternionic polynomials with quaternionic coefficients}\label{sec:putinar-quaternionic}
We next extend the positivity certificate to objectives with quaternionic coefficients.  The commutative reduction used in the preceding subsection is no longer directly available, because symmetric middle-coefficient quaternionic polynomials need not lie in the real symmetric part.  We instead lift such polynomials to operator-valued polynomial functions on the real Hilbert space $\H$ and apply an Archimedean Positivstellensatz for real $*$-algebras.  The constraints are still required to have real coefficients; this makes their operator lifts scalar and preserves the feasible set under simultaneous quaternionic conjugation.  Specifically, throughout this subsection,
\[
f\in\sym\P\langle\q,\bar\q\rangle,\quad g_i\in\sym\R\langle\q,\bar\q\rangle,\,i=1,\ldots,m,
\quad
h_j\in\R\langle\q,\bar\q\rangle,\,j=1,\ldots,\ell.
\]

Define
\begin{equation}\label{eq:QH-module}
\begin{split}
\mathcal Q^{\H}(\bg,\bh)\coloneqq\bigg\{&
\sigma_0+\frac12\sum_{i=1}^m(\sigma_i g_i+g_i\sigma_i)+\sum_{j=1}^\ell\sum_{\alpha}
\left(\tau_{j\alpha}h_j\nu_{j\alpha}^*
+\nu_{j\alpha}h_j^*\tau_{j\alpha}^*\right):\\
&\sigma_i\in\Sigma^{\H},\quad
\tau_{j\alpha},\nu_{j\alpha}\in\mathcal L
\bigg\}.
\end{split}
\end{equation}
We say that $\mathcal Q^{\H}(\bg,\bh)$ satisfies the \emph{Archimedean condition} if there exists $R>0$ such that
\begin{equation}\label{eq:QH-archimedean}
R-\sum_{i=1}^n|q_i|^2
\in\mathcal Q^{\H}(\bg,\bh).
\end{equation}

We first construct the operator algebra in which the lifted positivity problem will live.  Regard $\H$ as a four-dimensional real Hilbert space with inner product
$\langle\xi,\zeta\rangle_{\R}\coloneqq\RR(\bar\xi\zeta)$.
For \(a,b\in\H\), let
\[
L_a(\xi)\coloneqq a\xi,
\qquad
R_b(\xi)\coloneqq\xi b.
\]
Then $(L_aR_b)^*=L_{\bar a}R_{\bar b}$ and $L_aR_b=R_bL_a$.
\begin{lemma}
The $\R$-linear span
\begin{equation}\label{eq:operator-word-algebra}
\mathscr A_n
\coloneqq
\operatorname{span}_{\R}
\left\{
\q\longmapsto L_cR_{w(\q)}:
 c\in\H,\ w\in\cW
\right\}.
\end{equation}
is a unital real $*$-algebra of operator-valued polynomial functions on $\H^n$.
\end{lemma}
\begin{proof}
For $c\in\H$ and $w\in\cW$, the map
$\q\mapsto L_cR_{w(\q)}$ is an operator-valued polynomial function on $\H^n$. Indeed, after identifying $\H$ with $\R^4$, the four real components of $w(\q)$ are ordinary real polynomials in the $4n$ real coordinates of $\q$, and the real matrix representing $L_cR_{w(\q)}$ depends real-linearly on these components. Hence every element of $\mathscr A_n$ is an operator-valued polynomial function.

By definition, $\mathscr A_n$ is a real vector space. We next verify closure under multiplication.  Let \(c,d\in\H\) and \(u,v\in\cW\).  For every
\(\xi\in\H\) and \(\q\in\H^n\),
\begin{equation*}
\bigl(L_cR_{u(\q)}\bigr)
\bigl(L_dR_{v(\q)}\bigr)\xi=
c\bigl(d\xi v(\q)\bigr)u(\q)=
cd\xi v(\q)u(\q).
\end{equation*}
Therefore
\begin{equation}\label{eq:operator-word-product}
\bigl(L_cR_{u(\q)}\bigr)
\bigl(L_dR_{v(\q)}\bigr)
=
L_{cd}R_{(vu)(\q)}.
\end{equation}
Note that the word order is reversed because right multiplication is an antihomomorphism:
$R_{u(\q)}R_{v(\q)}=R_{v(\q)u(\q)}$.
Since $vu\in\cW$ after reduction to normal form if necessary, the right-hand side of \eqref{eq:operator-word-product} belongs to $\mathscr A_n$. Bilinearity therefore establishes closure under multiplication. The empty word $1$ belongs to $\cW$ and $L_1R_1=I_{\H}$, so $\mathscr A_n$ is unital.

It remains to verify closure under the involution.  With respect to
\(\langle\xi,\zeta\rangle_{\R}=\RR(\bar\xi\zeta)\), one has
\[
L_c^*=L_{\bar c},
\qquad
R_a^*=R_{\bar a}
\qquad(c,a\in\H).
\]
Indeed, for \(\xi,\zeta\in\H\),
\[
\langle\xi,L_c\zeta\rangle_{\R}
=
\RR(\bar\xi c\zeta)
=
\langle L_{\bar c}\xi,\zeta\rangle_{\R},
\]
while, using \(\RR(ab)=\RR(ba)\),
\begin{equation*}
\langle\xi,R_a\zeta\rangle_{\R}=
\RR(\bar\xi\zeta a)=
\RR(a\bar\xi\zeta)=
\RR\!\left(\overline{\xi\bar a}\,\zeta\right)
=\langle R_{\bar a}\xi,\zeta\rangle_{\R}.
\end{equation*}
Since left and right multiplications commute,
\begin{equation*}
\bigl(L_cR_{w(\q)}\bigr)^*=
R_{w(\q)}^*L_c^*=
R_{\overline{w(\q)}}L_{\bar c}=
L_{\bar c}R_{\overline{w(\q)}}=
L_{\bar c}R_{w^*(\q)}.
\end{equation*}
Because \(w^*\in\cW\), the right-hand side belongs to \(\mathscr A_n\).
Real-linearity of the adjoint then gives \(A^*\in\mathscr A_n\) for every
\(A\in\mathscr A_n\).
Thus $\mathscr A_n$ is a unital real $*$-algebra of operator-valued polynomial functions on $\H^n$.
\end{proof}

Having established that $\mathscr A_n$ is a real $*$-algebra, we now define the two lifts used in the Positivstellensatz proof: $\Theta$ for left-coefficient polynomials and $\Phi$ for symmetric middle-coefficient quaternionic polynomials.
For a left-coefficient polynomial $p=\sum_w c_w w\in\mathcal L$, define its operator lift by
\begin{equation}\label{eq:Theta-lift}
\Theta(p)(\q)
\coloneqq
\sum_wL_{c_w}R_{w(\q)}\in\mathscr A_n.
\end{equation}
By construction, $\mathscr A_n=\Theta(\mathcal L)$.
For $p=\sum_{u,v\in\mathcal W} up_{u,v}v\in\sym\P\langle\q,\bar\q\rangle$,
define
\begin{equation}\label{eq:Phi-lift-definition}
T_p(\q)
\coloneqq
\sum_{u,v\in\mathcal W} L_{p_{u,v}}R_{v(\q)u(\q)},
\qquad
\Phi(p)\coloneqq\frac12(T_p+T_p^*)\in\mathscr A_n.
\end{equation}
For each equality constraint, set
$H_j(\q)\coloneqq R_{h_j(\q)}=\Theta(h_j)(\q)$,
and let
\begin{equation}\label{eq:operator-equality-ideal}
\mathscr J_\bh^{\H}
\coloneqq
\langle H_1,\ldots,H_\ell\rangle_*
\subseteq\mathscr A_n.
\end{equation}
Let
$\mathcal J_{\bh}^{\R}\coloneqq
\langle h_1,\ldots,h_\ell\rangle_*^{\R\langle\q,\bar\q\rangle}$ and
$\mathcal J_\bh^{\H}\coloneqq
\langle h_1,\ldots,h_\ell\rangle_*^{\H\langle\q,\bar\q\rangle}$
denote the two-sided $*$-ideals generated by the equality constraints in $\R\langle\q,\bar\q\rangle$ and $\H\langle\q,\bar\q\rangle$, respectively.  The next lemma is the technical bridge between polynomial certificates and the operator algebra: it identifies scalar evaluations through quadratic forms, sends QSOS terms to operator squares, and tracks both equality and scalar-identity remainders.

\begin{lemma}\label{lem:operator-lift-equality-ideal}
For every \(\q\in\H^n\), every
\(p\in\sym\P\langle\q,\bar\q\rangle\), and every nonzero
\(\xi=|\xi|u\in\H\) with \(|u|=1\),
\begin{equation}\label{eq:Phi-quadratic-form}
\langle\xi,\Phi(p)(\q)\xi\rangle_{\R}
=
|\xi|^2p(u\q\bar u),
\end{equation}
where $u\q\bar u\coloneqq(uq_1\bar u,\ldots,uq_n\bar u)$.
Moreover,
\begin{align}
\Phi(p^*p)&=\Theta(p)^*\Theta(p),
&&p\in\mathcal L,\label{eq:Phi-square}\\
\Phi\left(\frac12(\sigma g+g\sigma)\right)
&=\sum_s\Theta(p_s)^*G\Theta(p_s),
&&\sigma=\sum_sp_s^*p_s,\,
G(\q)\coloneqq R_{g(\q)},\label{eq:Phi-weighted}\\
\ker\!\left(\Phi|_{\sym\P\langle\q,\bar\q\rangle}\right)
&=\mathcal I_{n}^\H\cap\sym\P\langle\q,\bar\q\rangle,\label{eq:Phi-kernel}\\
\Phi\left(\mathcal J_\bh^{\H}\cap\sym\P\langle\q,\bar\q\rangle\right)
&=
\mathscr J_\bh^{\H}\cap\sym\mathscr A_n,\label{eq:equality-ideal-lift}
\end{align}
for every \(g\in\sym\R\langle\q,\bar\q\rangle\). Finally, for any $p\in\mathcal J_\bh^{\H}\cap\sym\P\langle\q,\bar\q\rangle$, there exist finitely many
\(\tau_{j\alpha},\nu_{j\alpha}\in\mathcal L\) such that
\begin{equation}\label{eq:explicit-equality-operator-lift}
p-
\sum_{j=1}^{\ell}\sum_{\alpha}
\left(
\tau_{j\alpha}h_j\nu_{j\alpha}^*
+\nu_{j\alpha}h_j^*\tau_{j\alpha}^*
\right)
\in\mathcal I_{n}^\H\cap\sym\P\langle\q,\bar\q\rangle,
\end{equation}
where in each summand one of
\(\tau_{j\alpha},\nu_{j\alpha}\) may be chosen to have real coefficients.
\end{lemma}

\begin{proof}
For a summand $acb$ in the representation of $p\in\sym\P\langle\q,\bar\q\rangle$, where $a,b\in\mathcal W$ and $c\in\H$, cyclic invariance of the real part gives
\begin{equation*}
\left\langle
u,L_cR_{b(\q)a(\q)}u
\right\rangle_{\R}=
\RR\!\left(\bar ucub(\q)a(\q)\right)=
\RR\!\left(
a(u\q\bar u)cb(u\q\bar u)
\right).
\end{equation*}
Summing over the presentation of $p$, using $p=p^*$, and multiplying by $|\xi|^2$ proves \eqref{eq:Phi-quadratic-form}.

For $p=\sum_w c_w w\in\mathcal L$, the telescoping identity $w(u\q\bar u)=u w(\q)\bar u$ gives
\begin{equation}\label{eq:Theta-intertwining}
\Theta(p)(\q)u
=p(u\q\bar u)u.
\end{equation}
It follows that
$\|\Theta(p)(\q)\xi\|^2=
|\xi|^2|p(u\q\bar u)|^2$.
Comparing this identity with \eqref{eq:Phi-quadratic-form} proves \eqref{eq:Phi-square}.

If \(g\in\sym\R\langle\q,\bar\q\rangle\), then its real coefficients imply
$g(u\q\bar u)=u g(\q)\bar u$.
Since \(g(\q)\in\R\), this reduces to
\begin{equation}\label{eq:g-conjugation-invariance}
g(u\q\bar u)=g(\q).
\end{equation}
Consequently, $G(\q)=R_{g(\q)}=g(\q)I_{\H}$ is a scalar self-adjoint element of $\mathscr A_n$. Comparing the quadratic forms on both sides, using \eqref{eq:Theta-intertwining} and \eqref{eq:g-conjugation-invariance}, proves \eqref{eq:Phi-weighted}.

If \(p\in \mathcal I_{n}^\H\cap\sym\P\langle\q,\bar\q\rangle\), then \eqref{eq:Phi-quadratic-form} gives
\(\Phi(p)=0\). Conversely, if \(\Phi(p)=0\), then
\eqref{eq:Phi-quadratic-form} with \(\xi=1\) gives
$p(\q)=0$ for every $\q\in\H^n$,
so \(p\in \mathcal I_{n}^\H\). This proves \eqref{eq:Phi-kernel}.

To prove \eqref{eq:equality-ideal-lift}, identify the full quaternionic free algebra with the real free product $\H\langle\q,\bar\q\rangle\simeq
\H*_{\R}\mathcal F_n$.
By the universal property of the free product,
\begin{equation}\label{eq:free-product-quotient}
\bigl(\H*_{\R}\mathcal F_n\bigr)/\mathcal J_\bh^{\H}
\simeq
\H*_{\R}\bigl(\mathcal F_n/\mathcal J_{\bh}^{\R}\bigr).
\end{equation}
Write
\[
\H=\R1\oplus\H_0,
\qquad
\H_0\coloneqq\operatorname{span}_{\R}\{\i,\j,\k\}.
\]
We have
\begin{equation}\label{eq:P-reduced-sector}
\P\langle\q,\bar\q\rangle
\simeq
\mathcal F_n
\oplus
\bigl(
\mathcal F_n\otimes_{\R}\H_0\otimes_{\R}\mathcal F_n
\bigr).
\end{equation}
Under the quotient map in \eqref{eq:free-product-quotient}, the kernel on $\P\langle\q,\bar\q\rangle$ is
\begin{equation}\label{eq:J-intersection-P}
\mathcal J_\bh^{\H}\cap\P\langle\q,\bar\q\rangle
\simeq
\mathcal J_{\bh}^{\R}
\oplus
\bigl(
\mathcal J_{\bh}^{\R}\otimes\H_0\otimes\mathcal F_n
\bigr)+
\bigl(
\mathcal F_n\otimes\H_0\otimes\mathcal J_{\bh}^{\R}
\bigr).
\end{equation}
The last two summands need not be disjoint.
The map
\[
\mathcal F_n\longrightarrow\mathscr A_n,
\qquad
a\longmapsto R_{a(\q)},
\]
is a unital $*$-antihomomorphism. In particular,
\[
a\in\mathcal J_{\bh}^{\R}
\quad\Longrightarrow\quad
R_{a(\q)}\in\mathscr J_\bh^{\H}.
\]
For a basic term \(acb\), with \(a,b\in\mathcal W\) and \(c\in\H\),
$T_{acb}(\q)=L_cR_{b(\q)a(\q)}
=L_cR_{a(\q)}R_{b(\q)}$.
Thus \eqref{eq:J-intersection-P} implies
\[
p\in \mathcal J_\bh^{\H}\cap\P\langle\q,\bar\q\rangle
\quad\Longrightarrow\quad
T_p\in\mathscr J_\bh^{\H}.
\]
If \(p=p^*\), then
$\Phi(p)=\frac12(T_p+T_p^*)\in\mathscr J_\bh^{\H}$,
which proves the inclusion
$\Phi(\mathcal J_\bh^{\H}\cap\sym\P\langle\q,\bar\q\rangle)
\subseteq
\mathscr J_\bh^{\H}\cap\sym\mathscr A_n$.
For the reverse inclusion, let $D=D^*\in\mathscr J_\bh^{\H}$. Each elementary ideal term has the form
$AH_j^\varepsilon B, \varepsilon\in\{1,*\}$.
By linearity, it suffices to consider
$A=L_cR_{u(\q)}, B=L_dR_{v(\q)}$,
with \(c,d\in\H\) and \(u,v\in\mathcal W\). Then
\begin{equation}\label{eq:operator-ideal-basic-lift}
AH_j^\varepsilon B
=L_{cd}R_{v(\q)h_j^\varepsilon(\q)u(\q)}
=
\Theta\bigl((cd)vh_j^\varepsilon u\bigr).
\end{equation}
The polynomial $(cd)vh_j^\varepsilon u$ belongs to $\mathcal J_\bh^{\H}\cap\mathcal L$. Summing the corresponding terms therefore yields $t\in\mathcal J_\bh^{\H}\cap\mathcal L$ such that $D=\Theta(t)$.
Since \(D=D^*\), the polynomial
$p\coloneqq\frac12(t+t^*)$
belongs to \(\mathcal J_\bh^{\H}\cap\sym\P\langle\q,\bar\q\rangle\), and
$\Phi(p)=\frac12\bigl(\Theta(t)+\Theta(t)^*\bigr)
=D$.
This proves the reverse inclusion.

We finally prove \eqref{eq:explicit-equality-operator-lift}. By \eqref{eq:Phi-kernel}, it suffices to find finitely many
\(\tau_{j\alpha},\nu_{j\alpha}\in\mathcal L\) such that
\begin{equation}\label{eq:explicit-equality-operator-lift-1}
\Phi(p)=\Phi\left(
\sum_{j=1}^{\ell}\sum_{\alpha}
\left(
\tau_{j\alpha}h_j\nu_{j\alpha}^*
+\nu_{j\alpha}h_j^*\tau_{j\alpha}^*
\right)\right).
\end{equation}
Write $\Phi(p)\in\mathscr J_\bh^{\H}\cap\sym\mathscr A_n$ as a finite sum of elementary ideal terms
$\Phi(p)=\sum_{\beta=1}^{N}A_\beta H_{j_\beta}^{\varepsilon_\beta}B_\beta$, 
$\varepsilon_\beta\in\{1,*\}$.
After expanding \(A_\beta,B_\beta\) linearly, it suffices to take
$A_\beta=L_{c_\beta}R_{u_\beta(\q)}$,
$B_\beta=L_{d_\beta}R_{v_\beta(\q)}$,
where \(c_\beta,d_\beta\in\H\) and
\(u_\beta,v_\beta\in\mathcal W\). By
\eqref{eq:operator-ideal-basic-lift},
\[
A_\beta H_{j_\beta}^{\varepsilon_\beta}B_\beta
=\Theta(t_\beta),
\qquad
t_\beta
\coloneqq
(c_\beta d_\beta)v_\beta
h_{j_\beta}^{\varepsilon_\beta}u_\beta
\in\mathcal L.
\]
Since \(\Phi(p)=\Phi(p)^*\),
\begin{equation}\label{eq:selfadjoint-operator-ideal-symmetrization}
\Phi(p)
=\frac12\sum_{\beta=1}^{N}
\left(\Theta(t_\beta)+\Theta(t_\beta)^*\right).
\end{equation}
If \(\varepsilon_\beta=1\), set $\tau_\beta
\coloneqq
\frac12(c_\beta d_\beta)v_\beta$, 
$\nu_\beta\coloneqq u_\beta^*$.
Then
$\tau_\beta h_{j_\beta}\nu_\beta^*
+\nu_\beta h_{j_\beta}^*\tau_\beta^*
=\frac12(t_\beta+t_\beta^*)$.
If \(\varepsilon_\beta=*\), set instead
$\tau_\beta\coloneqq u_\beta^*$,
$\nu_\beta\coloneqq
\frac12(c_\beta d_\beta)v_\beta$;
the same identity holds with the two terms interchanged. In either case, one multiplier has real coefficients, so each symmetrized term belongs to $\sym\P\langle\q,\bar\q\rangle$. For every $t\in\mathcal L$, the definition of $\Phi$ gives
$\Phi\left(\frac12(t+t^*)\right)
=\frac12\left(\Theta(t)+\Theta(t)^*\right)$.
Summing this identity and using
\eqref{eq:selfadjoint-operator-ideal-symmetrization} proves
\eqref{eq:explicit-equality-operator-lift-1}.
\end{proof}

The lift identities above reduce the desired polynomial certificate to membership in an operator quadratic module.  We now invoke the real Archimedean representation theorem for $*$-algebras, which supplies the corresponding operator-level Positivstellensatz~\cite{cimpric2009representation}.
\begin{theorem}[\cite{cimpric2009representation}, Theorem~4.4]\label{star-Pos}
Let \(\mathscr C\) be a unital real $*$-algebra and let
\(\mathscr M\subseteq\sym\mathscr C\) be a proper (meaning $-I\notin \mathscr M$) Archimedean quadratic module.
Suppose $F\in\sym\mathscr C$.
If $\pi(F)\succ0$ for every irreducible $\mathscr M$-positive $*$-representation $\pi$ of \(\mathscr C\), then $F\in\mathscr M$.
\end{theorem}

Applying this representation theorem to the quotient operator algebra, and then pulling the resulting certificate back through $\Phi$, yields the quaternionic-coefficient Positivstellensatz.
\begin{theorem}\label{thm:putinar-quaternionic}
Let $f\in\sym\P\langle\q,\bar\q\rangle$,
$g_1,\ldots,g_m\in\sym\R\langle\q,\bar\q\rangle$,
$h_1,\ldots,h_\ell\in\R\langle\q,\bar\q\rangle$.
Assume that $\mathcal Q^{\H}(\bg,\bh)$ satisfies the Archimedean condition. If $f>0$ on $K_{\bg,\bh}$, then
$f\in\mathcal Q^{\H}(\bg,\bh)+\left(\mathcal I_{n}^\H\cap\sym\P\langle\q,\bar\q\rangle\right)$.
\end{theorem}

\begin{proof}
If \(1\in\mathcal J_{\bh}^{\R}\), then
\(\mathcal J_{\bh}^{\H}=\H\langle\q,\bar\q\rangle\), so
\(f\in \mathcal J_{\bh}^{\H}\cap\sym\P\langle\q,\bar\q\rangle\).  Applying
\eqref{eq:explicit-equality-operator-lift} to \(f\) immediately yields
\[
f\in
\mathcal Q^{\H}(\bg,\bh)
+
\left(\mathcal I_n^{\H}\cap\sym\P\langle\q,\bar\q\rangle\right).
\]
Henceforth, assume that $1\notin\mathcal J_{\bh}^{\R}$. We divide the remainder of the proof into six steps.
\medskip

\noindent{\bf Step 1: lifting to the operator quadratic module \(\mathscr M\).}
For each inequality constraint, set
\[
G_i(\q)
\coloneqq
R_{g_i(\q)}
=g_i(\q)I_{\H}.
\]
The second equality follows from the real-coefficient assumption and the symmetry of $g_i$. Work in the quotient real $*$-algebra
$\mathscr B
\coloneqq
\mathscr A_n/\mathscr J_\bh^{\H}$.
We suppress quotient brackets when no confusion can arise. Define
\begin{equation}\label{eq:operator-quadratic-module}
\mathscr M
\coloneqq
\left\{
\sum_sA_{0s}^*A_{0s}
+
\sum_{i=1}^m\sum_sA_{is}^*G_iA_{is}:
A_{is}\in\mathscr B
\right\}.
\end{equation}
This cone contains the identity operator $I$, is closed under addition, and satisfies
\[
C^*\mathscr MC\subseteq\mathscr M,
\qquad
\forall C\in\mathscr B.
\]
At this stage, the module may be improper; that is, $-I\in\mathscr M$ has not yet been excluded.

Choose a representation witnessing the Archimedean condition
\eqref{eq:QH-archimedean}:
\begin{equation}\label{eq:QH-archimedean-expanded}
R-\sum_{i=1}^n|q_i|^2
=
\sigma_0^{A}
+
\frac12\sum_{i=1}^m
\left(\sigma_i^{A}g_i+g_i\sigma_i^{A}\right)
+E_A,
\end{equation}
where $\sigma_i^{A}\in\Sigma^{\H}$ and $E_A$ is the finite aggregate of equality-multiplier terms from \eqref{eq:QH-module}. The left-hand side and the weighted QSOS terms on the right-hand side belong to $\sym\P\langle\q,\bar\q\rangle$; hence the aggregate equality term satisfies
$E_A\in \mathcal J_{\bh}^{\H}\cap\sym\P\langle\q,\bar\q\rangle$.
Therefore Lemma~\ref{lem:operator-lift-equality-ideal} gives
\(\Phi(E_A)\in\mathscr J_{\bh}^{\H}\). Put $r_i(\q)\coloneqq R_{q_i}$. Applying \(\Phi\) to
\eqref{eq:QH-archimedean-expanded} and passing to the quotient yields
\begin{equation}\label{eq:operator-ball-condition}
RI-\sum_{i=1}^nr_i^*r_i
\in\mathscr M.
\end{equation}
Here \eqref{eq:Phi-square} treats the QSOS term, whereas \eqref{eq:Phi-weighted} gives
\[
\Phi\left(\frac12(\sigma_i^{A}g_i+g_i\sigma_i^{A})\right)
=
\sum_s\Theta(p_{is}^{A})^*G_i\Theta(p_{is}^{A})
\in\mathscr M
\]
when \(\sigma_i^{A}=\sum_s(p_{is}^{A})^*p_{is}^{A}\).

\medskip
\noindent{\bf Step 2: proving Archimedeanness of \(\mathscr M\).}
For every \(i\),
\begin{equation}\label{eq:individual-generator-bound}
RI-r_i^*r_i
=
\left(RI-\sum_{k=1}^nr_k^*r_k\right)
+
\sum_{k\neq i}r_k^*r_k
\in\mathscr M.
\end{equation}
Pointwise,
$r_i^*r_i=r_ir_i^*=|q_i|^2I$,
so the same bound holds with \(r_ir_i^*\).
Let \(\widetilde{\mathscr B}\) be the set of all \(A\in\mathscr B\) for which there exist \(C,D>0\) satisfying $CI-A^*A\in\mathscr M$,
$DI-AA^*\in\mathscr M$.
We claim that $\widetilde{\mathscr B}$ is a unital $*$-subalgebra. It is closed under real scalar multiplication, and the two defining bounds immediately imply closure under involution. To verify closure under sums and products, suppose
$\alpha^2I-A^*A\in\mathscr M$, $\beta^2I-B^*B\in\mathscr M$.
Then
\begin{equation*}
2(\alpha^2+\beta^2)I-(A+B)^*(A+B)
=
2(\alpha^2I-A^*A)
+2(\beta^2I-B^*B)\notag+(A-B)^*(A-B)
\in\mathscr M,
\end{equation*}
and
\begin{equation*}
\alpha^2\beta^2I-(AB)^*(AB)
=\alpha^2(\beta^2I-B^*B)
+B^*(\alpha^2I-A^*A)B
\in\mathscr M.
\end{equation*}
Applying the sum identity to \(A^*,B^*\), and the product identity to \(B^*,A^*\), gives the corresponding bounds for
\((A+B)(A+B)^*\) and \((AB)(AB)^*\). Thus
\(\widetilde{\mathscr B}\) is a unital $*$-subalgebra.
The generators \(r_i,r_i^*\) belong to \(\widetilde{\mathscr B}\) by
\eqref{eq:individual-generator-bound}. The constant left multiplications do as well, since
$L_a^*L_a=L_aL_a^*=|a|^2I$.
These elements generate $\mathscr B$; hence $\widetilde{\mathscr B}=\mathscr B$.
In particular, for every \(B=B^*\in\mathscr B\), there exists \(N>0\) such that $NI-B^2\in\mathscr M$.
Therefore
\begin{equation}\label{eq:order-unit-from-square-bound}
\frac{N+1}{2}I\pm B
=
\frac12\left(NI-B^2+(B\pm I)^*(B\pm I)\right)
\in\mathscr M.
\end{equation}
Hence \(\mathscr M\) is Archimedean.

\medskip
\noindent{\bf Step 3: classifying irreducible representations of $\mathscr B$.}
Let
$\pi:\mathscr B\to B(\mathscr H_{\R})$ be an irreducible bounded unital $\mathscr M$-positive real $*$-representation, where $\mathscr H_{\R}$ is a real Hilbert space and $B(\mathscr H_{\R})$ denotes the algebra of bounded real-linear operators on $\mathscr H_{\R}$. Composing with the quotient map, we also regard \(\pi\) as a representation of \(\mathscr A_n\) annihilating \(\mathscr J_\bh^{\H}\).
Let
\[
x_j\coloneqq\frac{r_j+r_j^*}{2},
\qquad
v_j\coloneqq\frac{r_j-r_j^*}{2},
\qquad
c_{jk}\coloneqq-\frac12(v_jv_k+v_kv_j).
\]
Pointwise,
$x_j(\q)=\RR(q_j)I$, $c_{jk}(\q)=
\left\langle\operatorname{Im}(q_j),
\operatorname{Im}(q_k)\right\rangle I$, where $\operatorname{Im}(q)\coloneqq (b,c,d)$ for $q=a+b\i+c\j+d\k\in\H$.
Therefore, all $x_j$ and $c_{jk}$ are central and self-adjoint.
By the real Schur lemma, the commutant of an irreducible real $*$-representation is a real division $*$-algebra of type $\R$, $\C$, or $\H$. The self-adjoint part of each such division algebra is $\R I$; hence every central self-adjoint element acts as a real scalar.
Hence there exist real numbers \(\alpha_j\) and \(\gamma_{jk}\) such that
\begin{equation}\label{eq:central-scalar-images}
\pi(x_j)=\alpha_jI,
\qquad
\pi(c_{jk})=\gamma_{jk}I.
\end{equation}
Put $V_j\coloneqq\pi(v_j)$, $\Gamma\coloneqq(\gamma_{jk})_{j,k=1}^n$.
Then
\begin{equation}\label{eq:V-Clifford-Gram}
V_j^*=-V_j,
\qquad
V_jV_k+V_kV_j=-2\gamma_{jk}I.
\end{equation}
For every \(\bs\in\R^n\),
\begin{equation}\label{eq:Gamma-positive}
\left(\sum_js_jV_j\right)^*
\left(\sum_js_jV_j\right)
=
(\bs^{\intercal}\Gamma \bs)I,
\end{equation}
so \(\Gamma\succeq0\).
Pointwise, $(c_{jk}(\q))_{j,k=1}^n$ is the Gram matrix of vectors in
\(\operatorname{Im}\H\coloneqq\operatorname{span}_{\R}\{\i,\j,\k\}\simeq\R^3\). Hence all of its \(4\times4\) minors vanish identically in \(\mathscr A_n\), and therefore $\operatorname{rank}\Gamma\leq3$.
Let \(r=\operatorname{rank}\Gamma\), and choose vectors
$\bd_1,\ldots,\bd_n\in\R^r$
such that
$\gamma_{jk}=\bd_j^{\intercal}\bd_k$.
Define the map $T$ by
\begin{equation}\label{eq:T-from-Gram}
T\left(\sum_js_j\bd_j\right)
\coloneqq
\sum_js_jV_j.
\end{equation}
This map is well defined. Indeed, if \(\sum_js_j\bd_j=0\), then
$\bs^{\intercal}\Gamma \bs
=\left\|\sum_js_j\bd_j\right\|^2
=0$,
and \eqref{eq:Gamma-positive} implies
$\sum_js_jV_j=0$.

Let \(\be_1,\ldots,\be_r\) be the standard orthonormal basis of \(\R^r\), and set $W_a\coloneqq T(\be_a)$ for $a=1,\ldots,r$.
Then
\begin{equation}\label{eq:V-in-W}
V_j=\sum_{a=1}^rd_{ja}W_a,
\end{equation}
and polarization of \eqref{eq:V-Clifford-Gram} gives
\begin{equation}\label{eq:W-Clifford-relations}
W_a^*=-W_a,
\qquad
W_aW_b+W_bW_a=-2\delta_{ab}I.
\end{equation}
At this stage, the $W_a$'s have not yet been identified with right-multiplication operators. Their commutation with the left quaternionic action follows algebraically: in \(\mathscr A_n\), $[v_j,L_c]=0$ for $c\in\H$, so $[V_j,\pi(L_c)]=0$. 
Since each \(W_a\) is a real linear combination of the \(V_j\)'s,
\begin{equation}\label{eq:W-commutes-left}
[W_a,\pi(L_c)]=0,
\qquad a=1,\ldots,r,\ c\in\H.
\end{equation}
We now invoke the Clifford-algebra classification. Relations
\eqref{eq:W-Clifford-relations}--\eqref{eq:W-commutes-left} show that the represented algebra is a quotient of
$\H\otimes_{\R}\mathrm{Cl}_{0,r}^{\mathrm{op}}$.
For \(r\leq3\),
\begin{equation}\label{eq:low-Clifford-table}
\H\otimes_{\R}\mathrm{Cl}_{0,r}^{\mathrm{op}}
\simeq
\begin{cases}
\H,&r=0,\\
M_2(\C),&r=1,\\
M_4(\R),&r=2,\\
M_4(\R)\oplus M_4(\R),&r=3.
\end{cases}
\end{equation}
Each irreducible real representation in \eqref{eq:low-Clifford-table} can be realized on the real Hilbert space $\H$ by left quaternionic multiplication together with right multiplication by $r$ orthonormal imaginary quaternions. For \(r=3\), the two irreducible summands correspond to the two possible orientations.

Choose orthonormal imaginary quaternions
\[
\varepsilon_1,\ldots,\varepsilon_r
\in\operatorname{Im}\H,
\qquad
\langle\varepsilon_a,\varepsilon_b\rangle=\delta_{ab},
\]
with the orientation corresponding to the irreducible component of \(\pi\) when
\(r=3\).  The right multiplication operators \(R_{\varepsilon_a}\) satisfy the
same Clifford relations as the \(W_a\)'s,
\[
R_{\varepsilon_a}^*=-R_{\varepsilon_a},
\qquad
R_{\varepsilon_a}R_{\varepsilon_b}
+R_{\varepsilon_b}R_{\varepsilon_a}
=-2\delta_{ab}I.
\]
Hence, by the classification in \eqref{eq:low-Clifford-table}, after an orthogonal identification \(\mathscr H_{\R}\simeq\H\), we may put the
irreducible representation into the standard form
\begin{equation}\label{eq:representation-standard-form}
\pi(L_c)=L_c,
\qquad
\pi(W_a)=R_{\varepsilon_a}.
\end{equation}
Now define
\[
\eta_j\coloneqq\sum_{a=1}^rd_{ja}\varepsilon_a
\in\operatorname{Im}\H,
\qquad
q_j^{(0)}\coloneqq\alpha_j+\eta_j.
\]
Since \(V_j=\pi(v_j)=\sum_{a=1}^rd_{ja}W_a\),
\eqref{eq:representation-standard-form} gives
$\pi(v_j)=\sum_{a=1}^rd_{ja}R_{\varepsilon_a}
=R_{\eta_j}$.
On the other hand, \eqref{eq:central-scalar-images} gives $\pi(x_j)=\alpha_j I=R_{\alpha_j}$.
Because \(r_j=x_j+v_j\), it follows that
\begin{equation}\label{eq:rj-evaluation-explicit}
\pi(r_j)
=\pi(x_j)+\pi(v_j)
=R_{\alpha_j}+R_{\eta_j}
=R_{q_j^{(0)}}.
\end{equation}
Taking adjoints also gives
$\pi(r_j^*)=R_{\overline{q_j^{(0)}}}$.
Thus \(\pi\) agrees with the scalar evaluation at
$\q^{(0)}=(q_1^{(0)},\ldots,q_n^{(0)})\in\H^n$
on the generators of \(\mathscr A_n\): it sends each constant left
multiplication \(L_c\) to \(L_c\), each \(r_j\) to \(R_{q_j^{(0)}}\), and each
\(r_j^*\) to \(R_{\overline{q_j^{(0)}}}\).  Hence \(\pi\) is unitarily equivalent to the evaluation
representation at \(\q^{(0)}\).
Because \(\pi\) annihilates \(\mathscr J_\bh^{\H}\),
$0=\pi(H_j)=R_{h_j(\q^{(0)})}$,
so $h_j(\q^{(0)})=0$.
Since \(G_i\in\mathscr M\) and \(\pi\) is \(\mathscr M\)-positive, $0\preceq\pi(G_i)=g_i(\q^{(0)})I$, and hence $g_i(\q^{(0)})\geq0$.
Thus $\q^{(0)}\in K_{\bg,\bh}$.
Conversely, evaluation at any point of $K_{\bg,\bh}$ factors through $\mathscr B$, is $\mathscr M$-positive, and is irreducible because the included left action of $\H$ on itself is irreducible as a real representation.

\medskip
\noindent{\bf Step 4: addressing the case \(K_{\bg,\bh}\neq\varnothing\).}
Evaluating a representation witnessing \eqref{eq:QH-archimedean} at a feasible point gives
$R-\sum_{i=1}^n|q_i|^2\geq0$. Hence $K_{\bg,\bh}$ is bounded.
Assume first that \(K_{\bg,\bh}\neq\varnothing\), and choose
\(\q^{(1)}\in K_{\bg,\bh}\). Evaluation at \(\q^{(1)}\) is an \(\mathscr M\)-positive representation. Therefore, $-I\notin\mathscr M$, because otherwise the evaluation would yield the impossible inequality $-I\succeq0$. Thus $\mathscr M$ is proper.
Set $\delta\coloneqq\min_{\q\in K_{\bg,\bh}}f(\q)>0$.
Take any $u\in\H$ with $|u|=1$. The real-coefficient assumptions imply
$g_i(u\q\bar u)=g_i(\q)$ and
$h_j(u\q\bar u)=u h_j(\q)\bar u$.
Hence \(K_{\bg,\bh}\) is invariant under simultaneous conjugation by unit quaternions. For \(\xi=|\xi|u\neq0\), Lemma~\ref{lem:operator-lift-equality-ideal} gives
\[
\langle\xi,\Phi(f)(\q)\xi\rangle_{\R}
=
|\xi|^2f(u\q\bar u)
\geq
\delta|\xi|^2,
\qquad
\forall\q\in K_{\bg,\bh}.
\]
By Step~3, every irreducible $\mathscr M$-positive representation is an evaluation at a feasible point. Therefore,
$\pi([\Phi(f)])\succeq\delta I$ for every irreducible $\mathscr M$-positive representation $\pi$. The proper Archimedean $*$-Positivstellensatz (Theorem~\ref{star-Pos}) yields $[\Phi(f)]\in\mathscr M$.

\medskip
\noindent{\bf Step 5: addressing the case \(K_{\bg,\bh}=\varnothing\).}
Suppose now that \(K_{\bg,\bh}=\varnothing\). 
The representation classification shows that there are no irreducible \(\mathscr M\)-positive representations.
If $-I\in\mathscr M$, there is nothing to prove. Otherwise, $\mathscr M$ is proper, and the Archimedean $*$-Positivstellensatz applies to $-I$; because there are no irreducible $\mathscr M$-positive representations, its strict-positivity hypothesis is vacuous. In either case,
\begin{equation}\label{eq:minus-I-in-M-empty}
-I\in\mathscr M.
\end{equation}
For $B=B^*\in\mathscr B$, the identity
\begin{equation}\label{eq:improper-module-all-symmetric}
4B
=(B+I)^*(B+I)
+(B-I)^*(-I)(B-I)
\end{equation}
has its right-hand side in $\mathscr M$: the first term is a Hermitian square, and the second belongs to $\mathscr M$ by \eqref{eq:minus-I-in-M-empty} and closure under conjugation. Since \(\mathscr M\) is a cone, division by the positive scalar \(4\) gives \(B\in\mathscr M\). Thus every self-adjoint element of \(\mathscr B\) belongs to \(\mathscr M\), and in particular
$[\Phi(f)]\in\mathscr M$.

\medskip
\noindent{\bf Step 6: recovering the explicit QSOS certificate.}
Since $[\Phi(f)]\in\mathscr M$, there exist finitely many elements
\(A_{0s},A_{is}\in\mathscr A_n\) such that
\begin{equation}\label{eq:operator-certificate-before-pullback}
D
\coloneqq
\Phi(f)
-
\sum_sA_{0s}^*A_{0s}
-
\sum_{i=1}^m\sum_sA_{is}^*G_iA_{is}
\in\mathscr J_\bh^{\H}.
\end{equation}
Every element of $\mathscr A_n$ has the form $\Theta(p)$ for some left-coefficient polynomial $p\in\mathcal L$. Choose
$A_{is}=\Theta(p_{is})$
and set
$\sigma_i\coloneqq\sum_sp_{is}^*p_{is}
\in\Sigma^{\H}$.
By \eqref{eq:Phi-square} and \eqref{eq:Phi-weighted},
\begin{equation}\label{eq:D-Phi-before-explicit-equalities}
D
=
\Phi\left(
f-
\sigma_0-
\frac12\sum_{i=1}^m(\sigma_i g_i+g_i\sigma_i)
\right).
\end{equation}
The operator $D$ is self-adjoint. The final assertion of Lemma~\ref{lem:operator-lift-equality-ideal} yields finitely many
\(\tau_{j\alpha},\nu_{j\alpha}\in\mathcal L\) such that
\begin{equation}\label{eq:explicit-equality-pullback-final}
D=\Phi(E),
\qquad
E
=
\sum_{j=1}^{\ell}\sum_{\alpha}
\left(
\tau_{j\alpha}h_j\nu_{j\alpha}^*
+\nu_{j\alpha}h_j^*\tau_{j\alpha}^*
\right)
\in\sym\P\langle\q,\bar\q\rangle.
\end{equation}
Combining \eqref{eq:D-Phi-before-explicit-equalities} and
\eqref{eq:explicit-equality-pullback-final} gives
\[
\Phi\left(
f-
\sigma_0-
\frac12\sum_{i=1}^m(\sigma_i g_i+g_i\sigma_i)
-E
\right)=0.
\]
The polynomial inside the parentheses belongs to
\(\sym\P\langle\q,\bar\q\rangle\). Hence, by \eqref{eq:Phi-kernel},
\begin{equation}\label{eq:final-scalar-identity-remainder}
f-
\sigma_0-
\frac12\sum_{i=1}^m(\sigma_i g_i+g_i\sigma_i)
-E
\in
\mathcal I_n^{\H}\cap\sym\P\langle\q,\bar\q\rangle.
\end{equation}
Therefore,
$f\in\mathcal Q^{\H}(\bg,\bh)+
\left(
\mathcal I_n^{\H}\cap\sym\P\langle\q,\bar\q\rangle
\right)$.
\end{proof}

The certificate is written with symmetrized weighted terms to ensure that it remains in $\sym\P\langle\q,\bar\q\rangle$.  Modulo scalar identities, however, it agrees with the more familiar weighted-sandwich form, as the following remark records.
\begin{remark}
For $p\in\mathcal L$ and $g\in\sym\R\langle\q,\bar\q\rangle$, the expressions
\[
p^*gp
\qquad\text{and}\qquad
\frac12\bigl((p^*p)g+g(p^*p)\bigr)
\]
have the same scalar quaternionic evaluations. Hence
\[
p^*gp-
\frac12\bigl((p^*p)g+g(p^*p)\bigr)
\in \mathcal I_{n}^\H.
\]
Thus, modulo the scalar-identity ideal, the certificate in Theorem~\ref{thm:putinar-quaternionic} can equivalently be viewed as a standard weighted sandwich certificate. The symmetrized formulation is used because it remains in
\(\sym\P\langle\q,\bar\q\rangle\).
\end{remark}

Theorem~\ref{thm:putinar-quaternionic} leaves the scalar-identity remainder in the abstract space $\mathcal I_n^\H\cap\sym\P\langle\q,\bar\q\rangle$.  For later formulations and finite algebraic tests, it is useful to have generators that already lie inside this symmetric intersection.  The next proposition supplies such an internal spanning family.

\begin{proposition}
\label{prop:Isa-internal-generators}
For $w\in\cW$ and $i=1,\ldots,n$, let
\begin{equation}\label{eq:Isa-real-basic-generator}
s_{w,i}\coloneqq[w+w^*,q_i].
\end{equation}
For $a,b,u,v,w\in\cW$, $e\in\{1,\i,\j,\k\}$,
$i\in\{1,\ldots,n\}$, and $\varepsilon\in\{1,*\}$, define
\begin{align}
\kappa_{a,e,b}
&\coloneqq
\bigl(aeb+(aeb)^*\bigr)
-
\bigl(eba+(eba)^*\bigr),
\label{eq:Isa-kappa-generator}\\
\lambda_{e,u,v;w,i}^{\varepsilon}
&\coloneqq
 eus_{w,i}^{\varepsilon}v
+
\bigl(eus_{w,i}^{\varepsilon}v\bigr)^*.
\label{eq:Isa-lambda-generator}
\end{align}
Then every polynomial defined in \eqref{eq:Isa-kappa-generator}--\eqref{eq:Isa-lambda-generator} belongs to $\mathcal I_n^{\H}\cap
\sym\P\langle\q,\bar\q\rangle$, and
\begin{equation}\label{eq:Isa-internal-span}
\boxed{
\mathcal I_n^{\H}\cap
\sym\P\langle\q,\bar\q\rangle
=
\operatorname{span}_{\R}
\left\{
\kappa_{a,e,b},\,
\lambda_{e,u,v;w,i}^{\varepsilon}:
\begin{array}{l}
a,b,u,v,w\in\cW,\ e\in\{1,\i,\j,\k\},\\
i=1,\ldots,n,\ \varepsilon\in\{1,*\}
\end{array}
\right\}.}
\end{equation}
\end{proposition}

\begin{proof}
We first verify that the displayed generators belong to
$\mathcal I_n^{\H}\cap\sym\P\langle\q,\bar\q\rangle$. They are manifestly symmetric elements of $\P\langle\q,\bar\q\rangle$.  For every scalar quaternionic tuple
$\ba\in\H^n$, cyclic invariance of the real part gives
$\RR\!\left(a(\ba)eb(\ba)\right)=\RR\!\left(eb(\ba)a(\ba)\right)$,
and hence
$\kappa_{a,e,b}(\ba,\bar\ba)=0$.
Moreover, by the preceding description of the real scalar-identity ideal,
$s_{w,i}\in\mathcal I_n^{\R}\subseteq\mathcal I_n^{\H}$.  Therefore
$s_{w,i}^{\varepsilon}$ vanishes under every scalar quaternionic evaluation,
and so does
$\lambda_{e,u,v;w,i}^{\varepsilon}$.  Hence every generator on the
right-hand side of \eqref{eq:Isa-internal-span} lies in $\mathcal I_n^{\H}\cap
\sym\P\langle\q,\bar\q\rangle$.
This proves the inclusion ``$\supseteq$'' in \eqref{eq:Isa-internal-span}.

For the reverse inclusion, let $p\in \mathcal I_n^{\H}\cap
\sym\P\langle\q,\bar\q\rangle$ and write
\begin{equation}\label{eq:Isa-F-presentation}
p=\sum_{\alpha=1}^N a_\alpha c_\alpha b_\alpha,
\qquad
a_\alpha,b_\alpha\in\cW,
\quad
c_\alpha\in\H.
\end{equation}
Since $p=p^*$, $p=
\frac12\sum_{\alpha=1}^N
\left(
 a_\alpha c_\alpha b_\alpha
 +(a_\alpha c_\alpha b_\alpha)^*
\right)$.
Write each $c_\alpha$ as a real linear combination of $1,\i,\j,\k$. Applying \eqref{eq:Isa-kappa-generator} term by term allows us to move the unique quaternionic coefficient cyclically to the left.  Consequently,
there exists a finite real linear combination $\omega$ of the
$\kappa$-generators and a left-coefficient polynomial $\hat p\in\mathcal L$ such
that
\begin{equation}\label{eq:Isa-reduce-left-coefficient}
p=\omega+\frac12(\hat p+\hat p^*).
\end{equation}
More explicitly, writing
$c_\alpha=\sum_{e\in\{1,\i,\j,\k\}}\gamma_{\alpha,e}e,\gamma_{\alpha,e}\in\R$,
one may take
$$\hat p=\sum_{\alpha=1}^N\sum_{e\in\{1,\i,\j,\k\}}
\gamma_{\alpha,e}eb_\alpha a_\alpha.$$
Since both $p$ and every $\kappa$-generator are scalar identities,
\eqref{eq:Isa-reduce-left-coefficient} implies
\begin{equation}\label{eq:Isa-Re-p-zero}
\RR\bigl(\hat p(\ba,\bar\ba)\bigr)=0,
\qquad
\forall\ba\in\H^n.
\end{equation}
Decompose the quaternionic coefficients of $\hat p$ in the fixed basis:
\begin{equation}\label{eq:Isa-p-component-decomposition}
\hat p=\hat p_0+\i\hat p_1+\j\hat p_2+\k\hat p_3,
\qquad
\hat p_0,\hat p_1,\hat p_2,\hat p_3\in\R\langle\q,\bar\q\rangle.
\end{equation}
Fix $\ba\in\H^n$ and write
\[
\hat p_s(\ba,\bar\ba)=\alpha_s+v_s,
\qquad
\alpha_s\in\R,
\quad v_s\in\operatorname{Im}\H\simeq\R^3,
\quad
s=0,1,2,3.
\]
For a unit quaternion $u$, real coefficients give
$\hat p_s(u\ba\bar u, u\bar\ba\bar u)
=u\hat p_s(\ba,\bar\ba)\bar u$.
Let $O_u\in\mathrm{SO}(3)$ be the rotation induced by inner conjugation on
$\operatorname{Im}\H$, and let
$e_1,e_2,e_3$ denote the standard vectors
corresponding to $\i,\j,\k$.  Applying
\eqref{eq:Isa-Re-p-zero} to $u\ba\bar u$ yields
\begin{equation}\label{eq:Isa-rotation-identity}
\alpha_0-\sum_{s=1}^3
e_s^{\intercal}(O_u v_s)=0.
\end{equation}
As $u$ varies, $O_u$ ranges over $\mathrm{SO}(3)$. Averaging
\eqref{eq:Isa-rotation-identity} over $\mathrm{SO}(3)$ gives
$\alpha_0=0$.  The remaining linear functional of $O$ therefore vanishes for
every $O\in\mathrm{SO}(3)$.  Since the real linear span of
$\mathrm{SO}(3)$ is all of $M_3(\R)$, it follows that
\begin{equation}\label{eq:Isa-component-conclusions}
v_1=v_2=v_3=0.
\end{equation}
For completeness, this spanning assertion follows from the rotations by $\pi$ about the coordinate axes, the rotations by $\pi$ about the axes $(e_r+e_s)/\sqrt2$, and the rotations by $\pm\pi/2$ in the coordinate planes: their real span contains every matrix unit $E_{rs}$.

Since $\ba$ was arbitrary, \eqref{eq:Isa-component-conclusions} and
$\alpha_0=0$ imply
\begin{align}
\hat p_0+\hat p_0^*&\in\mathcal I_n^{\R},
\label{eq:Isa-p0-real-identity}\\
\hat p_s-\hat p_s^*&\in\mathcal I_n^{\R},
\qquad s=1,2,3.
\label{eq:Isa-ps-real-identity}
\end{align}
For $s=1,2,3$, put
\[
a_s\coloneqq\frac{\hat p_s+\hat p_s^*}{2},
\qquad
d_s\coloneqq\frac{\hat p_s-\hat p_s^*}{2}.
\]
Then $a_s=a_s^*$ and, by \eqref{eq:Isa-ps-real-identity},
$d_s\in\mathcal I_n^{\R}$ with $d_s^*=-d_s$.  Writing
$e_1=\i,e_2=\j,e_3=\k$, a direct expansion of
\eqref{eq:Isa-p-component-decomposition} gives
\begin{equation}\label{eq:Isa-p-plus-p-star-decomposition}
\hat p+\hat p^*
=
(\hat p_0+\hat p_0^*)
+
\sum_{s=1}^3
\left(
 e_sa_s-a_se_s
 +e_sd_s+d_se_s
\right).
\end{equation}

We show that every term on the right-hand side belongs to the span in \eqref{eq:Isa-internal-span}. By Lemma~\ref{real-scalar-identity},
$\mathcal I_n^{\R}=
\left\langle s_{w,i}:w\in\cW,\ i=1,\ldots,n\right\rangle_*
$ as a two-sided $*$-ideal.  Hence every element of $\mathcal I_n^{\R}$ is a
finite real linear combination of terms
$u s_{w,i}^{\varepsilon}v$.  Since
$\hat p_0+\hat p_0^*$ is symmetric, taking Hermitian parts of such a representation
shows that
$\hat p_0+\hat p_0^*\in
\operatorname{span}_{\R}
\left\{
\lambda_{1,u,v;w,i}^{\varepsilon}
\right\}$.
Likewise, $d_s\in\mathcal I_n^{\R}$ and $d_s^*=-d_s$, so $e_sd_s+d_se_s=
e_sd_s+(e_sd_s)^*$.
Expanding $d_s$ into two-sided multiples of the generators
$s_{w,i}^{\varepsilon}$ therefore gives
\begin{equation}\label{eq:Isa-ds-lambda-span}
e_sd_s+d_se_s
\in
\operatorname{span}_{\R}
\left\{
\lambda_{e_s,u,v;w,i}^{\varepsilon}
\right\}.
\end{equation}
Finally, every symmetric real polynomial $a_s$ is a finite real linear combination of polynomials of the form $w+w^*$.  For such a polynomial one has
\begin{equation}\label{eq:Isa-commutator-kappa}
e_s(w+w^*)-(w+w^*)e_s
=
-\kappa_{w,e_s,1}.
\end{equation}
Thus
\begin{equation}\label{eq:Isa-as-kappa-span}
e_sa_s-a_se_s
\in
\operatorname{span}_{\R}\{\kappa_{a,e,b}\}.
\end{equation}
Equations
\eqref{eq:Isa-p-plus-p-star-decomposition}--\eqref{eq:Isa-as-kappa-span}
show that $\hat p+\hat p^*$ belongs to the span on the right-hand side of
\eqref{eq:Isa-internal-span}.  Together with
\eqref{eq:Isa-reduce-left-coefficient}, the same is true of $p$.
This proves the reverse inclusion and completes the proof.
\end{proof}

\subsection{A counterexample to a more general Archimedean Positivstellensatz}\label{sec:counterexample-general-putinar}

The preceding theorem gives a positive result for quaternionic-coefficient objectives under real-coefficient constraints.  We now delimit the scope of that result by asking whether the same conclusion remains valid when the constraint polynomials themselves have quaternionic coefficients.  The answer is negative: the real-coefficient assumption on the constraints cannot generally be removed, even after quotienting by the full scalar-identity ideal and imposing a sphere constraint.  The following univariate counterexample exhibits this obstruction, and all of its polynomials admit quadratic Gram representations.

Let $q\in\H$ and write
$q=x_0+x_1\i+x_2\j+x_3\k$.
Set $h\coloneqq 1-|q|^2$, and
$g_1\coloneqq \frac{1}{10}+x_0$,
$g_2\coloneqq \frac{1}{10}-x_0$,
$g_3\coloneqq \frac12+x_1$,
$g_4\coloneqq \frac12+x_2$,
$g_5\coloneqq \frac12+x_3$.
Define
\begin{equation*}
 K\coloneqq
 \left\{
 q\in\H:h(q,\bar q)=0,\,g_i(q,\bar q)\ge0,\,i=1,\ldots,5\right\}.
\end{equation*}
Let $f\coloneqq x_1+x_2+x_3+\frac25$.
We show that $\min_{q\in K}f(q)=\frac{1}{10}>0$, whereas $f\notin\mathcal Q^{\H}(\bg,h)+
\left(
\mathcal I_1^{\H}\cap\sym\P\langle q,\bar q\rangle
\right)$.

For $q\in K$, put $r^2\coloneqq x_1^2+x_2^2+x_3^2$.
Since $|q|^2=1$ and $|x_0|\leq \frac{1}{10}$, $r^2=1-x_0^2\geq\frac{99}{100}$.
Set
\[
 y_s\coloneqq x_s+\frac12\geq0,
 \qquad s=1,2,3,
 \qquad
 T\coloneqq y_1+y_2+y_3.
\]
Then
\[
 r^2
 =\sum_{s=1}^3\left(y_s-\frac12\right)^2
 =\sum_{s=1}^3y_s^2-T+\frac34.
\]
This yields $\sum_{s=1}^3y_s^2\geq T+\frac{6}{25}$.
Since the $y_s$'s are nonnegative,
$\sum_s y_s^2\leq T^2$, and hence
$T^2-T-\frac{6}{25}\geq0$.
Therefore $T\geq6/5$ and
$x_1+x_2+x_3=T-\frac32\geq-\frac3{10}$.
Equality is attained at
$\hat q\coloneqq
\frac{1}{10}+\frac{7}{10}\i-\frac12\j-\frac12\k$.
Thus $\min_{q\in K}(x_1+x_2+x_3)=-\frac3{10}$, and consequently
$\min_{q\in K}f(q)
 =-\frac3{10}+\frac25
 =\frac1{10}>0$.

Put
\[
 u\coloneqq\frac{\i+\j+\k}{\sqrt3},
 \qquad
 \alpha\coloneqq\frac{\sqrt3-1}{12},
 \qquad
 \beta\coloneqq\frac{3-\sqrt3}{12}.
\]
Define the positive probability measure
\begin{equation}\label{eq:counterexample-separating-measure}
 \nu
 \coloneqq
 \frac12\delta_{-u}
 +\alpha\left(\delta_{\i}+\delta_{\j}+\delta_{\k}\right)
 +\beta\left(\delta_{-\i}+\delta_{-\j}+\delta_{-\k}\right).
\end{equation}
Indeed, $\alpha,\beta>0$ and $\frac12+3(\alpha+\beta)=1$. Let
$L(p)\coloneqq\int_{\H}p(q,\bar q)\,d\nu(q)$,
where the quaternion-valued integral is interpreted componentwise.  Since $L$ is constructed from scalar quaternion evaluations, $L(p)=0$ for $p\in\mathcal I_1^{\H}$.
Every point in $\supp(\nu)$ is a purely imaginary unit quaternion, so
\[
 \bar q=-q,\qquad q^2=-1,\qquad \bar q q=1,\qquad x_0=0
 \quad\text{on }\supp(\nu).
\]
It follows that $L(p)=0$ for $p\in\mathcal J_h^\H$.
A direct calculation gives, for $i,j\in\{1,2,3\}$,
\begin{equation}\label{eq:counterexample-first-second-moments}
 L(x_i)=-\frac13,
 \qquad
 L(x_i^2)=\frac13,
 \qquad
 L(x_ix_j)=\frac16\quad(i\neq j).
\end{equation}

Let $p\in\mathcal L$ be arbitrary. On $\supp(\nu)$, the relations $\bar q=-q$ and $q^2=-1$ reduce
every word in $q,\bar q$ to one of $\pm1,\pm q$.  Hence there exist $a,b\in\H$, independent of the support point, such that
$p(q,\bar q)=a+bq$ for every $q\in\supp(\nu)$.
Consequently,
\begin{equation}\label{eq:counterexample-L-square-positive}
 L(p^*p)=\int_{\H}|p(q,\bar q)|^2\,d\nu(q)\geq0.
\end{equation}
Because $x_0=0$ on $\supp(\nu)$,
\begin{equation}\label{eq:counterexample-g12-positive} L(p^*g_1p)=L(p^*g_2p)=\frac1{10}L(p^*p)\geq0.
\end{equation}

Write $e_1\coloneqq\i,e_2\coloneqq\j,e_3\coloneqq\k$. We have $L(q)=-\frac13(\i+\j+\k)$
and $L(x_iq)=\frac13e_i+\frac16\sum_{j\neq i}e_j$ for $i=1,2,3$.
Therefore, for $i=1,2,3$, $n_i\coloneqq L(g_{i+2}q)=\frac16e_i$. Moreover, let $\mu_i\coloneqq L(g_{i+2})=\frac16$.
Using the fact that $g_{i+2}$ is real-valued, we obtain
\begin{align*}
 L(p^*g_{i+2}p)
 &=\mu_i(|a|^2+|b|^2)
   +\bar a bn_i+\bar n_i\bar b a\\
 &=\mu_i(|a|^2+|b|^2)
   +2\RR(\bar a bn_i)\\
 &\geq
 \mu_i(|a|^2+|b|^2)-2|a||b||n_i|\\
 &=\frac16(|a|-|b|)^2\geq0.
\end{align*}
Together with \eqref{eq:counterexample-g12-positive}, this proves
$L(p^*g_ip)\geq0$ for every $p\in\mathcal L$ and $i=1,\ldots,5$. By linearity and the real-valuedness of the constraints, it follows that $L(p)\geq0$ for every $p\in\mathcal Q^{\H}(\bg,h)$.
By \eqref{eq:counterexample-first-second-moments}, $L(f)
 =L(x_1+x_2+x_3)+\frac25
 =-1+\frac25
 =-\frac35<0$.
Since $L$ vanishes on $\mathcal I_1^{\H}$, this is impossible if
$f\in\mathcal Q^{\H}(\bg,h)+\bigl(\mathcal I_1^{\H}\cap\sym\P\langle q,\bar q\rangle\bigr)$.

It remains to verify the asserted quadratic Gram representations.  If
$a_0,a_1,a_2,a_3,b,c\in\R$ and
\[
 \mathbf w\coloneqq\begin{bmatrix}1\\q\end{bmatrix},
 \qquad
 Q\coloneqq
 \begin{bmatrix}
 c & \frac12(a_0-a_1\i-a_2\j-a_3\k)\\[1mm]
 \frac12(a_0+a_1\i+a_2\j+a_3\k) & b
 \end{bmatrix},
\]
then $Q=Q^*$ and, modulo $\mathcal I_1^{\H}$, $\mathbf w^*Q\mathbf w
=c+a_0x_0+a_1x_1+a_2x_2+a_3x_3+b\bar q q$.
Hence $f,h,g_1,\ldots,g_5$ all have the asserted quadratic Gram
representations.

In summary, the scalar-identity quotient supports Archimedean QSOS certificates for real-coefficient data and, via the operator lift, for quaternionic-coefficient objectives with real-coefficient constraints; the counterexample shows that this latter constraint hypothesis is sharp for the present certificate cone.  These positivity results provide the dual algebraic foundation for the quaternionic moment problems studied next.

\section{The quaternionic moment problem}
The Positivstellens\"atze of Section~\ref{sec:qsos} provide algebraic certificates for positivity of quaternionic polynomials.  We now develop the corresponding quaternionic moment theory, which characterizes when linear functionals on quaternionic polynomials arise from positive Borel measures on scalar quaternionic tuples.  We proceed in three stages. First, we treat full real-valued moment sequences and derive unconstrained and constrained Riesz--Haviland-type characterizations. Second, we consider quaternion-valued moment sequences and obtain the analogous full moment criteria. Finally, we turn to truncated quaternion-valued data, for which flatness and scalar consistency lead to finitely atomic representing measures. Throughout the section, the scalar-identity ideals introduced earlier provide the link between formal word moments and genuine scalar quaternionic evaluations.

\subsection{The full real-valued quaternionic moment problem}\label{sec:full-real-valued-quaternionic-moment}
We begin with real-valued moment sequences.  The main objective of this subsection is to characterize representability first on the whole space $\H^n$ and then on a quaternionic basic semialgebraic set $K_{\bg,\bh}$.  Under the Archimedean hypothesis, these abstract positivity conditions will reduce to positivity of moment and localizing matrices together with the equality and scalar-identity relations.

Recall
\[
\mathcal F_n\coloneqq\R\langle\q,\bar\q\rangle,
\qquad
\pi:\mathcal F_n\longrightarrow
\mathcal A\coloneqq\mathcal F_n/\mathcal I_{n}^\R,
\qquad
\mathcal B\coloneqq\sym\mathcal A.
\]
Two quaternionic monomials $u$ and $v$ are \emph{cyclically equivalent}, written $u\stackrel{\mathrm{cyc}}{\sim}v$, if $u=w_1w_2$ and $v=w_2w_1$ for some quaternionic monomials $w_1,w_2$.
Let $\y=(y_w)_{w\in\cW}\in\R^{|\cW|}$ be a real-valued sequence. We call $\y$ a \emph{quaternionic pseudo-moment sequence} if
\begin{enumerate}[(i)]
    \item $y_w=y_{w^*}$ for $w\in\cW$;
    \item $y_u=y_{v}$ if $u\stackrel{\mathrm{cyc}}{\sim}v$ for $u,v\in\cW$.
\end{enumerate}
A quaternionic pseudo-moment sequence $\y$ defines a real-linear functional
$L_{\y}:\mathcal F_n\to\R$ by
\begin{equation}
p=\sum_{w}p_{w}w\longmapsto L_{\y}(p)=\sum_{w}p_{w}y_{w}.
\end{equation}
A positive Borel measure $\mu$ on $\H^n$ represents $\y$ if
\[
y_w=
\int_{\H^n}\RR\bigl(w(\ba,\overline\ba)\bigr)\,d\mu(\ba)
\qquad(w\in\mathcal W),
\]
where the quaternion-valued integral is interpreted componentwise.
Let
\[
\mathcal P_n^{\R}
\coloneqq
\left\{
 p\in\sym\mathcal F_n:
 p(\ba,\bar\ba)\geq0\ \text{for every }\ba\in\H^n
\right\}.
\]

With this notation in place, we first characterize the unconstrained full moment problem.  The result is the quaternionic analogue of the classical Riesz--Haviland theorem for the real-valued functional $L_{\y}$.
\begin{theorem}\label{thm:real-quaternionic-riesz-haviland}
Let \(\y=(y_w)_{w\in\cW}\in\R^{|\cW|}\) be a quaternionic pseudo-moment sequence. The following are equivalent.
\begin{enumerate}[(i)]
\item There exists a positive Borel representing measure \(\mu\) for \(\y\).
\item For every $p\in\mathcal P_n^{\R}$, $L_{\y}(p)\geq0$.
\item $L_{\y}(\mathcal I_{n}^\R)=0$ and
the induced functional $\widetilde{L}_{\y}:\mathcal B\longrightarrow\R$
is nonnegative on every \(b\in\mathcal B\) whose scalar evaluation is
nonnegative on \(\H^n\).
\end{enumerate}
Moreover, whenever a representing measure exists, one may choose it invariant
under simultaneous quaternionic conjugation
$(q_1,\ldots,q_n)\longmapsto
(uq_1\bar u,\ldots,uq_n\bar u)$ for any $u\in\H$ with $|u|=1$.
\end{theorem}

\begin{proof}
The implication (i)$\Rightarrow$(ii) is immediate, because the integral of a nonnegative symmetric polynomial is nonnegative.

We next prove (ii)$\Rightarrow$(iii). Since $p^*p$ is nonnegative at every quaternionic point,
$L_{\y}(p^*p)\geq0$ for every $p\in\mathcal F_n$. Define
\[
\langle p,r\rangle_{\y}\coloneqq L_{\y}(r^*p),
\qquad p,r\in\mathcal F_n.
\]
This is a real PSD bilinear form. Moreover,
$L_{\y}(r^*p)=L_{\y}((r^*p)^*)=L_{\y}(p^*r)$, so the form is symmetric and the Cauchy--Schwarz inequality applies.
Let \(r\in \mathcal I_{n}^\R\). Then \(r^*r\) is identically zero on \(\H^n\), so both
\(r^*r\) and \(-r^*r\) belong to \(\mathcal P_n^{\R}\). Hence
$L_{\y}(r^*r)=0$.
The Cauchy--Schwarz inequality, applied to \(1\) and \(r\), gives
\[
|L_{\y}(r^*)|^2
=|\langle1,r\rangle_{\y}|^2
\leq
L_{\y}(1)L_{\y}(r^*r)
=0.
\]
Thus $L_{\y}(r^*)=L_{\y}(r)=0$, which proves
$L_{\y}(\mathcal I_{n}^\R)=0$.

If \(b\in\mathcal B\) and \(p\in\mathcal F_n\) represents \(b\), define
\begin{equation}\label{eq:ell-y-definition}
\widetilde{L}_{\y}(b)
\coloneqq
L_{\y}\!\left(\frac{p+p^*}{2}\right).
\end{equation}
This definition is independent of the chosen representative. Indeed, if \(p'\) also represents
\(b\), then \(p-p'\in \mathcal I_{n}^\R\), and therefore
$(p+p^*)/2-(p'+(p')^*)/2
\in \mathcal I_{n}^\R$.
Since \(L_{\y}(\mathcal I_{n}^\R)=0\), the two values in
\eqref{eq:ell-y-definition} agree. If \(b\) is nonnegative on \(\H^n\), its
symmetric representative \((p+p^*)/2\) belongs to
\(\mathcal P_n^{\R}\), so \(\widetilde{L}_{\y}(b)\geq0\). This proves (iii).

The implication (iii)$\Rightarrow$(ii) follows directly from the definition of $\widetilde{L}_{\y}$. It remains to prove (iii)$\Rightarrow$(i).

By Lemma~\ref{lem:symmetric-central-quotient}, $\mathcal B$ is a central commutative real algebra.
In \(\mathcal A\), suppressing quotient brackets, put
\[
x_i\coloneqq\frac{q_i+\bar q_i}{2},
\qquad
v_i\coloneqq\frac{q_i-\bar q_i}{2}.
\]
Define the symmetric elements
\begin{equation}\label{eq:moment-invariant-generators}
s_{ij}
\coloneqq
-\frac12(v_iv_j+v_jv_i),
\qquad
t_{ijk}
\coloneqq
\frac12(v_kv_jv_i-v_iv_jv_k).
\end{equation}
Write each component of a scalar quaternionic point as
\[
a_i=\xi_i+\mathbf v_i,
\qquad
\xi_i\in\R,
\quad
\mathbf v_i\in\operatorname{Im}\H\simeq\R^3.
\]
The quaternion multiplication rules give
\begin{equation}\label{eq:moment-invariant-evaluations}
x_i(\ba)=\xi_i,
\qquad
s_{ij}(\ba)=\mathbf v_i\mathbin{\boldsymbol\cdot}\mathbf v_j,
\qquad
t_{ijk}(\ba)=\det(\mathbf v_i,\mathbf v_j,\mathbf v_k).
\end{equation}

Let
\[
V\coloneqq\R^n\oplus(\R^3)^n\simeq\H^n,
\qquad
G\coloneqq\mathrm{SO}(3),
\]
where \(G\) fixes the scalar coordinates and acts diagonally on the imaginary
vectors. Define
\[
\Theta_{\mathrm{inv}}:\mathcal B\longrightarrow\R[V]^G,
\qquad
\Theta_{\mathrm{inv}}(b)(\ba)\coloneqq b(\ba,\bar\ba).
\]
This is a well-defined unital real-algebra homomorphism. It is injective by the definition of $\mathcal I_n^\R$. Its image is contained in $\R[V]^G$: if \(|u|=1\),
then every real-coefficient polynomial is equivariant,
$p(u\ba\bar u,u\bar\ba\bar u)
=
up(\ba,\bar\ba)\bar u$,
and a scalar-valued equivariant polynomial is invariant. Conversely, the first fundamental theorem for $\mathrm{SO}(3)$ states that $\R[V]^G$ is generated by
the scalar coordinates \(\xi_i\), the inner products
\(\mathbf v_i\boldsymbol\cdot\mathbf v_j\), and the oriented triple products
\(\det(\mathbf v_i,\mathbf v_j,\mathbf v_k)\); see, e.g.,
\cite{weyl1946classical}. By \eqref{eq:moment-invariant-evaluations}, all these
generators belong to the image of \(\Theta_{\mathrm{inv}}\). Therefore,
\begin{equation}\label{eq:symmetric-invariant-isomorphism}
\Theta_{\mathrm{inv}}:\mathcal B\xrightarrow{\ \simeq\ }\R[V]^{\mathrm{SO}(3)}.
\end{equation}
Equivalently, $\mathcal B=
\R[x_i,s_{ij},t_{ijk}:1\leq i,j,k\leq n]$.

Let \(dg\) denote the normalized Haar measure on \(G\), and let
\[
\mathscr R_G:\R[V]\longrightarrow\R[V]^G,
\qquad
(\mathscr R_GP)(\mathbf z)
\coloneqq
\int_G P(g\cdot \mathbf z)\,dg
\]
be the Reynolds operator. Define
\begin{equation}\label{eq:Reynolds-extended-functional}
\widehat{L}_{\y}(P)
\coloneqq
\widetilde{L}_{\y}\!\left(\Theta_{\mathrm{inv}}^{-1}(\mathscr R_GP)\right),
\qquad P\in\R[V].
\end{equation}
If \(P\geq0\) on \(V\), then \(\mathscr R_GP\geq0\) on \(V\), so condition~(iii)
gives \(\widehat{L}_{\y}(P)\geq0\). The classical Riesz--Haviland theorem~\cite[Theorem~3.1.2]{marshall2008positive} therefore yields a positive Borel measure $\mu$ on $V\simeq\H^n$ such that
\begin{equation}\label{eq:classical-measure-for-Reynolds-extension}
\widehat{L}_{\y}(P)=\int_VP(\mathbf z)\,d\mu(\mathbf z)
\qquad(P\in\R[V]).
\end{equation}

For \(p\in\mathcal F_n\), let
$P_p(\ba)
\coloneqq
\RR\bigl(p(\ba,\bar\ba)\bigr)$.
This polynomial is \(G\)-invariant and, under
\eqref{eq:symmetric-invariant-isomorphism}, corresponds to
$\pi\!\left((p+p^*)/2\right)\in\mathcal B$.
Hence \(\mathscr R_GP_p=P_p\), and
\begin{equation*}
L_{\y}(p)=L_{\y}\!\left(\frac{p+p^*}{2}\right)=\widetilde{L}_{\y}\!\left(\pi\!\left(\frac{p+p^*}{2}\right)\right)=\widehat{L}_{\y}(P_p)=\int_{\H^n}\RR\bigl(p(\ba,\bar\ba)\bigr)\,d\mu(\ba).
\end{equation*}
Thus $\mu$ represents $\y$.

Finally, define the orbit average of $\mu$ by
\[
\mu_G(E)\coloneqq\int_G\mu(g^{-1}E)\,dg.
\]
Every function $P_p$ is $G$-invariant, so $\mu_G$ has the same quaternionic moments as $\mu$. This proves the final assertion.
\end{proof}

\begin{remark}\label{rem:full-real-cyclicity-automatic}
Before adding polynomial constraints, it is useful to record that the scalar-identity condition in the preceding theorem already contains the cyclic moment relations. Thus, cyclicity is compatible with, and in this sense subordinate to, scalar consistency. Indeed, let \(u,v\in\cW\)
and put \(c=uv-vu\). At every quaternionic point,
\[
\frac{c+c^*}{2}(\ba)
=
\RR\bigl(u(\ba)v(\ba)-v(\ba)u(\ba)\bigr)
=0.
\]
Thus \((c+c^*)/2\in \mathcal I_{n}^\R\). By the scalar-identity condition,
$L_{\y}(uv-vu)=
L_{\y}\!\left((c+c^*)/2\right)=0$.
\end{remark}

We now pass from the unconstrained problem to measures supported on a prescribed feasible set.  Let $g_1,\ldots,g_m\in\sym\mathcal F_n$ and $h_1,\ldots,h_\ell\in\mathcal F_n$.  The feasible set $K_{\bg,\bh}$ is closed and invariant under simultaneous quaternionic conjugation.  The same Reynolds-operator argument used above then yields the constrained Riesz--Haviland characterization.

\begin{theorem}
\label{thm:constrained-real-quaternionic-riesz-haviland}
A real-valued quaternionic pseudo-moment sequence \(\y\) has a positive Borel representing measure supported on
\(K_{\bg,\bh}\) if and only if
\begin{equation}\label{eq:constrained-full-haviland-positivity}
L_{\y}(f)\geq0
\end{equation}
for every \(f\in\sym\mathcal F_n\) satisfying $f\geq0$ on $K_{\bg,\bh}$.
Under these conditions,
\begin{equation}\label{eq:constrained-full-ideal-annihilation}
L_{\y}\!\left(
\mathcal I_{n}^\R+\mathcal J_{\bh}^{\R}
\right)=0.
\end{equation}
\end{theorem}

\begin{proof}
Necessity follows immediately by integration. Conversely, assume \eqref{eq:constrained-full-haviland-positivity}. The bilinear form
\(\langle p,r\rangle_{\y}=L_{\y}(r^*p)\) is again symmetric and PSD, because \(p^*p\geq0\) on \(K_{\bg,\bh}\). If
$r\in \mathcal I_{n}^\R+\mathcal J_{\bh}^{\R}$,
then \(r\) vanishes on \(K_{\bg,\bh}\). Hence \(r^*r\) and \(-r^*r\) are both
nonnegative there, so \(L_{\y}(r^*r)=0\). The Cauchy--Schwarz inequality gives
$|L_{\y}(r^*)|^2\leq
L_{\y}(1)L_{\y}(r^*r)
=0$,
and adjoint symmetry gives \(L_{\y}(r)=0\). This proves \eqref{eq:constrained-full-ideal-annihilation}. In particular, the functional $\widetilde{L}_{\y}$ on $\mathcal B$ is well defined exactly as in \eqref{eq:ell-y-definition}.

Use the isomorphism \(\Theta_{\mathrm{inv}}:\mathcal B\to\R[V]^G\) from
\eqref{eq:symmetric-invariant-isomorphism} and define
\(\widehat{L}_{\y}\) by \eqref{eq:Reynolds-extended-functional}. Let
\(P\in\R[V]\) satisfy \(P\geq0\) on \(K_{\bg,\bh}\). For
\(\mathbf z\in K_{\bg,\bh}\), invariance of the feasible set gives
\(g\cdot \mathbf z\in K_{\bg,\bh}\) for every \(g\in G\), and therefore
$(\mathscr R_GP)(\mathbf z)=
\int_G P(g\cdot \mathbf z)\,dg
\geq0$.
Thus the invariant polynomial \(\mathscr R_GP\), viewed through \(\Theta_{\mathrm{inv}}^{-1}\),
is nonnegative on \(K_{\bg,\bh}\), and $\widehat{L}_{\y}(P)\geq0$.
The Riesz--Haviland theorem~\cite[Theorem~3.1.2]{marshall2008positive} yields a positive Borel measure
\(\mu\), supported on \(K_{\bg,\bh}\), such that
$\widehat{L}_{\y}(P)=\int_{K_{\bg,\bh}}P\,d\mu, P\in\R[V]$.
Applying this identity to
$P_p=\RR(p(\ba,\bar\ba))$, exactly as in the proof of
Theorem~\ref{thm:real-quaternionic-riesz-haviland}, gives
$L_{\y}(p)=
\int_{K_{\bg,\bh}}
\RR\bigl(p(\ba,\bar\ba)\bigr)\,d\mu(\ba)$.
Hence $\mu$ is the required representing measure.
\end{proof}

The constrained theorem is exact but requires positivity on every symmetric polynomial nonnegative on $K_{\bg,\bh}$.  Under the Archimedean hypothesis, Theorem~\ref{thm:putinar-real-equalities} converts this Haviland condition into positivity on the quadratic module together with annihilation of the equality and scalar-identity ideals.  This gives the following algebraic characterization.
\begin{corollary}
\label{cor:archimedean-full-real-quaternionic-moment}
Assume that $\mathcal Q^{\R}(\bg,\bh)$ is Archimedean. A real-valued quaternionic pseudo-moment sequence \(\y\) has a positive Borel representing measure supported on
\(K_{\bg,\bh}\) if and only if all of the following conditions hold:
\begin{align}
L_{\y}(p^*p)&\geq0
&& (p\in\mathcal F_n),
\label{eq:arch-full-moment-positive}\\
L_{\y}(p^*g_ip)&\geq0
&& (p\in\mathcal F_n,\ i=1,\ldots,m),\label{eq:arch-full-localizing-positive}\\
L_\y(u h_jv)&=0
&&(u,v\in\cW,\ j=1,\ldots,\ell),\label{eq:arch-full-equality}\\
L_\y(u[w+w^*,q_i]v)&=0
&&(u,v,w\in\cW,\ i=1,\ldots,n).\label{eq:arch-full-scalar}
\end{align}
\end{corollary}

\begin{proof}
Necessity is immediate. Conversely, conditions
\eqref{eq:arch-full-moment-positive}--\eqref{eq:arch-full-scalar} induce a
well-defined linear functional on \(\mathcal B_\bh\) that is nonnegative on
\(\mathcal Q_\bh(\bg)\). Let \(f\in\sym\mathcal F_n\) be nonnegative on
\(K_{\bg,\bh}\). For every \(\varepsilon>0\), the polynomial
\(f+\varepsilon\) is strictly positive on \(K_{\bg,\bh}\), with the usual vacuous interpretation if the feasible set is empty. The real-coefficient
quotient Positivstellensatz, Theorem~\ref{thm:putinar-real-equalities}, gives
$\widehat f+\varepsilon\in \mathcal Q_\bh(\bg)$.
Applying the induced functional and using
\eqref{eq:arch-full-moment-positive}--\eqref{eq:arch-full-localizing-positive}
yields
$L_{\y}(f)+\varepsilon L_{\y}(1)\geq0$.
Letting \(\varepsilon\downarrow0\) gives \(L_{\y}(f)\geq0\). Theorem~\ref{thm:constrained-real-quaternionic-riesz-haviland} now gives a representing measure supported on \(K_{\bg,\bh}\).
\end{proof}

For subsequent semidefinite relaxations, it is convenient to express the preceding conditions in matrix form.  This leads to the quaternionic moment and localizing matrices. Given a real-valued quaternionic pseudo-moment sequence \(\y\),
the \emph{moment matrix} $\M_d(\y)$ is defined by
\begin{equation}\label{eq:full-moment-lmatrices}
\left[\M_d(\y)\right]_{u,v}\coloneqq L_{\y}(u^*v)=y_{u^*v},
\quad u,v\in\cW_{d}.
\end{equation}
Let $p=\sum_{w}p_ww\in\sym\R\langle\q,\bar\q\rangle$.
The \emph{localizing matrix} $\M_d(p\y)$ associated with $p$ is defined by
\begin{equation}\label{eq:full-local-lmatrices}
    \left[\M_{d}(p\y)\right]_{u,v}\coloneqq L_{\y}(u^*pv)=\sum_{w}p_wy_{u^*wv},
\quad u,v\in\cW_{d}.
\end{equation}
Conditions~\eqref{eq:arch-full-moment-positive} and~\eqref{eq:arch-full-localizing-positive} are equivalent to the positive semidefiniteness of all matrices in \eqref{eq:full-moment-lmatrices}--\eqref{eq:full-local-lmatrices}. Thus, under the Archimedean hypothesis,
\begin{equation}\label{eq:boxed-full-moment-characterization}
\boxed{
\begin{gathered}
\y\text{ has a representing measure on }K_{\bg,\bh}\\
\Longleftrightarrow\\[-1mm]
\begin{aligned}
&\M_d(\y)\succeq0\quad(d\geq0),\\
&\M_d(g_i\y)\succeq0
\quad(d\geq0,\ i=1,\ldots,m),\\
&L_\y(u h_jv)=0
\quad(u,v\in\cW,\ j=1,\ldots,\ell),\\
&L_\y(u[w+w^*,q_i]v)=0
\quad(u,v,w\in\cW,\ i=1,\ldots,n).
\end{aligned}
\end{gathered}}
\end{equation}

\subsection{The full quaternion-valued quaternionic moment problem}
\label{sec:full-quaternion-valued-moment}
We next turn to the quaternion-valued moment problem.  In contrast with the preceding real-valued setting, the moment functional now retains the full quaternionic value of each word, while positivity is tested through its real part on symmetric quaternionic polynomials.  We first establish a Riesz--Haviland-type theorem for an arbitrary closed set, then specialize it to the Archimedean semialgebraic setting and derive a moment--localizing matrix characterization.

Let
$\y=(y_w)_{w\in\cW}\in\H^{|\cW|}$ be a quaternion-valued sequence. We call $\y$ a \emph{quaternionic pseudo-moment sequence} if
\begin{enumerate}[(i)]
    \item $y_w=\bar y_{w^*}$ for $w\in\cW$;
    \item $\RR(y_u)=\RR(y_{v})$ if $u\stackrel{\mathrm{cyc}}{\sim}v$ for $u,v\in\cW$.
\end{enumerate}
A quaternionic pseudo-moment sequence $\y$ defines a real-linear functional
$L^{\RR}_{\y}:\P\langle\q,\bar\q\rangle\to\R$ by
\begin{equation}\label{eq:full-H-functional}
p=\sum_{a,b}ap_{a,b}b\longmapsto L^{\RR}_{\y}(p)=\RR\left(\sum_{a,b}p_{a,b}y_{ba}\right).
\end{equation}
One has
$L_{\y}^{\RR}(p^*)=L_{\y}^{\RR}(p)$ for every
$p\in\P\langle\q,\bar\q\rangle$.
Let \(K\subseteq\H^n\) be closed.  A positive Borel measure \(\mu\) on
\(K\) is a \emph{\(K\)-representing measure}
for \(\y\) if
\begin{equation}\label{eq:full-H-representing-measure}
y_w
=
\int_K w(\ba,\bar\ba)\,d\mu(\ba)
\qquad(w\in\cW),
\end{equation}
where the quaternion-valued integral is taken componentwise.
Let
\begin{equation}\label{eq:Pos-P-K}
\mathcal P^{\H}_K
\coloneqq
\left\{
p\in\sym\P\langle\q,\bar\q\rangle:
p(\ba,\bar\ba)\geq0\ \text{for every }\ba\in K
\right\}.
\end{equation}
The proof of the quaternion-valued Riesz--Haviland theorem requires extending a positive functional from the real coordinate polynomials obtained from symmetric quaternionic polynomials to the full real polynomial ring.  We isolate the required order-theoretic extension argument in the following elementary lemma.  It is formulated for a cone-induced preorder because the cone of polynomials nonnegative on a closed set need not be pointed.

\begin{lemma}
\label{lem:positive-extension-majorizing}
Let \(V\) be a real vector space, let \(V_+\subseteq V\) be a convex cone, and
write
\[
r\leq s
\quad\Longleftrightarrow\quad
s-r\in V_+.
\]
The relation \(\leq\) is allowed to be a preorder.  Let \(E\subseteq V\) be a
linear subspace that is \emph{majorizing}: for every \(r\in V\), there are
\(e_-,e_+\in E\) such that
$e_-\leq r\leq e_+$.
If \(\lambda:E\to\R\) is linear and nonnegative on \(E\cap V_+\), then
\(\lambda\) admits a linear extension
\(\widetilde\lambda:V\to\R\) that is nonnegative on \(V_+\).
\end{lemma}

\begin{proof}
We first extend the functional along one new direction. Let \(r\in V\setminus E\), and define
\begin{align*}
\alpha
&\coloneqq
\sup\{\lambda(e):e\in E,\ e\leq r\},\\
\beta
&\coloneqq
\inf\{\lambda(f):f\in E,\ r\leq f\}.
\end{align*}
Choose \(e_-,e_+\in E\) with \(e_-\leq r\leq e_+\).  If \(e\in E\) and
\(e\leq r\), then \(e\leq e_+\), so positivity gives
\(\lambda(e)\leq\lambda(e_+)\).  If \(f\in E\) and \(r\leq f\), then
\(e_-\leq f\), so \(\lambda(e_-)\leq\lambda(f)\).  Consequently,
\begin{equation}\label{eq:alpha-beta-finite}
\lambda(e_-)
\leq\alpha
\leq\beta
\leq\lambda(e_+),
\end{equation}
where \(\alpha\leq\beta\) follows from
\(e\leq r\leq f\Rightarrow\lambda(e)\leq\lambda(f)\).  Thus
\(\alpha,\beta\in\R\).

Choose \(c\in[\alpha,\beta]\) and set
\[
\lambda_c(e+tr)\coloneqq\lambda(e)+tc,
\qquad e\in E,\ t\in\R.
\]
If \(e+tr\in V_+\) and \(t>0\), then
$-\frac{e}{t}\leq r$,
so \(-\lambda(e)/t\leq\alpha\leq c\), which gives
\(\lambda(e)+tc\geq0\).  If \(e+tr\in V_+\) and \(t<0\), then
$r\leq-\frac{e}{t}$,
so \(c\leq\beta\leq-\lambda(e)/t\), again giving
\(\lambda(e)+tc\geq0\).  The case \(t=0\) follows from positivity of
\(\lambda\).  Hence \(\lambda_c\) is a positive extension to
\(E+\R r\).

Partially order all positive linear extensions of $\lambda$ by inclusion of their graphs. The union of any chain is an upper bound, so Zorn's lemma yields a maximal positive extension.  If its domain were not all of \(V\), the preceding one-dimensional
construction would extend it further, a contradiction.
\end{proof}

With the extension lemma available, we can now characterize the existence of a representing measure for a full quaternion-valued pseudo-moment sequence on an arbitrary closed set.
\begin{theorem}
\label{thm:quaternionic-riesz-haviland}
Let \(K\subseteq\H^n\) be closed, and let
\(\y=(y_w)_{w\in\cW}\in\H^{|\cW|}\) be a quaternionic pseudo-moment sequence.  The
following are equivalent.
\begin{enumerate}[(i)]
\item The sequence \(\y\) admits a positive Borel \(K\)-representing measure.
\item For every $p\in\mathcal P^{\H}_K$, one has $L_{\y}^{\RR}(p)\geq0$.
\end{enumerate}
\end{theorem}

\begin{proof}
For (i)$\Rightarrow$(ii), let
$p=\sum_{u,v}u p_{u,v}v\in\mathcal P^{\H}_K$. Then
\begin{align*}
L_{\y}^{\RR}(p)
&=
\RR\!\left(
\sum_{u,v}p_{u,v}
\int_K v(\ba)u(\ba)\,d\mu(\ba)
\right)\notag\\
&=
\int_K
\RR\!\left(
\sum_{u,v}u(\ba)p_{u,v}v(\ba)
\right)d\mu(\ba)\notag\\
&=
\int_K p(\ba,\bar\ba)\,d\mu(\ba)\ge0.
\end{align*}

For (ii)$\Rightarrow$(i), identify \(\H^n\) with \(\R^{4n}\) by writing
$q_j=x_{j1}+x_{j2}\i+x_{j3}\j+x_{j4}\k$.
Let $\rho:\H\langle\q,\bar\q\rangle
\longrightarrow
\H[x_{11},\ldots,x_{n4}]$
be the scalar coordinate homomorphism introduced in Section~\ref{sec:qsos}.
For $p\in\sym\P\langle\q,\bar\q\rangle$, the coordinate polynomial $\rho(p)$ is real-valued and therefore belongs to \(\R[x_{11},\ldots,x_{n4}]\).  Put
\begin{equation}\label{eq:full-H-E-space}
E
\coloneqq
\rho\bigl(\sym\P\langle\q,\bar\q\rangle\bigr)
\subseteq
V\coloneqq\R[x_{11},\ldots,x_{n4}].
\end{equation}
Define
\begin{equation}\label{eq:full-H-lambda-on-E}
\lambda:E\longrightarrow\R,
\qquad
\lambda(\rho(p))\coloneqq L_{\y}^{\RR}(p).
\end{equation}
This definition is well defined. If \(\rho(p)=0\), then both \(p\) and \(-p\) vanish on
\(K\) and hence belong to \(\mathcal P_K^\H\).  Therefore
\(L_{\y}^{\RR}(p)=0\).  Moreover,
\begin{equation}\label{eq:full-H-lambda-positive}
r\in E,\quad r|_K\geq0
\quad\Longrightarrow\quad
\lambda(r)\geq0.
\end{equation}

Define the cone $V_+$ and its induced preorder by
\begin{equation}\label{eq:K-polynomial-cone}
V_+
\coloneqq
\{r\in V:r|_K\geq0\},
\qquad
r\leq_Ks
\quad\Longleftrightarrow\quad
s-r\in V_+.
\end{equation}
Unless $K$ is Zariski dense, $\leq_K$ is only a preorder. This causes no difficulty, because Lemma~\ref{lem:positive-extension-majorizing} does not require the cone to be pointed.  Equivalently, one may pass to the ordered
quotient by
\[
N_K
\coloneqq
V_+\cap(-V_+)
=
\{r\in V:r|_K=0\}.
\]
We claim that $E$ is majorizing in $(V,V_+)$. Given \(r\in V\), choose
\(C>0\) and \(N\in\N\) such that
\[
|r(\x)|\leq C(1+\|\x\|^2)^N
\qquad(\x\in\R^{4n}).
\]
The symmetric real-coefficient polynomial
\[
S(\q,\bar\q)
\coloneqq
C\left(1+\sum_{j=1}^n\bar q_j q_j\right)^N
\in\sym\P\langle\q,\bar\q\rangle
\]
satisfies $\rho(S)=C(1+\|\x\|^2)^N\in E$.
Thus $-\rho(S)\leq_K r\leq_K\rho(S)$,
which proves majorization. Lemma~\ref{lem:positive-extension-majorizing} therefore gives a linear extension
$\widetilde\lambda:V\longrightarrow\R$
with
\[
\widetilde\lambda(r)\geq0
\qquad
\text{whenever }r|_K\geq0.
\]
Since $K$ is closed in $\R^{4n}$, the classical Riesz--Haviland theorem~\cite[Theorem~3.1.2]{marshall2008positive} yields a positive Borel measure
\(\mu\), supported on \(K\) and having moments of all orders, such that
\begin{equation}\label{eq:full-H-coordinate-representation}
\widetilde\lambda(r)
=
\int_Kr(x)\,d\mu(x)
\qquad(r\in V).
\end{equation}

It remains to recover the quaternionic moments. Fix \(w\in\cW\) and
\(c\in\H\), and set
\begin{equation}\label{eq:full-H-test-polynomial}
F_{c,w}
\coloneqq
\frac12\bigl(cw+w^*\bar c\bigr)
\in\sym\P\langle\q,\bar\q\rangle.
\end{equation}
Pointwise, $F_{c,w}(\ba,\bar\ba)
=\RR\bigl(cw(\ba,\bar\ba)\bigr)$.
We have
\begin{align}
L_{\y}^{\RR}(F_{c,w})
&=
\frac12\RR\bigl(cy_w+\bar c\,y_{w^*}\bigr)
=
\RR(cy_w).
\label{eq:full-H-test-functional}
\end{align}
On the other hand, \eqref{eq:full-H-coordinate-representation} gives
$L_{\y}^{\RR}(F_{c,w})
=\RR\!\left(
 c\int_K w(\ba,\bar\ba)\,d\mu(\ba)
\right)$.
Hence
\[
\RR\!\left(
 c\left[y_w-
 \int_K w(\ba,\bar\ba)\,d\mu(\ba)
 \right]
\right)=0
\qquad(c\in\H).
\]
Taking $c$ to be the quaternionic conjugate of the bracketed difference shows that its squared modulus is zero.  Therefore
\eqref{eq:full-H-representing-measure} holds for every word \(w\), completing
the proof.
\end{proof}

We next specialize Theorem~\ref{thm:quaternionic-riesz-haviland} to the quaternionic basic semialgebraic set $K_{\bg,\bh}$ from \eqref{eq:Kgh-real}.  Throughout the remainder of this subsection, $g_i\in\sym\R\langle\q,\bar\q\rangle$ and $h_j\in\R\langle\q,\bar\q\rangle$.  Under the Archimedean condition, the quaternionic Positivstellensatz converts Haviland positivity into positivity on the certificate cone.

\begin{corollary}
\label{cor:archimedean-full-quaternionic-moment}
Assume the Archimedean condition
\eqref{eq:QH-archimedean}.  A full quaternion-valued quaternionic pseudo-moment sequence
\(\y\) has a positive Borel representing measure supported on
\(K_{\bg,\bh}\) if and only if
\begin{equation}\label{eq:full-H-module-positivity}
L_{\y}^{\RR}(p)\geq0,\qquad
\forall p\in\mathcal Q^{\H}(\bg,\bh)+\left(\mathcal I_{n}^\H\cap\sym\P\langle\q,\bar\q\rangle\right).
\end{equation}
\end{corollary}

\begin{proof}
If $\mu$ is supported on $K_{\bg,\bh}$, then every QSOS term is nonnegative, every weighted QSOS term has the scalar value
$g_i(\ba)\sigma_i(\ba)\geq0$, and all equality-ideal and scalar-identity terms vanish.  Integration therefore gives
\eqref{eq:full-H-module-positivity}.

Conversely, let
\(p\in\sym\P\langle\q,\bar\q\rangle\) satisfy
\(p\geq0\) on \(K_{\bg,\bh}\).  For every \(\varepsilon>0\),
\(p+\varepsilon\) is strictly positive on the feasible set, with the usual vacuous interpretation if the feasible set is empty.  Theorem~\ref{thm:putinar-quaternionic}
gives $p+\varepsilon\in
\mathcal Q^{\H}(\bg,\bh)+\left(\mathcal I_{n}^\H\cap\sym\P\langle\q,\bar\q\rangle\right)$.
Thus
$L_{\y}^{\RR}(p)+\varepsilon L_{\y}^{\RR}(1)\geq0$.
Letting \(\varepsilon\downarrow0\) yields
\(L_{\y}^{\RR}(p)\geq0\).  Theorem~\ref{thm:quaternionic-riesz-haviland},
applied to the closed set \(K_{\bg,\bh}\), now provides the required
representing measure.
\end{proof}

For use in the hierarchy of the next section, we now translate the preceding certificate criterion into explicit moment and localizing matrix conditions.  Let
\[
p=\sum_{a\in\cW_d}c_a a\in\mathcal L,
\qquad
\widehat c_a\coloneqq\bar c_{a^*}
\quad(a\in\cW_d),
\]
and $g=\sum_{w}g_ww\in\sym\R\langle\q,\bar\q\rangle$.
Then
\begin{align}
\widehat c^*\M_d(g\y)\widehat c
=
\sum_{u,v,w}
c_{u^*}g_wy_{u^*wv}\bar c_{v^*}
=
\sum_{a,b,w}c_bg_wy_{bwa^*}\bar c_a.
\label{eq:full-H-localizing-vector-form}
\end{align}
Because $\M_d(g\y)$ is Hermitian, the quaternion in
\eqref{eq:full-H-localizing-vector-form} is real.  Cyclic invariance of the
real part therefore gives
\begin{equation}\label{eq:full-H-localizing-vector-real}
\widehat c^*\M_d(g\y)\widehat c
=
\RR\!\left(
\sum_{a,b,w}\bar c_ac_bg_wy_{bwa^*}
\right).
\end{equation}
Since $p^*p=
\sum_{a,b}a^*\bar c_a c_b b$,
the definition of \(L_{\y}^{\RR}\) yields
\begin{equation}
L_{\y}^{\RR}\bigl((p^*p)g\bigr)
=
\RR\!\left(
\sum_{a,b,w}\bar c_ac_bg_wy_{bwa^*}
\right),\quad
L_{\y}^{\RR}\bigl(g(p^*p)\bigr)
=
\RR\!\left(
\sum_{a,b,w}\bar c_ac_bg_wy_{bwa^*}
\right).
\notag
\end{equation}
Consequently,
\begin{equation}\label{eq:full-H-localizing-identity}
L_{\y}^{\RR}\!\left(
\frac12\bigl((p^*p)g+g(p^*p)\bigr)
\right)
=
\widehat c^*\M_d(g\y)\widehat c.
\end{equation}
The map \(c\mapsto\widehat c\) is a bijection of quaternionic coefficient
vectors.  Hence positivity against all left-coefficient polynomials \(p\) is
equivalent to $\M_d(g\y)\succeq0$.

The only additional work is to identify precisely how the equality and scalar-identity parts of the certificate are encoded by the quaternion-valued moments.
\begin{proposition}\label{prop:scalar-ideal-generator}
Let
\(\y=(y_w)_{w\in\cW}\in\H^{|\cW|}\) be a quaternionic pseudo-moment sequence. Then the following equivalences hold:
\begin{equation}\label{eq:full-H-equality-equivalence}
L_{\y}^{\RR}\bigl(\mathcal J_\bh^{\H}\cap\sym\P\langle\q,\bar\q\rangle\bigr)=0\Longleftrightarrow
L_\y(u h_jv)=0,\,
u,v\in\cW,\, j=1,\ldots,\ell,
\end{equation}
and
\begin{equation}\label{eq:full-scalar-equality-equivalence}
L_{\y}^{\RR}\left(\mathcal I_{n}^\H\cap\sym\P\langle\q,\bar\q\rangle\right)=0\Longleftrightarrow L_\y(u[w+w^*,q_i]v)=0,\, u,v,w\in\cW,\, i=1,\ldots,n.
\end{equation}
\end{proposition}
\begin{proof}
First consider \eqref{eq:full-H-equality-equivalence}.
To prove the forward implication, put \(s=u h_j v\) and, for \(c\in\H\), set
\[
H_{c,s}
\coloneqq
\frac12\bigl(cs+s^*\bar c\bigr)
\in
\mathcal J_\bh^{\H}\cap\sym\P\langle\q,\bar\q\rangle.
\]
Then $L_{\y}^{\RR}(H_{c,s})=\RR(cL_\y(s))$.
If the left-hand side of \eqref{eq:full-H-equality-equivalence} holds, this quantity vanishes for every $c\in\H$.  Taking \(c=\overline{L_\y(s)}\) gives
\(L_\y(s)=0\).
Conversely, assume the right-hand side of
\eqref{eq:full-H-equality-equivalence}. By real linearity,
\(L_\y(s)=0\) for every two-sided multiple of an \(h_j\).  The same conclusion
holds for multiples of \(h_j^*\), because $L_\y(u h_j^*v)=\overline{L_\y(v^*h_ju^*)}$.
Thus
\begin{equation}\label{eq:full-H-Ih-moments-zero}
L_\y(s)=0
\qquad(s\in \mathcal J_{\bh}^\R).
\end{equation}
If \(1\in \mathcal J_{\bh}^\R\), then \(\mathcal J_{\bh}^\R=\mathcal F_n\), and \eqref{eq:full-H-Ih-moments-zero} immediately makes
\(L_{\y}^{\RR}\) vanish on all of \(\P\langle\q,\bar\q\rangle\), so the converse follows.  Suppose
that \(1\notin \mathcal J_{\bh}^\R\).  The free-product decomposition established in
\eqref{eq:J-intersection-P} gives
\begin{equation}\label{eq:full-H-Jh-P-decomposition}
\mathcal J_\bh^{\H}\cap\P\langle\q,\bar\q\rangle
\simeq
\mathcal J_{\bh}^\R
\oplus
\left(
\mathcal J_{\bh}^\R\otimes_{\R}\H_0\otimes_{\R}\mathcal F_n
+\mathcal F_n\otimes_{\R}\H_0\otimes_{\R}\mathcal J_{\bh}^\R
\right).
\end{equation}
For \(s\in \mathcal J_{\bh}^\R\), one has
\(L_{\y}^{\RR}(s)=\RR(L_{\y}(s))=0\).  For a basic one-coefficient term
\(acb\), where \(a\in \mathcal J_{\bh}^\R\) or \(b\in \mathcal J_{\bh}^\R\),
$L_{\y}^{\RR}(acb)=
\RR(cL_{\y}(ba))=0$,
because \(ba\in \mathcal J_{\bh}^\R\).  Equation
\eqref{eq:full-H-Jh-P-decomposition} therefore implies
$L_{\y}^{\RR}(\mathcal J_\bh^{\H}\cap\P\langle\q,\bar\q\rangle)=0$,
and in particular the left-hand side of
\eqref{eq:full-H-equality-equivalence} holds.

Now turn to \eqref{eq:full-scalar-equality-equivalence}.
For $w\in\cW$ and $i=1,\ldots,n$, set $s_{w,i}\coloneqq[w+w^*,q_i]$.
By Proposition~\ref{prop:Isa-internal-generators},
$\mathcal I_n^\H\cap\sym\P\langle\q,\bar\q\rangle$ is spanned over
$\R$ by the polynomials
\[
\kappa_{a,e,b}
=\bigl(aeb+(aeb)^*\bigr)-\bigl(eba+(eba)^*\bigr)
\]
and
\[
\lambda_{e,u,v;w,i}^{\varepsilon}
=eu s_{w,i}^{\varepsilon}v
+\bigl(eu s_{w,i}^{\varepsilon}v\bigr)^*,
\qquad \varepsilon\in\{1,*\},
\]
where $a,b,u,v,w\in\cW$ and $e\in\{1,\i,\j,\k\}$.
We first note that, for every $r\in\mathcal F_n$ and $e\in\H$,
$L_\y^{\RR}(er)=\RR\bigl(eL_\y(r)\bigr)$.
Since $L_\y^{\RR}(p^*)=L_\y^{\RR}(p)$, it follows that
\begin{equation}\label{eq:scalar-ideal-generator-sym-eval}
L_\y^{\RR}\bigl(er+(er)^*\bigr)
=2\RR\bigl(eL_\y(r)\bigr).
\end{equation}
Moreover, for $a,b\in\cW$,
\[
L_\y^{\RR}(aeb)=\RR(e y_{ba})=L_\y^{\RR}(eba),
\]
and hence
\begin{equation}\label{eq:scalar-ideal-generator-kappa-zero}
L_\y^{\RR}(\kappa_{a,e,b})=0.
\end{equation}
Thus the $\kappa$-generators are annihilated automatically.

Suppose first that
$L_\y^{\RR}(p)=0$ for every $p\in
\mathcal I_n^\H\cap\sym\P\langle\q,\bar\q\rangle$.
Fix $u,v,w\in\cW$ and $i\in\{1,\ldots,n\}$, and put
\[
z\coloneqq L_\y(us_{w,i}v)\in\H.
\]
For every $e\in\{1,\i,\j,\k\}$,
$\lambda_{e,u,v;w,i}^{1}$ belongs to
$\mathcal I_n^\H\cap\sym\P\langle\q,\bar\q\rangle$; therefore,
by \eqref{eq:scalar-ideal-generator-sym-eval},
\[
0=L_\y^{\RR}\bigl(\lambda_{e,u,v;w,i}^{1}\bigr)
=2\RR(ez).
\]
Since $1,\i,\j,\k$ form a real basis of $\H$, real linearity gives
$\RR(ez)=0$ for every $e\in\H$. Taking $e=\bar z$ yields
$|z|^2=\RR(\bar z z)=0$, and hence $z=0$. Thus
$L_\y\bigl(u[w+w^*,q_i]v\bigr)=0$
for all $u,v,w\in\cW$ and $i=1,\ldots,n$.

Conversely, suppose that
\begin{equation}\label{eq:scalar-ideal-generator-moment-zero}
L_\y(us_{w,i}v)=0
\qquad
(u,v,w\in\cW,\ i=1,\ldots,n).
\end{equation}
Because $\y$ is Hermitian,
$L_\y(r^*)=\overline{L_\y(r)}$ for every $r\in\mathcal F_n$. Hence
\[
L_\y(us_{w,i}^*v)
=
\overline{L_\y\bigl((us_{w,i}^*v)^*\bigr)}
=
\overline{L_\y(v^*s_{w,i}u^*)}
=0,
\]
where the last equality follows from
\eqref{eq:scalar-ideal-generator-moment-zero}. Therefore
$L_\y(us_{w,i}^{\varepsilon}v)=0, \varepsilon\in\{1,*\}$.
Using \eqref{eq:scalar-ideal-generator-sym-eval}, we obtain
\[
L_\y^{\RR}\bigl(\lambda_{e,u,v;w,i}^{\varepsilon}\bigr)
=
2\RR\bigl(eL_\y(us_{w,i}^{\varepsilon}v)\bigr)
=0.
\]
Together with \eqref{eq:scalar-ideal-generator-kappa-zero}, this shows that
$L_\y^{\RR}$ annihilates every generator in the spanning family of
Proposition~\ref{prop:Isa-internal-generators}. By real linearity,
$L_\y^{\RR}(p)=0$
for every $p\in
\mathcal I_n^\H\cap\sym\P\langle\q,\bar\q\rangle$.
This proves the equivalence.
\end{proof}

The moment/localizing positivity relations and the preceding ideal-annihilation equivalence account for every component of the Archimedean certificate.  Combining them with Corollary~\ref{cor:archimedean-full-quaternionic-moment} yields the following complete matrix criterion.

\begin{corollary}
\label{cor:full-H-matrix-characterization}
Assume the Archimedean condition \eqref{eq:QH-archimedean}, and let \(\y\) be a full quaternion-valued quaternionic pseudo-moment sequence.  Then \(\y\) has a positive Borel representing measure
supported on \(K_{\bg,\bh}\) if and only if
\begin{align}
\M_d(\y)&\succeq0
&& (d\geq0),
\label{eq:full-H-all-moment-PSD}\\
\M_d(g_i\y)&\succeq0
&& (d\geq0,\ i=1,\ldots,m),
\label{eq:full-H-all-localizing-PSD}\\
L_\y(u h_jv)&=0
&& (u,v\in\cW,\ j=1,\ldots,\ell),
\label{eq:full-H-all-equality-moments}\\
L_\y(u[w+w^*,q_i]v)&=0, &&(u,v,w\in\cW,\, i=1,\ldots,n).
\label{eq:full-H-all-scalar-identities}
\end{align}
\end{corollary}

\begin{proof}
By \eqref{eq:full-H-localizing-identity}, conditions
\eqref{eq:full-H-all-moment-PSD}--\eqref{eq:full-H-all-localizing-PSD} are
exactly positivity on the QSOS and weighted QSOS parts of
\(\mathcal Q^{\H}(\bg,\bh)\).  Because both $\eta$ and $-\eta$ are admissible whenever
$\eta\in \mathcal J_\bh^{\H}\cap\sym\P\langle\q,\bar\q\rangle$, positivity on the module forces and is forced
by annihilation of the equality-ideal part; this is equivalent to
\eqref{eq:full-H-all-equality-moments} by
\eqref{eq:full-H-equality-equivalence}.  The same arbitrary-sign argument
shows that the scalar-identity part is equivalent to
\eqref{eq:full-H-all-scalar-identities}.  Thus
\eqref{eq:full-H-all-moment-PSD}--\eqref{eq:full-H-all-scalar-identities} are
equivalent to \eqref{eq:full-H-module-positivity}, and the conclusion follows
from Corollary~\ref{cor:archimedean-full-quaternionic-moment}.
\end{proof}

% The scalar-identity condition in \eqref{eq:full-H-all-scalar-identities} is essential rather than merely formal.  Even in one variable, positivity of all moment matrices does not prevent a sequence from arising from a matrix representation instead of scalar quaternionic evaluation.

% \begin{example}
% Let
% \[
% A=
% \begin{pmatrix}
% 0&1\\
% 0&0
% \end{pmatrix},
% \]
% and, for every word \(w\) in \(q,\bar q\), define
% \begin{equation}\label{eq:full-H-matrix-trace-sequence}
% y_w
% \coloneqq
% \frac12\operatorname{Tr}\bigl(w(A,A^{\intercal})\bigr)
% \in\R\subseteq\H.
% \end{equation}
% Then $y_1=1$. Adjoint symmetry follows from invariance of the trace under transposition, and cyclicity follows from the cyclic trace identity.  If
% \(W_u=u(A,A^{\intercal})\), then
% \[
% [\M_d(\y)]_{u,v}
% =
% \frac12\operatorname{Tr}(W_u^{\intercal}W_v),
% \]
% so \(\M_d(\y)\) is a real Gram matrix and therefore is quaternionic positive
% semidefinite.

% However,
% \[
% r
% \coloneqq
% (q\bar q-\bar qq)^*(q\bar q-\bar qq)
% \in \mathcal I_1^{\H}\cap\sym\P\langle\q,\bar\q\rangle
% \]
% vanishes at every scalar quaternion.  At the matrix tuple,
% \[
% AA^{\intercal}-A^{\intercal}A
% =
% \begin{pmatrix}
% 1&0\\
% 0&-1
% \end{pmatrix},
% \]
% and therefore $L_{\y}^{\RR}(r)=
% \frac12\operatorname{Tr}(I_2)
% =1$.
% Every scalar quaternionic representing measure would annihilate \(r\), a
% contradiction.
% \end{example}

Existence of a representing measure is therefore characterized by the preceding conditions, but uniqueness is a separate issue.  In general, even the full quaternion-valued moment sequence does not determine its representing measure uniquely, as the following example shows.
\begin{example}
\label{rem:full-H-nonunique-measure}
For one variable, consider
\[
\mu_1
=
\frac12\left(\delta_{\i}+\delta_{-\i}\right),
\qquad
\mu_2
=
\frac12\left(\delta_{\j}+\delta_{-\j}\right).
\]
Let \(w\) have length \(m\), and let \(r(w)\) be the number of occurrences of
\(\bar q\) in \(w\).  On either imaginary-axis support,
$w(q,\bar q)=(-1)^{r(w)}q^m$.
Consequently,
\begin{equation}\label{eq:full-H-nonunique-word-moments}
\int w\,d\mu_1
=
\int w\,d\mu_2
=
\begin{cases}
0,&m\ \text{odd},\\[1mm]
(-1)^{r(w)+m/2},&m\ \text{even}.
\end{cases}
\end{equation}
Thus each fixed even word has the same scalar value under the two measures, although different even words may take the values $1$ or $-1$. Hence, in general, the full moments do not determine the angular component of a representing measure.
\end{example}

\subsection{The truncated quaternionic moment problem}
\label{sec:truncated-quaternionic-moment}
The full moment criteria above involve moments of all orders.  In polynomial optimization, however, only finitely many moments are available at a fixed relaxation order.  We therefore study when a truncated positive moment functional already determines a representing measure.  The central mechanism is flatness: once the column space of the moment matrix stabilizes, one obtains a finite-dimensional $*$-representation.  Because such a representation need not automatically correspond to scalar quaternions, scalar consistency with $\mathcal I_n^\R$ must also be imposed in general.

Given a truncated quaternionic pseudo-moment sequence $\y=(y_w)_{w\in\cW_{2d}}\in\H^{|\cW_{2d}|}$, assume the following flatness condition:
\begin{equation}\label{eq:flat-truncated-PSD-rank}
\M_d(\y)\succeq0,
\qquad
\rank_{\H}\M_d(\y)=\rank_{\H}\M_{d-1}(\y)=r.
\end{equation}
Since $\M_d(\y)\succeq0$ has quaternionic rank $r$, there exist vectors $\bv_w\in\H^r$, indexed by $w\in\mathcal W_d$, such that
$\langle\bv_u,\bv_v\rangle=L_\y(u^*v)$ for all $u,v\in\mathcal W_d$.
The inner product is conjugate-linear in its first argument and right-linear in
its second argument.  For $0\leq t\leq d$, put
\[
\mathscr H_t
\coloneqq
\operatorname{span}_{\H}
\{\bv_w:w\in\mathcal W_t\}.
\]
Then $\dim_{\H}\mathscr H_t=\rank_{\H}\M_t(\y)$.
The flatness assumption and the inclusion
$\mathscr H_{d-1}\subseteq\mathscr H_d$ therefore give
$\mathscr H_{d-1}=\mathscr H_d$.
Set $\mathscr H\coloneqq\mathscr H_{d-1}=\mathscr H_d$ and $\xi\coloneqq\bv_1$. Then $\dim_{\H}\mathscr H=r$.

For a letter $a\in\{q_1,\bar q_1,\ldots,q_n,\bar q_n\}$,
define, initially on the spanning family of $\mathscr H$,
\begin{equation}\label{eq:flat-shift-definition}
X_a \bv_w\coloneqq \bv_{aw},
\qquad
w\in\mathcal W_{d-1}.
\end{equation}
We show that this rule is well defined. Suppose
$\sum_{w\in\mathcal W_{d-1}}\bv_wc_w=0,
c_w\in\H$.
For every $v\in\mathcal W_{d-1}$,
\begin{equation*}
\left\langle
\bv_v,
\sum_w\bv_{aw}c_w
\right\rangle=
\sum_wL_\y(v^*aw)c_w=
\sum_wL_\y\bigl((a^*v)^*w\bigr)c_w=
\left\langle
\bv_{a^*v},
\sum_w\bv_wc_w
\right\rangle
=0.
\end{equation*}
Here $|a^*v|\leq d$, so $\bv_{a^*v}\in\mathscr H_d=\mathscr H$.
Since the vectors $\bv_v$, $v\in\mathcal W_{d-1}$, span $\mathscr H$, it follows that $\sum_w\bv_{aw}c_w=0$.  Hence $X_a$ is a well-defined right-$\H$-linear operator on $\mathscr H$.

For $u,v\in\mathcal W_{d-1}$,
\begin{equation*}
\langle \bv_u,X_a\bv_v\rangle=L_\y(u^*av)=L_\y\bigl((a^*u)^*v\bigr)=\langle X_{a^*}\bv_u,\bv_v\rangle.
\end{equation*}
Therefore $X_{a^*}=X_a^*$.
For a word $w=a_1\cdots a_m$, define
\[
\pi(w)\coloneqq X_{a_1}\cdots X_{a_m},
\qquad
\pi(1)\coloneqq I_{\mathscr H},
\]
and extend the definition real-linearly. The resulting map
$\pi:\mathcal F_n\longrightarrow B_{\H}(\mathscr H)$
is a unital real $*$-representation.  For every
$w\in\mathcal W_d$, induction on $|w|$ gives $\pi(w)\xi=\bv_w$.
In particular,
\begin{equation}\label{eq:flat-cyclicity}
\mathscr H
=
\operatorname{span}_{\H}
\{\pi(w)\xi:w\in\mathcal W_{d-1}\},
\end{equation}
so $\xi$ is cyclic. We call the triple $(\mathscr H,\pi,\xi)$ the \emph{canonical flat representation} of $\y$.

Before proving the truncated quaternionic moment theorem, we record a structural lemma that will be used below.
\begin{lemma}\label{structural-result}
Let $\mathscr K$ be a finite-dimensional right quaternionic Hilbert space, and let
$\mathscr S\subseteq B_{\H}(\mathscr K)$ be a finite-dimensional unital real $*$-subalgebra satisfying
$\sym\mathscr S\subseteq Z(\mathscr S)$. Then $\mathscr K$ admits an orthogonal decomposition
\begin{equation}\label{eq:orthogonal-line-decomposition}
\mathscr K=E_1\oplus\cdots\oplus E_{\dim_{\H}\mathscr K},
\qquad
\dim_{\H}E_\ell=1,
\end{equation}
such that every $E_\ell$ reduces every operator in $\mathscr S$.
\end{lemma}
\begin{proof}
Because $\mathscr S$ is finite-dimensional, it is norm closed and therefore a finite-dimensional real $C^*$-algebra.  By the standard classification of finite-dimensional real $C^*$-algebras~\cite{goodearl1987classification},
\begin{equation}\label{eq:real-Cstar-decomposition}
\mathscr S
\cong
\bigoplus_{\alpha=1}^sM_{m_\alpha}(\mathbb F_\alpha),
\qquad
\mathbb F_\alpha\in\{\R,\C,\H\}.
\end{equation}
If some $m_\alpha\geq2$, the matrix unit $E_{11}$ in that summand is
self-adjoint but noncentral, contradicting
$\sym\mathscr S\subseteq Z(\mathscr S)$.  Hence $m_\alpha=1$ for every
$\alpha$, and
\begin{equation}\label{eq:division-algebra-decomposition}
\mathscr S
\cong
\bigoplus_{\alpha=1}^s\mathbb F_\alpha.
\end{equation}

Let $e_\alpha\in\mathscr S$ be the identity of the $\alpha$-th summand.  The
$e_\alpha$ are mutually orthogonal central self-adjoint projections and
$\sum_\alpha e_\alpha=I$.  Therefore
\[
\mathscr K
=
\bigoplus_{\alpha=1}^s\mathscr K_\alpha,
\qquad
\mathscr K_\alpha\coloneqq e_\alpha\mathscr K,
\]
and $\mathbb F_\alpha$ acts on $\mathscr K_\alpha$ through a unital
$*$-representation
$\theta_\alpha:\mathbb F_\alpha
\longrightarrow B_{\H}(\mathscr K_\alpha)$.
Regard $\mathscr K_\alpha$ as a real vector space.  Since every
$\theta_\alpha(a)$ is right-$\H$-linear, the left $\mathbb F_\alpha$-action
commutes with right quaternionic multiplication.  Thus $\mathscr K_\alpha$ is
a left module over
$D_\alpha
\coloneqq
\mathbb F_\alpha\otimes_{\R}\H^{\mathrm{op}}$
through
\begin{equation}\label{eq:D-alpha-module-action}
(a\otimes b^{\mathrm{op}})\cdot v
\coloneqq
\theta_\alpha(a)(v)b.
\end{equation}
The three possibilities are
\begin{equation}\label{eq:D-alpha-types}
D_\alpha
\cong
\begin{cases}
\H^{\mathrm{op}},&\mathbb F_\alpha=\R,\\
M_2(\C),&\mathbb F_\alpha=\C,\\
M_4(\R),&\mathbb F_\alpha=\H.
\end{cases}
\end{equation}
Each $D_\alpha$ is simple Artinian.  Its irreducible real modules are,
respectively, $\H$, $\C^2$, and $\R^4$, and hence have real dimension four.
By semisimplicity, $\mathscr K_\alpha$ is an algebraic direct sum of
irreducible $D_\alpha$-modules.  If $E$ is one such irreducible module, then
$E$ is invariant under $1\otimes\H^{\mathrm{op}}$, so it is a right
quaternionic subspace.  Since $\dim_{\R}E=4$, one has
$\dim_{\H}E=1$.

It remains to choose this decomposition orthogonally.  Let $E$ be an irreducible $D_\alpha$-submodule, $x\in E$, $y\in E^\perp$, $a\in\mathbb F_\alpha$, and
$b\in\H$.  Then
\begin{equation*}
\langle x,\theta_\alpha(a)y\rangle=
\langle\theta_\alpha(a^*)x,y\rangle
=0,\qquad
\langle x,yb\rangle=
\langle x,y\rangle b
=0.
\end{equation*}
Thus $E^\perp$ is invariant both under
$\theta_\alpha(\mathbb F_\alpha)$ and under right quaternionic multiplication;
it is therefore again a $D_\alpha$-submodule.  Induction on the quaternionic dimension yields an orthogonal decomposition of each $\mathscr K_\alpha$ into irreducible $D_\alpha$-modules, each of quaternionic dimension one.  Combining
the summands over all $\alpha$ yields \eqref{eq:orthogonal-line-decomposition}.
\end{proof}

\begin{theorem}
\label{thm:flat-truncated-quaternionic-moment}
Let $d\geq1$ and $\y=(y_w)_{w\in\cW_{2d}}\in\H^{|\cW_{2d}|}$ be a truncated quaternionic pseudo-moment sequence with $y_1=1$ satisfying the flatness condition \eqref{eq:flat-truncated-PSD-rank}.
Let $(\mathscr H,\pi,\xi)$ be the canonical flat representation of $\y$ and assume the scalar-consistency condition $\pi(\mathcal I_{n}^\R)=0$.
Then there exist distinct points
$\mathbf q^{(1)},\ldots,\mathbf q^{(N)}\in\H^n,
N\leq r$,
and weights
$\lambda_1,\ldots,\lambda_N>0,
\sum_{\nu=1}^N\lambda_\nu=1$,
such that
$\mu=\sum_{\nu=1}^N\lambda_\nu\delta_{\mathbf q^{(\nu)}}$ is a positive atomic probability measure representing $\y$.
\end{theorem}

\begin{proof}
Define
\begin{equation}\label{eq:flat-image-algebra}
\mathscr R\coloneqq\pi(\mathcal F_n)
\subseteq B_{\H}(\mathscr H),
\qquad
\sym\mathscr R
\coloneqq
\{A\in\mathscr R:A^*=A\}.
\end{equation}
Let $A\in\sym\mathscr R$ and $B\in\mathscr R$, and choose
$p,r\in\mathcal F_n$ such that
$\pi(p)=A, \pi(r)=B$.
Since $[p+p^*,r]\in \mathcal I_n^\R$,
\[
0=\pi([p+p^*,r])=[\pi(p)+\pi(p)^*,\pi(r)]=[2A,B].
\]
Therefore $[A,B]=0$, proving $\sym\mathscr R\subseteq Z(\mathscr R)$. Now applying Lemma~\ref{structural-result} to $\mathscr S=\mathscr R$ and $\mathscr K=\mathscr H$ gives the orthogonal decomposition
\begin{equation}\label{eq:R-reducing-lines}
\mathscr H=E_1\oplus\cdots\oplus E_r,
\qquad
\dim_{\H}E_\ell=1,
\end{equation}
such that every $E_\ell$ reduces $\mathscr R$.
Choose a unit vector $e_\ell\in E_\ell$, so that
$E_\ell=e_\ell\H$.  Put $X_j\coloneqq\pi(q_j)$.  Since $E_\ell$ reduces
$X_j$, there is a unique quaternion $a_j^{(\ell)}$ such that
$X_je_\ell=e_\ell a_j^{(\ell)}$.
Likewise, write $X_j^* e_\ell=e_\ell b_j^{(\ell)}$.  Then
\begin{equation*}
b_j^{(\ell)}=
\langle e_\ell,X_j^* e_\ell\rangle=
\langle X_je_\ell,e_\ell\rangle=
\langle e_\ell a_j^{(\ell)},e_\ell\rangle
=
\overline{a_j^{(\ell)}}.
\end{equation*}
Hence $X_j^* e_\ell=e_\ell\overline{a_j^{(\ell)}}$.
Set $\ba^{(\ell)}
\coloneqq
(a_1^{(\ell)},\ldots,a_n^{(\ell)})\in\H^n$.
Induction on word lengths gives
\begin{equation}\label{eq:line-word-evaluation}
\pi(w)e_\ell=e_\ell w\bigl(\ba^{(\ell)},\overline{\ba^{(\ell)}}\bigr), \qquad w\in\mathcal W.
\end{equation}
Expand the cyclic vector in the orthonormal basis from
\eqref{eq:R-reducing-lines}:
\begin{equation}\label{eq:cyclic-vector-expansion}
\xi=\sum_{\ell=1}^re_\ell c_\ell,
\qquad
c_\ell\in\H.
\end{equation}
Cyclicity implies $c_\ell\neq0$ for every $\ell$. Indeed, if
$c_\ell=0$ and $P_\ell$ denotes the orthogonal projection onto $E_\ell$, then
$P_\ell\xi=0$.  Since $E_\ell$ reduces $\mathscr R$, $P_\ell$ commutes with
$\mathscr R$, and hence
\[
P_\ell\pi(w)\xi=\pi(w)P_\ell\xi=0
\qquad
\bigl(w\in\mathcal W\bigr).
\]
This contradicts \eqref{eq:flat-cyclicity}.

Let
\begin{equation}\label{eq:canonical-flat-extension}
\widetilde L_\y(p)
\coloneqq
\langle\xi,\pi(p)\xi\rangle,
\qquad
p\in\mathcal F_n.
\end{equation}
For $w\in\mathcal W_{2d}$, write $w=ab$ with
$|a|,|b|\leq d$, and set $u=a^*$ and $v=b$.  Then $w=u^*v$, and
\begin{equation*}
\langle\xi,\pi(w)\xi\rangle=
\langle\xi,\pi(u)^*\pi(v)\xi\rangle=
\langle\pi(u)\xi,\pi(v)\xi\rangle=
\langle \bv_u,\bv_v\rangle=L_\y(u^*v)=L_\y(w).
\end{equation*}
Thus \eqref{eq:canonical-flat-extension} extends the original truncated functional $L_\y$ and it is positive because $\widetilde L_\y(p^*p)=
\|\pi(p)\xi\|^2\geq0$.
Using orthogonality and \eqref{eq:line-word-evaluation}, we have
\begin{equation}
\widetilde L_\y(w)=
\langle\xi,\pi(w)\xi\rangle=
\sum_{\ell=1}^r
\overline{c_\ell}\,
w\bigl(\ba^{(\ell)},\overline{\ba^{(\ell)}}\bigr)c_\ell.
\label{eq:vector-state-line-sum}
\end{equation}
Write
\[
c_\ell=\sqrt{\lambda_\ell}\,u_\ell,
\qquad
\lambda_\ell=|c_\ell|^2>0,
\qquad
|u_\ell|=1,
\]
and define
\begin{equation}\label{eq:rotated-atoms}
q_j^{(\ell)}
\coloneqq
u_\ell^{-1}a_j^{(\ell)}u_\ell,
\qquad
\mathbf q^{(\ell)}
\coloneqq
(q_1^{(\ell)},\ldots,q_n^{(\ell)}).
\end{equation}
Consequently,
$\overline{c_\ell}\,w\bigl(\ba^{(\ell)},\overline{\ba^{(\ell)}}\bigr)c_\ell=\lambda_\ell
w\bigl(\mathbf q^{(\ell)},\overline{\mathbf q^{(\ell)}}\bigr)$.
Substitution into \eqref{eq:vector-state-line-sum} gives
$\widetilde L_\y(w)=
\sum_{\ell=1}^r\lambda_\ell
w\bigl(\mathbf q^{(\ell)},\overline{\mathbf q^{(\ell)}}\bigr)$.
For $w\in\mathcal W_{2d}$, one has $\widetilde L_\y(w)=L_\y(w)=y_w$.  Finally,
$\sum_{\ell=1}^r\lambda_\ell
=\sum_{\ell=1}^r|c_\ell|^2
=\|\xi\|^2=L_\y(1)
=1$.
Combining equal tuples yields distinct atoms, with at most $r$ atoms in total. This proves that $\mu=\sum_{\nu=1}^N\lambda_\nu\delta_{\mathbf q^{(\nu)}}$ is a positive atomic measure representing $\y$.
\end{proof}

\begin{remark}\label{d-equal-one}
If $d=1$, the flatness condition gives $\rank_{\H}\M_1(\y)=\rank_{\H}\M_{0}(\y)=1$. The sequence $\y$ already admits a representing measure by the proof of Theorem~\ref{thm:rank-one-quaternionic-moment}. Thus in this case the assumption $\pi(\mathcal I_{n}^\R)=0$ in Theorem~\ref{thm:flat-truncated-quaternionic-moment} is redundant.
\end{remark}

Although scalar consistency is an algebraic condition on the entire ideal $\mathcal I_n^\R$, flatness compresses the problem to a finite-dimensional operator algebra.  The next proposition shows that scalar consistency can therefore be checked by finitely many commutator tests.
\begin{proposition}
\label{prop:finite-scalar-consistency-test}
Let $(\mathscr H,\pi,\xi)$ be the canonical flat representation in
Theorem~\ref{thm:flat-truncated-quaternionic-moment} and $\mathscr R=\pi(\mathcal F_n)$.
Let $F_1,\ldots,F_s$ be any real vector-space basis of $\mathscr R$.  The
following conditions are equivalent:
\begin{enumerate}[(i)]
\item $\pi(\mathcal I_n^\R)=0$;
\item $\sym\mathscr R\subseteq Z(\mathscr R)$;
\item $[F_a+F_a^*,F_b]=0,\,a,b=1,\ldots,s$;
\item $[F_a+F_a^*,\pi(q_i)]=0,\,a=1,\ldots,s,i=1,\ldots,n$.
\end{enumerate}
Consequently, scalar consistency can be verified by finitely many tests in the compressed shift algebra.  Moreover,
$s\leq\dim_{\R}B_{\H}(\mathscr H)=4r^2$.
\end{proposition}

\begin{proof}
The implication (i)$\Rightarrow$(ii) is precisely the argument establishing $\sym\mathscr R\subseteq Z(\mathscr R)$ in the proof of Theorem~\ref{thm:flat-truncated-quaternionic-moment}.

Conditions (ii) and (iii) are equivalent by real bilinearity.  Indeed, if
\[
A=\sum_a\alpha_aF_a,
\qquad
B=\sum_b\beta_bF_b,
\qquad
\alpha_a,\beta_b\in\R,
\]
then
\[
[A+A^*,B]
=
\sum_{a,b}\alpha_a\beta_b[F_a+F_a^*,F_b].
\]
Thus (iii) implies that every Hermitian element of $\mathscr R$ is central,
and the converse is immediate.

It remains to prove (ii)$\Rightarrow$(i). Under (ii), the structural decomposition in the proof of Theorem~\ref{thm:flat-truncated-quaternionic-moment} yields an orthogonal decomposition of $\mathscr H$ into one-dimensional quaternionic lines reducing $\mathscr R$. On each such line, the subsequent evaluation argument identifies the representation with the scalar evaluation at a tuple
$\ba^{(\ell)}\in\H^n$:
$\pi(p)e_\ell=
e_\ell p\bigl(\ba^{(\ell)},\overline{\ba^{(\ell)}}\bigr), p\in\mathcal F_n.$
If $p\in \mathcal I_n^\R$, then
$p(\ba^{(\ell)},\overline{\ba^{(\ell)}})=0$ for every $\ell$.  Hence
$\pi(p)$ vanishes on every reducing line and therefore $\pi(p)=0$.  Thus
$\pi(\mathcal I_n^\R)=0$.

Conditions (iii) and (iv) are equivalent because the $\pi(q_i)$'s and their adjoints generate $\mathscr R$.
\end{proof}

Proposition~\ref{prop:finite-scalar-consistency-test} also yields another finite test for scalar consistency when $d\ge2$. Let $\R\langle\q,\bar\q\rangle_{2d}\subseteq\mathcal F_n$ denote the subspace of polynomials of degree at most $2d$.
\begin{corollary}\label{scalar-consistency-quartic}
Let $d\ge2$ and $(\mathscr H,\pi,\xi)$ be the canonical flat representation in Theorem~\ref{thm:flat-truncated-quaternionic-moment}. The scalar-consistency condition $\pi(\mathcal I_{n}^\R)=0$ is equivalent to
\begin{equation}\label{eq:flat-localized-identities}
\pi\left(\mathcal I_{n}^\R\cap\R\langle\q,\bar\q\rangle_{4}\right)=0.
\end{equation}
\end{corollary}
\begin{proof}
Assume $\pi\left(\mathcal I_{n}^\R\cap\R\langle\q,\bar\q\rangle_{4}\right)=0$.
Put
\[
X_i\coloneqq\pi(q_i),
\qquad
x_i\coloneqq\frac{X_i+X_i^*}{2},
\qquad
v_i\coloneqq\frac{X_i-X_i^*}{2}.
\]
Thus $x_i^*=x_i$ and $v_i^*=-v_i$.  For skew-adjoint operators $u,v$,
define
\[
u\mathbin{\boldsymbol\cdot}v
\coloneqq-\frac12(uv+vu),
\qquad
u\mathbin{\boldsymbol\times}v
\coloneqq\frac12(uv-vu),
\]
and put
\[
s_{ij}\coloneqq v_i\mathbin{\boldsymbol\cdot}v_j,
\qquad
c_{ij}\coloneqq v_i\mathbin{\boldsymbol\times}v_j,
\qquad
t_{ijk}\coloneqq v_i\mathbin{\boldsymbol\cdot}c_{jk}.
\]
We use the same symbols for the corresponding polynomials and their images
under $\pi$.
Under scalar quaternionic evaluations, $x_i,s_{ij},t_{ijk}$ are real-valued.
Consequently,
\[
[x_i,q_\ell],\qquad [s_{ij},q_\ell],\qquad [t_{ijk},q_\ell]
\]
belong to $\mathcal I_n^\R$ and have degrees at most $2$, $3$, and $4$,
respectively.  The hypothesis, followed by taking adjoints, shows that all
$x_i,s_{ij},t_{ijk}$ commute with every $X_\ell$ and $X_\ell^*$.
Therefore
\[
\mathscr B
\coloneqq
\R[x_i,s_{ij},t_{ijk}:1\leq i,j,k\leq n]
\]
consists of central Hermitian elements of
$\mathscr R=\pi(\mathcal F_n)$.
The differences between the two sides of
\eqref{eq:dot-cross-identities-a}--\eqref{eq:dot-cross-identities-d} are
scalar quaternionic identities of degree at most four, and hence those
identities hold in $\mathscr R$.  Let $\mathscr C$ be the
$\mathscr B$-module generated by the $v_i$ and $c_{ij}$.  The displayed
identities and the definitions of $s_{ij},c_{ij},t_{ijk}$ imply
\[
\mathscr C\mathbin{\boldsymbol\cdot}\mathscr C\subseteq\mathscr B,
\qquad
\mathscr C\mathbin{\boldsymbol\times}\mathscr C\subseteq\mathscr C.
\]
Since $uv=-u\mathbin{\boldsymbol\cdot}v
   +u\mathbin{\boldsymbol\times}v$
for skew-adjoint $u,v$, the space $\mathscr B+\mathscr C$ is a
subalgebra of $\mathscr R$.  It contains
$X_i=x_i+v_i$ and $X_i^*=x_i-v_i$,
so $\mathscr R=\mathscr B+\mathscr C$.
Every element of $\mathscr C$ is skew-adjoint, because its coefficients are
central and Hermitian.  If $A\in\sym\mathscr R$ and
$A=b+c$ with $b\in\mathscr B$ and $c\in\mathscr C$, then
$b+c=b-c$, whence $c=0$.  Thus
$\sym\mathscr R\subseteq\mathscr B\subseteq Z(\mathscr R)$.
Proposition~\ref{prop:finite-scalar-consistency-test} now gives
$\pi(\mathcal I_n^\R)=0$.
\end{proof}

\begin{corollary}
\label{cor:localized-quartic-scalar-consistency}
Choose $b_1,\ldots,b_r\in\mathcal W_{d-1}$ such that the corresponding columns of \(\M_d(\y)\) form a basis of its column space. Then $\pi\left(\mathcal I_{n}^\R\cap\R\langle\q,\bar\q\rangle_{4}\right)=0$ if and only if $\widetilde L_\y(b_\alpha^*pb_\beta)=0$
for every $p\in \mathcal I_{n}^\R\cap\R\langle\q,\bar\q\rangle_{4}$ and
$\alpha,\beta=1,\ldots,r$.
If, in addition, $d\geq2$ and
\[
\rank_{\H}\M_{d}(\y)=\rank_{\H}\M_{d-2}(\y),
\]
then the quartic operator condition is equivalent to
\begin{equation}\label{eq:localized-quartic-original-test}
L_\y(upv)=0
\qquad
\bigl(
p\in \mathcal I_{n}^\R\cap\R\langle\q,\bar\q\rangle_{4},\
u,v\in\mathcal W_{d-2}
\bigr).
\end{equation}
\end{corollary}

\begin{proof}
For all $\alpha,\beta$ and $p\in\mathcal I_{n}^\R\cap\R\langle\q,\bar\q\rangle_{4}$, one has $$\widetilde L_\y(b_\alpha^*pb_\beta)=
\left\langle
\pi(b_\alpha)\xi,\pi(p)\pi(b_\beta)\xi
\right\rangle.$$
The vectors $\pi(b_\alpha)\xi$ form a basis of $\mathscr H$, so all these
matrix coefficients vanish if and only if $\pi(p)=0$.
Under the additional rank equality, words of length at most $d-2$ already
span $\mathscr H$.  The same argument applies, and
$|u|+\deg(p)+|v|\leq(d-2)+4+(d-2)=2d$.
Thus $\widetilde L_\y$ agrees with $L_\y$ on every expression in
\eqref{eq:localized-quartic-original-test}.
\end{proof}

% The preceding finite test is not redundant.  Flatness and positivity alone may produce a finite-dimensional non-scalar representation, so they do not by themselves guarantee a scalar quaternionic representing measure.  The following example makes this obstruction explicit.
% \begin{example}
% Take $n=1$, $\mathscr H=\H^2$, and
% \[
% \xi=e_2,
% \qquad
% X=
% \begin{pmatrix}
% 0&1\\
% 0&0
% \end{pmatrix}.
% \]
% Let $\pi$ be the unital $*$-representation of
% $\mathcal F_1$ determined by $\pi(q)=X,\pi(q^*)=X^*$,
% and define
% \[
% L_\y(w)\coloneqq\langle\xi,\pi(w)\xi\rangle
% \qquad
% (|w|\leq4).
% \]
% Then $L_\y(1)=1$, the functional $L_\y$ is Hermitian, and $\M_2(\y)$ is the Gram matrix of the vectors $\pi(w)\xi$, $|w|\leq2$; hence $\M_2(\y)\succeq0$.  Moreover,
% \[
% \operatorname{span}_{\H}
% \{\xi,X\xi,X^*\xi\}
% =
% \operatorname{span}_{\H}\{e_2,e_1\}
% =
% \H^2.
% \]
% Thus $\rank_{\H}\M_1(\y)=2$, while all vectors defining $\M_2(\y)$ lie in
% $\H^2$, so $\rank_{\H}\M_2(\y)=2$ as well.
% The polynomial $p=q^*q-qq^*$
% belongs to $\mathcal I_1^\R$.  However,
% \[
% \pi(p)
% =X^* X-XX^*
% =
% \begin{pmatrix}
% -1&0\\
% 0&1
% \end{pmatrix},
% \qquad
% L_\y(p)=1.
% \]
% Every scalar quaternionic representing measure would annihilate $p$, contradicting $L_\y(p)=1$.  Hence no scalar representing measure exists.
% \end{example}

Finally, we solve the truncated quaternionic moment problem on the quaternionic semialgebraic set $K_{\bg,\bh}$ when $\bg$ and $\bh$ have real coefficients.
\begin{corollary}\label{extract-constrained}
Let $\y=(y_w)_{w\in\cW_{2d}}\in\H^{|\cW_{2d}|}$ be a truncated quaternionic pseudo-moment sequence with $y_1=1$. Assume that $\bg$ and $\bh$ have real coefficients. Set
\begin{equation}\label{def:delta}
\delta\coloneqq
\max\!\left\{
1,\,\left\lceil\frac{\deg(g_i)}2\right\rceil,i=1,\ldots,m,\,\left\lceil\frac{\deg(h_j)}2\right\rceil,
j=1,\ldots,\ell\right\}.
\end{equation}
Assume $d\geq\delta$,
\[
\M_d(\y)\succeq0,\quad\rank_{\H}\M_d(\y)=\rank_{\H}\M_{d-\delta}(\y)=r,\quad\M_{d-\lceil\deg(g_i)/2\rceil}(g_i\y)\succeq0,\,i=1,\dots,m,
\]
and
\[
L_{\y}(uh_jv)=0,\quad u,v\in\mathcal W,\,|u|+|v|\le 2d-\deg(h_j),\,j=1,\dots,\ell.
\]
If $d\ge2$, further assume the scalar-consistency condition
$\pi\left(\mathcal I_{n}^\R\cap\R\langle\q,\bar\q\rangle_{4}\right)=0$.
Then there exist distinct points
$\ba^{(1)},\ldots,\ba^{(N)}\in K_{\bg,\bh}$, $N\leq r$,
and weights
$\lambda_1,\ldots,\lambda_N>0$ with
$\sum_{\nu=1}^N\lambda_\nu=1$,
such that
$\mu=\sum_{\nu=1}^N\lambda_\nu\delta_{\ba^{(\nu)}}$ represents $\y$.
\end{corollary}
\begin{proof}
By Theorem \ref{thm:flat-truncated-quaternionic-moment}, Remark~\ref{d-equal-one}, and Corollary~\ref{scalar-consistency-quartic}, there exists a finitely atomic measure $\mu=\sum_{\nu=1}^N\lambda_\nu\delta_{\ba^{(\nu)}}$ representing $\y$.
It remains to prove that every atom belongs to $K_{\bg,\bh}$.

Fix an inequality constraint $g_i\ge0$ and put
$s_i=\lceil\deg(g_i)/2\rceil$.  Since $s_i\leq\delta$, the rank assumption gives $\rank_{\H}\M_{d-s_i}(\y)=r$.
Hence the vectors
$\{\pi(w)\xi:|w|\leq d-s_i\}$ span $\mathscr H$.  For
$\zeta=\sum_{|w|\leq d-s_i}\pi(w)\xi c_w$, positivity of the localizing matrix $\M_{d-s_i}(g_i\y)$ gives
\[
\langle\zeta,\pi(g_i)\zeta\rangle
= c^*\M_{d-s_i}(g_i\y)c\geq0.
\]
Thus $\pi(g_i)\succeq0$ on $\mathscr H$.  In the one-dimensional reducing-line decomposition used in the proof of Theorem~\ref{thm:flat-truncated-quaternionic-moment}, $\pi(g_i)$ acts on the line corresponding to an atom by the real scalar $g_i(\ba^{(\nu)},\overline{\ba^{(\nu)}})$.  Consequently, $g_i(\ba^{(\nu)},\overline{\ba^{(\nu)}})\geq0$, $\nu=1,\ldots,N$.

Now fix an equality constraint $h_j=0$ and put
$t_j=\lceil\deg(h_j)/2\rceil$.  Again
$\rank_{\H}\M_{d-t_j}(\y)=r$, so
$\{\pi(w)\xi:|w|\leq d-t_j\}$ spans $\mathscr H$.  If
$|u|,|v|\leq d-t_j$, then
$|u|+|v|+\deg(h_j)\leq 2d$, and so
$$\langle\pi(u)\xi,\pi(h_j)\pi(v)\xi\rangle
=L_\y(u^*h_jv)=0.$$
All matrix coefficients of $\pi(h_j)$ therefore vanish, and hence $\pi(h_j)=0$.  Evaluating on each reducing line gives
$h_j(\ba^{(\nu)},\overline{\ba^{(\nu)}})=0, \nu=1,\ldots,N$. Thus every atom lies in $K_{\bg,\bh}$.
\end{proof}

The results of this section provide the moment-theoretic foundation for quaternionic polynomial optimization.  The full Archimedean characterizations identify the infinite-dimensional limit of the moment conditions, while the flat-extension theorem explains when finite truncated data already admit atomic extraction.  We now use these results, together with the Positivstellens\"atze of Section~\ref{sec:qsos}, to construct and analyze the Moment--QSOS hierarchy.

\section{The Moment--QSOS hierarchy}\label{sec:qpop}
The Positivstellens\"atze and moment results developed in the preceding sections now lead to semidefinite hierarchies for quaternionic polynomial optimization.  We first treat problems with real coefficients, and then pass to quaternionic coefficients.

\subsection{The case of real coefficients}
Consider the following QPOP with real coefficients:
\begin{equation}\label{qpop-r}
\begin{cases}
\inf\limits_{\q\in\H^n} &f(\q,\bar\q)\\
\,\,\mathrm{s.t.}&g_i(\q,\bar\q)\ge0,\quad i=1,\ldots,m,\\
&h_j(\q,\bar\q)=0,\quad j=1,\ldots,\ell,
\end{cases}\tag{QPOP-\mbox{$\R$}}
\end{equation}
where $f,g_1,\ldots,g_m\in\sym\R\langle\q,\bar\q\rangle$, $h_1,\ldots,h_\ell\in\R\langle\q,\bar\q\rangle$.
In this setting, the moment variables can be taken real, and positivity is expressed through the real moment and localizing matrices.
The order-$d$ moment relaxation of \eqref{qpop-r} is given by
\begin{equation}\label{rmoment-relax}
f_d\coloneqq\begin{cases}
\inf\limits_{\y\in\R^{|\cW_{2d}|}} & L_{\y}(f)\\
\,\quad\mathrm{s.t.} & y_1=1,\\
&y_w=y_{w^*},\quad w\in\cW_{2d},\\
&y_u=y_{v},\quad u\stackrel{\mathrm{cyc}}{\sim}v,\, u,v\in\cW_{2d},\\
& \M_d(\y)\succeq0,\\
& \M_{d-\lceil\deg(g_i)/2\rceil}(g_i\y)\succeq0,\quad i=1,\dots,m,\\
& L_{\y}(uh_jv)=0,\quad u,v\in\mathcal W,\,|u|+|v|\le 2d-\deg(h_j),\, j=1,\dots,\ell,\\
&L_\y(u[w+w^*,q_i]v)=0,
\quad u,v,w\in\cW,|u|+|v|+|w|\le 2d-1,i=1,\ldots,n.
\end{cases}
\tag{Mom-\mbox{$\R$}}
\end{equation}

Let
$\Sigma^{\R}_{d}\coloneqq\Sigma^{\R}\cap\R\langle\q,\bar\q\rangle_{2d}$. Define
\begin{equation}\label{eq:QR-module-truncate}
\begin{split}
\mathcal Q^{\R}(\bg,\bh)_{2d}\coloneqq\bigg\{&\sigma_0+\frac12\sum_{i=1}^m(\sigma_ig_i+g_i\sigma_i)
+\sum_{j=1}^\ell\sum_{\alpha}
\left(\tau_{j\alpha}h_j\nu_{j\alpha}
+\nu_{j\alpha}^*h_j^*\tau_{j\alpha}^*\right):\sigma_0\in\Sigma^{\R}_{d},\\
&\sigma_i\in\Sigma^{\R}_{d-\lceil\deg(g_i)/2\rceil},
\tau_{j\alpha},\nu_{j\alpha}\in\mathcal F_n,\deg(\tau_{j\alpha})+\deg(\nu_{j\alpha})\le 2d-\deg(h_j)\bigg\}.
\end{split}
\end{equation}
Then the corresponding order-$d$ QSOS relaxation reads
\begin{equation}\label{qsos-r-sdp}
f_d^*\coloneqq\begin{cases}
\sup\limits_{\lambda}&\lambda\\
\,\mathrm{s.t.}&f-\lambda\in\mathcal Q^{\R}(\bg,\bh)_{2d}+(\mathcal I_{n}^\R\cap\sym\R\langle\q,\bar\q\rangle_{2d}).
\end{cases}
\tag{QSOS-\mbox{$\R$}}
\end{equation}

\begin{theorem}\label{thm:real-moment-qsos-convergence}
Let $f,g_1,\ldots,g_m\in\sym\R\langle\q,\bar\q\rangle$, 
$h_1,\ldots,h_\ell\in\R\langle\q,\bar\q\rangle$.
Assume that $\mathcal Q^{\R}(\bg,\bh)$ is Archimedean. Then
$f_d^*\nearrow f_{\min}$ and $f_d\nearrow f_{\min}$ as $d\to\infty$.
\end{theorem}
\begin{proof}
We first record the weak-duality inequality.  Let $\y$ be feasible for \eqref{rmoment-relax}, and let $\lambda$ be feasible for \eqref{qsos-r-sdp}.  Applying $L_\y$ to a certificate for $f-\lambda$ gives a nonnegative value.  Positivity of $\M_d(\y)$ gives $L_\y(\sigma_0)\geq0$.  If
$\sigma_i=\sum_s p_{is}^*p_{is}$, cyclicity and adjoint symmetry give
\[
L_\y\!\left(\frac12(\sigma_i g_i+g_i\sigma_i)\right)
=\sum_s L_\y(p_{is}g_ip_{is}^*)
=\sum_s L_\y\bigl((p_{is}^*)^*g_i p_{is}^*\bigr)\geq0,
\]
where the last inequality follows from the corresponding localizing matrix.  The equality constraints annihilate every degree-compatible two-sided multiple of the $h_j$ and therefore also its adjoint.  Finally, the last line of \eqref{rmoment-relax} annihilates the truncated scalar-identity part.  Indeed, Lemma~\ref{real-scalar-identity} gives homogeneous generators $[w+w^*,q_i]$ of $\mathcal I_n^\R$; because the ideal is graded, every element of $\mathcal I_n^\R\cap\R\langle\q,\bar\q\rangle_{2d}$ is a linear combination of degree-compatible two-sided multiples of these generators and their adjoints.  Hence
$0\leq L_\y(f-\lambda)=L_\y(f)-\lambda$,
and therefore
\begin{equation}\label{eq:real-hierarchy-weak-duality}
 f_d^*\leq f_d.
\end{equation}

If $K_{\bg,\bh}=\varnothing$, adopt the standard convention $f_{\min}=+\infty$.  For every $M\in\R$, the polynomial $f-M$ is then vacuously strictly positive on $K_{\bg,\bh}$, so Theorem~\ref{thm:putinar-real-equalities} places $f-M$ in the full certificate cone.  Hence $f_d^*\geq M$ for all sufficiently large $d$.  Thus $f_d^*\to+\infty$, and \eqref{eq:real-hierarchy-weak-duality} gives $f_d\to+\infty$ as well.  We may therefore assume from now on that $K_{\bg,\bh}\neq\varnothing$.

For every $\ba\in K_{\bg,\bh}$, the moment sequence induced by the point mass $\delta_{\ba}$ in the sense of Section~\ref{sec:full-real-valued-quaternionic-moment}, namely
$y_w=\RR(w(\ba,\bar\ba))$, is feasible for \eqref{rmoment-relax} and has objective value $f(\ba,\bar\ba)$.  Consequently,
\begin{equation}\label{eq:real-moment-upper-bound}
 f_d\leq f_{\min}.
\end{equation}
The feasible sets of the moment relaxations shrink under projection as $d$ increases, whereas the truncated QSOS cones are nested.  Thus both $(f_d)_d$ and $(f_d^*)_d$ are nondecreasing.

It remains to identify their common limit.  Since the Archimedean condition makes $K_{\bg,\bh}$ compact and nonempty, $f_{\min}$ is attained.  Fix $\varepsilon>0$.  The polynomial $f-f_{\min}+\varepsilon$
is strictly positive on $K_{\bg,\bh}$.  By Theorem~\ref{thm:putinar-real-equalities},
\[
 f-f_{\min}+\varepsilon
 \in
 \mathcal Q^{\R}(\bg,\bh)
 +\bigl(\mathcal I_n^\R\cap\sym\R\langle\q,\bar\q\rangle\bigr).
\]
Hence $f_d^*\geq f_{\min}-\varepsilon$
for all sufficiently large $d$.  Together with
\eqref{eq:real-hierarchy-weak-duality}--\eqref{eq:real-moment-upper-bound}, this gives
\[
 f_{\min}-\varepsilon\leq f_d^*\leq f_d\leq f_{\min}.
\]
Letting $d\to\infty$ and then $\varepsilon\downarrow0$ proves both asserted convergences.
\end{proof}

\begin{remark}
If $h^*=h$, then the corresponding equality constraints in \eqref{rmoment-relax} can be replaced by
\begin{equation*}
L_{\y}(wh_j)=0,\quad w\in\mathcal W,\,|w|\le 2d-\deg(h_j),\, j=1,\dots,\ell.
\end{equation*}
\end{remark}

\begin{remark}
In practice, the scalar-identity constraints in \eqref{rmoment-relax} can be imposed by augmenting the equality constraints in \eqref{qpop-r}:
\begin{equation*}
\bh'\coloneqq\bh\cup\left\{[w+w^*,q_i]:w\in\cW,\ |w|\le 2d-1,\ i=1,\ldots,n\right\}.
\end{equation*}
\end{remark}

The scalar-identity constraints are not merely formal: omitting them can change the limiting bound even for a compact quadratic problem.  The following example isolates this phenomenon.

\begin{example}
Consider the optimization problem
\begin{equation}\label{qpop-counterexample}
    \min_{\q\in\mathbb{H}^3}\ f(\q,\bar\q)
\quad \mathrm{s.t.} \quad
h_j = 1 - |q_j|^2 = 0,\quad j=1,2,3,
\end{equation}
where
\[
f = 1 + a_1 a_2 + a_2 a_3 + a_3 a_1,
\quad
a_i = \frac{q_i+\bar q_i}{2}.
\]
For any feasible point, \(|q_i|=1\) implies \(a_i\in[-1,1]\). Moreover,
\[
2f
=
(1+a_1)(1+a_2)(1+a_3)
+
(1-a_1)(1-a_2)(1-a_3)
\ge 0.
\]
Equality is attained, for example, at
\(
q_1=1,q_2=-1,q_3=1.
\)
Hence $f_{\min}=0$.

We now show that, without imposing the scalar-identity constraints, the QSOS hierarchy cannot converge to \(f_{\min}\). Let
\[
X=
\begin{pmatrix}
0&1\\
1&0
\end{pmatrix},
\qquad
Z=
\begin{pmatrix}
1&0\\
0&-1
\end{pmatrix},
\]
and define
\[
Q_1=Z,\quad
Q_2=-\frac12 Z+\frac{\sqrt3}{2}X,\quad
Q_3=-\frac12 Z-\frac{\sqrt3}{2}X.
\]
These matrices satisfy
\[
Q_i^2=I,
\qquad
\frac12\operatorname{tr}(Q_iQ_j)=-\frac12
\quad(i\neq j).
\]
Consider the \(*\)-representation
\[
q_i\mapsto Q_i,\quad \bar q_i\mapsto Q_i,
\]
and define the linear functional
\[
L(p)\coloneqq \frac12\operatorname{tr}
\bigl(p(Q_1,Q_2,Q_3)\bigr).
\]
Since $Q_j^2=I$, one has
$L(r^*h_j r)=0$ for every quaternionic polynomial $r$ and every $j=1,2,3$.
Moreover, $L(s^*s)\geq0$ for every quaternionic polynomial $s$.
Since $a_i=(q_i+\bar q_i)/2$, the representation maps $a_i$ to $Q_i$. Therefore,
\[
\begin{aligned}
L(f)
&=
1+\frac12\operatorname{tr}(Q_1Q_2)
+\frac12\operatorname{tr}(Q_2Q_3)
+\frac12\operatorname{tr}(Q_3Q_1) \\
&=
1-\frac12-\frac12-\frac12
=
-\frac12.
\end{aligned}
\]
Suppose that an order-$d$ QSOS relaxation without the scalar-identity constraints certifies a lower bound $\lambda$, that is,
\[
f-\lambda
=
\sum_{\ell} s_\ell^*s_\ell
+
\sum_{j=1}^3 \Psi_j(h_j),
\]
where $\Psi_j(h_j)$ denotes the terms generated by the equality constraint $h_j=0$. Applying $L$ gives
\[
-\frac12-\lambda
=
L(f-\lambda)
=
\sum_{\ell} L(s_\ell^*s_\ell)
\ge 0.
\]
Thus $\lambda\leq-\frac12$ at every relaxation order $d$.

After imposing the scalar-identity constraints, the same polynomial admits the following third-order QSOS representation:
\[
\begin{aligned}
f
&= \frac{1}{16}
\bigl((1+q_1)(1+q_2)(1+q_3)\bigr)^*
\bigl((1+q_1)(1+q_2)(1+q_3)\bigr) \\
&\quad
+ \frac{1}{16}
\bigl((1-q_1)(1-q_2)(1-q_3)\bigr)^*
\bigl((1-q_1)(1-q_2)(1-q_3)\bigr) \\
&\quad
+ h_1A_1+h_2A_2+h_3A_3,
\end{aligned}
\]
where
\[
\begin{aligned}
A_1
&=
\tfrac12+\tfrac12 a_2a_3-\tfrac14 h_2-\tfrac14 h_3+\tfrac18 h_2h_3,\\
A_2
&=
\tfrac12+\tfrac12 a_1a_3-\tfrac14 h_3,\\
A_3
&=
\tfrac12+\tfrac12 a_1a_2.
\end{aligned}
\]
This identity uses the relations $a_iq_j=q_ja_i$ for $i\neq j$. Hence $f$ admits a third-order QSOS certificate with lower bound $0$. Since $f_{\min}=0$, the third-order relaxation with the scalar-identity constraints is exact.
\end{example}

\subsection{The case of quaternionic coefficients}
We now allow the objective to have quaternionic coefficients while retaining real-coefficient inequality and equality constraints, precisely the setting covered by Theorem~\ref{thm:putinar-quaternionic}.  The moment variables are then quaternion-valued, and positivity is expressed through the quaternionic moment and localizing matrices.  This gives a direct quaternionic analogue of the preceding hierarchy.

We consider the following QPOP:
\begin{equation}\label{qpop-h}
\begin{cases}
\inf\limits_{\q\in\H^n} &f(\q,\bar\q)\\
\,\,\mathrm{s.t.}&g_i(\q,\bar\q)\ge0,\quad i=1,\ldots,m,\\
&h_j(\q,\bar\q)=0,\quad j=1,\ldots,\ell,
\end{cases}\tag{QPOP-\mbox{$\H$}}
\end{equation}
where $f\in\sym\P\langle\q,\bar\q\rangle$, $g_1,\ldots,g_m\in\sym\R\langle\q,\bar\q\rangle$, $h_1,\ldots,h_\ell\in\R\langle\q,\bar\q\rangle$.
The order-$d$ moment relaxation of \eqref{qpop-h} is given by
\begin{equation}\label{qmoment-relax}
f_d\coloneqq
\begin{cases}
\inf\limits_{\y\in\H^{|\cW_{2d}|}} & L^{\RR}_{\y}(f)\\
\,\quad\mathrm{s.t.} & y_1=1,\\
&y_w=\bar y_{w^*},\quad w\in\cW_{2d},\\
&\RR(y_u)=\RR(y_{v}),\quad u\stackrel{\mathrm{cyc}}{\sim}v,\, u,v\in\cW_{2d},\\
& \M_d(\y)\succeq0,\\
& \M_{d-\lceil\deg(g_i)/2\rceil}(g_i\y)\succeq0,\quad i=1,\dots,m,\\
& L_{\y}(uh_jv)=0,\quad u,v\in\mathcal W,\,|u|+|v|\le 2d-\deg(h_j),\, j=1,\dots,\ell,\\
&L_\y(u[w+w^*,q_i]v)=0,
\quad u,v,w\in\cW,|u|+|v|+|w|\le 2d-1,i=1,\ldots,n.
\end{cases}
\tag{Mom-\mbox{$\H$}}
\end{equation}

Let $\P\langle\q,\bar\q\rangle_{2d}$ denote the subspace of polynomials of degree at most $2d$ in $\P\langle\q,\bar\q\rangle$, and set
$\Sigma^{\H}_{d}\coloneqq\Sigma^{\H}\cap\P\langle\q,\bar\q\rangle_{2d}$. Define
\begin{equation}\label{eq:QH-module-truncate}
\begin{split}
\mathcal Q^{\H}(\bg,\bh)_{2d}\coloneqq\bigg\{&
\sigma_0+\frac12\sum_{i=1}^m(\sigma_i g_i+g_i\sigma_i)+\sum_{j=1}^\ell\sum_{\alpha}
\left(\tau_{j\alpha}h_j\nu_{j\alpha}^*
+\nu_{j\alpha}h_j^*\tau_{j\alpha}^*\right):\sigma_0\in\Sigma^{\H}_{d},\\
&\sigma_i\in\Sigma^{\H}_{d-\lceil\deg(g_i)/2\rceil},
\tau_{j\alpha},\nu_{j\alpha}\in\mathcal L,\deg(\tau_{j\alpha})+\deg(\nu_{j\alpha})\le 2d-\deg(h_j)
\bigg\}.
\end{split}
\end{equation}
Then the corresponding order-$d$ QSOS relaxation reads
\begin{equation}\label{qsos-q-sdp}
f_d^*\coloneqq
\begin{cases}
\sup\limits_{\lambda}&\lambda\\
\,\mathrm{s.t.}&f-\lambda\in\mathcal Q^{\H}(\bg,\bh)_{2d}+(\mathcal I_{n}^\H\cap\sym\P\langle\q,\bar\q\rangle_{2d}).
\end{cases}
\tag{QSOS-\mbox{$\H$}}
\end{equation}

\begin{theorem}\label{thm:quaternionic-moment-qsos-convergence}
Let $f\in\sym\P\langle\q,\bar\q\rangle$, $g_1,\ldots,g_m\in\sym\R\langle\q,\bar\q\rangle$, 
$h_1,\ldots,h_\ell\in\R\langle\q,\bar\q\rangle$.
Assume that $\mathcal Q^{\H}(\bg,\bh)$ is Archimedean. Then
$f_d^*\nearrow f_{\min}$ and $f_d\nearrow f_{\min}$ as $d\to\infty$.
\end{theorem}
\begin{proof}
The argument parallels the proof of Theorem~\ref{thm:real-moment-qsos-convergence}, with $L_\y^{\RR}$ in place of $L_\y$.  Let $\y$ be feasible for \eqref{qmoment-relax} and let $\lambda$ be feasible for \eqref{qsos-q-sdp}.  The degree-truncated form of \eqref{eq:full-H-localizing-identity} shows that the moment and localizing matrix constraints make $L_\y^{\RR}$ nonnegative on the QSOS and weighted QSOS parts of $\mathcal Q^{\H}(\bg,\bh)_{2d}$.  Likewise, the degree-truncated argument in \eqref{eq:full-H-equality-equivalence} shows that the equality-moment constraints annihilate the equality-ideal terms, while the last line of \eqref{qmoment-relax} annihilates the scalar-identity term.  Therefore
$0\leq L_\y^{\RR}(f-\lambda)
=L_\y^{\RR}(f)-\lambda$,
so
\begin{equation}\label{eq:quaternionic-hierarchy-weak-duality}
 f_d^*\leq f_d.
\end{equation}

If $K_{\bg,\bh}=\varnothing$, then for every $M\in\R$ the polynomial $f-M$ is vacuously strictly positive on the feasible set.  Theorem~\ref{thm:putinar-quaternionic} therefore gives a finite-degree certificate for $f-M$, so $f_d^*\geq M$ eventually.  Hence $f_d^*\to+\infty$ and, by \eqref{eq:quaternionic-hierarchy-weak-duality}, also $f_d\to+\infty$.  Thus it remains only to consider the nonempty case.

For each $\ba\in K_{\bg,\bh}$, the point-moment sequence
$y_w=w(\ba,\bar\ba)$ is feasible for \eqref{qmoment-relax} and has objective value $f(\ba,\bar\ba)$.  Hence
\begin{equation}\label{eq:quaternionic-moment-upper-bound}
 f_d\leq f_{\min}.
\end{equation}
As before, projection of a feasible order-$(d+1)$ moment sequence gives a feasible order-$d$ sequence, whereas every order-$d$ QSOS certificate remains feasible at higher orders.  Thus $(f_d)_d$ and $(f_d^*)_d$ are nondecreasing.

Fix $\varepsilon>0$.  Since $f-f_{\min}+\varepsilon$ is strictly positive on $K_{\bg,\bh}$, Theorem~\ref{thm:putinar-quaternionic} gives
\[
 f-f_{\min}+\varepsilon
 \in
 \mathcal Q^{\H}(\bg,\bh)
 +\bigl(\mathcal I_n^\H\cap\sym\P\langle\q,\bar\q\rangle\bigr).
\]
The certificate has finite degree, and hence is feasible in \eqref{qsos-q-sdp} for all sufficiently large $d$.  Consequently,
$f_d^*\geq f_{\min}-\varepsilon$ eventually.  Combining this with
\eqref{eq:quaternionic-hierarchy-weak-duality} and
\eqref{eq:quaternionic-moment-upper-bound} yields
\[
 f_{\min}-\varepsilon\leq f_d^*\leq f_d\leq f_{\min}
\]
for all sufficiently large $d$.  Letting $\varepsilon\downarrow0$ proves the result.
\end{proof}

The preceding theorem is asymptotic.  Finite convergence can be certified from a flat optimal moment matrix, provided that flatness reaches far enough to cover the degrees of the defining constraints and that the associated flat representation is scalar-consistent.  The next result makes these two requirements explicit.

\begin{theorem}\label{thm:real-flat-exactness}
Let $\y=(y_w)_{w\in\cW_{2d}}\in\H^{|\cW_{2d}|}$ be an optimal solution of \eqref{qmoment-relax}, and let $\delta$ be defined by \eqref{def:delta}. Assume $d\geq\delta$ and
\[
\rank_{\H}\M_d(\y)=\rank_{\H}\M_{d-\delta}(\y)=r.
\]
If $d\ge2$ and $\delta=1$, further assume the scalar-consistency condition
$\pi\left(\mathcal I_{n}^\R\cap\R\langle\q,\bar\q\rangle_{4}\right)=0$.
Then $f_d=f_{\min}$. Moreover, 
there exist distinct points
$\mathbf q^{(1)},\ldots,\mathbf q^{(N)}\in K_{\bg,\bh}$, $N\leq r$,
and weights
$\lambda_1,\ldots,\lambda_N>0$ with
$\sum_{\nu=1}^N\lambda_\nu=1$,
such that
$\mu=\sum_{\nu=1}^N\lambda_\nu\delta_{\mathbf q^{(\nu)}}$ represents $\y$, and every $\mathbf q^{(\nu)}$ is an optimal solution of \eqref{qpop-h}.
\end{theorem}
\begin{proof}
We first verify the scalar-consistency hypothesis required in Corollary~\ref{extract-constrained}. If $d=1$, it is unnecessary by Remark~\ref{d-equal-one}. If $d\geq2$ and $\delta=1$, it is assumed explicitly. Finally, if $d\geq2$ and $\delta\geq2$, monotonicity of the moment-matrix ranks gives
$\rank_{\H}\M_d(\y)=\rank_{\H}\M_{d-2}(\y)$.
The scalar-identity constraints in \eqref{qmoment-relax} annihilate every degree-compatible element of $\mathcal I_n^\R$; hence \eqref{eq:localized-quartic-original-test} holds. Corollaries~\ref{cor:localized-quartic-scalar-consistency} and~\ref{scalar-consistency-quartic} then yield the required scalar consistency.
By Corollary~\ref{extract-constrained}, there exists a finitely atomic measure
$\mu=\sum_{\nu=1}^N\lambda_\nu\delta_{\mathbf q^{(\nu)}}$
representing $\y$.
We have $f_d=L^{\RR}_\y(f)
 =\sum_{\nu=1}^N\lambda_\nu
 f(\mathbf q^{(\nu)},\overline{\mathbf q^{(\nu)}})
 \geq f_{\min}$.
On the other hand, \eqref{eq:quaternionic-moment-upper-bound} gives $f_d\leq f_{\min}$, so $f_d=f_{\min}$.  Because all weights are strictly positive and every summand in the displayed average is at least $f_{\min}$, equality forces
$f(\mathbf q^{(\nu)},\overline{\mathbf q^{(\nu)}})=f_{\min}$ for every $\nu$. Furthermore, as $\mathbf q^{(\nu)}\in K_{\bg,\bh}$, every $\mathbf q^{(\nu)}$ is an optimal solution of \eqref{qpop-h}.
\end{proof}

The preceding convergence theorems rely essentially on the real-coefficient hypothesis for the constraints.  When the inequalities and equalities themselves have quaternionic coefficients, the general Positivstellensatz fails in the pure-word setting, as shown in Subsection~\ref{sec:counterexample-general-putinar}.  Nevertheless, it is useful to record the corresponding moment and QSOS relaxations.  For this purpose, consider \eqref{qpop-h} with
$f,g_1,\ldots,g_m\in\sym\P\langle\q,\bar\q\rangle$ and
$h_1,\ldots,h_\ell\in\P\langle\q,\bar\q\rangle$.
Let $p=\sum_{a,b}ap_{a,b}b\in\sym\P\langle\q,\bar\q\rangle$.
The \emph{$\RR$-localizing matrix}
$\M^{\RR}_{d-\lceil\deg(p)/2\rceil}(p\y)$ associated with $p$ is defined by
\begin{equation}
    \left[\M^{\RR}_{d-\lceil\deg(p)/2\rceil}(p\y)\right]_{u,v}\coloneqq L^{\RR}_{\y}(u^*pv)=\RR\left(\sum_{a,b}p_{a,b}y_{bvu^*a}\right),
\quad u,v\in\cW_{d-\lceil\deg(p)/2\rceil}.
\end{equation}

The order-$d$ moment relaxation of \eqref{qpop-h} is given by
\begin{equation}\label{qqmoment-relax}
\begin{cases}
\inf\limits_{\y\in\H^{|\cW_{2d}|}} & L^{\RR}_{\y}(f)\\
\,\quad\mathrm{s.t.} & y_1=1,\\
&y_w=\bar y_{w^*},\quad w\in\cW_{2d},\\
&\RR(y_u)=\RR(y_{v}),\quad u\stackrel{\mathrm{cyc}}{\sim}v,\, u,v\in\cW_{2d},\\
& \M_d(\y)\succeq0,\\
& \M^{\RR}_{d-\lceil\deg(g_i)/2\rceil}(g_i\y)\succeq0,\quad i=1,\dots,m,\\
& L^{\RR}_{\y}(uh_jv)=0,\quad u,v\in\mathcal W,\,|u|+|v|\le 2d-\deg(h_j),\, j=1,\dots,\ell,\\
&L_\y(u[w+w^*,q_i]v)=0,
\quad u,v,w\in\cW,|u|+|v|+|w|\le 2d-1,i=1,\ldots,n.
\end{cases}
\tag{Mom-\mbox{$\H,\R$}}
\end{equation}

Define
\begin{equation}\label{eq:QHR-module-truncate}
\begin{split}
\mathcal Q^{\H,\R}(\bg,\bh)_{2d}\coloneqq\bigg\{&
\sigma_0+\frac12\sum_{i=1}^m(\sigma_i g_i+g_i\sigma_i)+\sum_{j=1}^\ell\sum_{\alpha}
\left(\tau_{j\alpha}h_j\nu_{j\alpha}^*
+\nu_{j\alpha}h_j^*\tau_{j\alpha}^*\right):\sigma_0\in\Sigma^{\H}_{d},\\
&\sigma_i\in\Sigma^{\R}_{d-\lceil\deg(g_i)/2\rceil},
\tau_{j\alpha},\nu_{j\alpha}\in\mathcal F_n,\deg(\tau_{j\alpha})+\deg(\nu_{j\alpha})\le 2d-\deg(h_j)
\bigg\}.
\end{split}
\end{equation}
Then the corresponding order-$d$ QSOS relaxation reads
\begin{equation}\label{qsos-qq-sdp}
\begin{cases}
\sup\limits_{\lambda}&\lambda\\
\,\mathrm{s.t.}&f-\lambda\in\mathcal Q^{\H,\R}(\bg,\bh)_{2d}+(\mathcal I_{n}^\H\cap\sym\P\langle\q,\bar\q\rangle_{2d}).
\end{cases}
\tag{QSOS-\mbox{$\H,\R$}}
\end{equation}

Before closing this section, we present a tightness result under a rank-one moment matrix.
Rank-one moment matrices are substantially more rigid than higher-rank flat matrices.  Although the scalar-consistency hypothesis in Theorem~\ref{thm:flat-truncated-quaternionic-moment} is genuinely needed in higher rank, rank one forces multiplicativity of the truncated functional directly.  As a result, the moment sequence corresponds to an evaluation at a single scalar quaternionic point, which yields an immediate tightness criterion.

\begin{theorem}
\label{thm:rank-one-quaternionic-moment}
Consider \eqref{qpop-h} and one of the following two moment relaxations:
\begin{enumerate}[(i)]
\item the relaxation \eqref{qmoment-relax}, when $g_i\in\sym\R\langle\q,\bar\q\rangle$ and $h_j\in\R\langle\q,\bar\q\rangle$;
\item the relaxation \eqref{qqmoment-relax}, when $g_i,h_j\in\sym\P\langle\q,\bar\q\rangle$.
\end{enumerate}
Let $\y$ be an optimal solution of the corresponding relaxation. If
$\rank_{\H}\M_d(\y)=1$, then the relaxation is tight.  Moreover, the point
$\q^{\star}\coloneqq (y_{q_i})_{i=1}^n$ is an optimal solution of \eqref{qpop-h} and $\delta_{\q^{\star}}$ represents the truncated sequence $\y$.
\end{theorem}

\begin{proof}
For $v\in\mathcal W_d$, let $C_v$ denote the column of $\M_d(\y)$ indexed by $v$. Since
$[\M_d(\y)]_{1,1}=L_\y(1)=1$, the column $C_1$ is nonzero. The rank-one assumption therefore implies that, for every $v\in\mathcal W_d$, there exists $c_v\in\H$ such that
$C_v=C_1c_v$. Taking the entry in the row indexed by $1$ gives
\[
L_\y(v)
=[\M_d(\y)]_{1,v}
=[\M_d(\y)]_{1,1}c_v
=c_v.
\]
Thus $C_v=C_1L_\y(v)$ for every $v\in\mathcal W_d$. Taking the entry in the row indexed by $u\in\mathcal W_d$ yields
\begin{equation}
L_\y(u^*v)
=[\M_d(\y)]_{u,v}
=[\M_d(\y)]_{u,1}L_\y(v)
=L_\y(u^*)L_\y(v)
=\overline{L_\y(u)}L_\y(v).
\label{eq:rank-one-Gram-identity}
\end{equation}
Let $a,b$ be words with $|a|,|b|\leq d$. Applying
\eqref{eq:rank-one-Gram-identity} with $u=a^*$ and $v=b$ gives
\begin{equation}\label{eq:rank-one-multiplicativity}
L_\y(ab)
=\overline{L_\y(a^*)}L_\y(b)
=L_\y(a)L_\y(b).
\end{equation}
By definition, $L_\y(q_j)=q_j^\star$ and
$L_\y(\bar q_j)=\overline{q_j^\star}$. Repeated application of
\eqref{eq:rank-one-multiplicativity} shows that
\[
L_\y(w)=w(\q^\star,\bar\q^\star)
\qquad(|w|\leq d).
\]
If $d<|w|\leq2d$, write $w=ab$ with $|a|,|b|\leq d$. Then
\eqref{eq:rank-one-multiplicativity} gives
\[
L_\y(w)
=L_\y(a)L_\y(b)
=a(\q^{\star},\bar\q^{\star})b(\q^{\star},\bar\q^{\star})
=w(\q^{\star},\bar\q^{\star}).
\]
Hence $\delta_{\q^{\star}}$ represents the truncated sequence $\y$.

In case~(i), the constraints of \eqref{qmoment-relax} give
\[
g_i(\q^{\star})=L_\y(g_i)\geq0,
\qquad
h_j(\q^{\star})=L_\y(h_j)=0.
\]
In case~(ii), symmetry of $g_i$ and $h_j$ and the constraints of
\eqref{qqmoment-relax} give instead
\[
g_i(\q^{\star})=L_\y^{\RR}(g_i)\geq0,
\qquad
h_j(\q^{\star})=L_\y^{\RR}(h_j)=0.
\]
Thus $\q^{\star}\in K_{\bg,\bh}$ in either case.  Since $f=f^*$, $f_{\min}
\leq f(\q^{\star})
=L_\y^{\RR}(f)
\leq f_{\min}$.  Hence equality holds throughout, the relaxation is tight, and $\q^{\star}$ is optimal.
\end{proof}

\begin{remark}
As in real and complex polynomial optimization, correlative sparsity~\cite{waki2006sums} and term sparsity~\cite{wang2021tssos,wang2022cs} provide natural means of reducing the size of the Moment--QSOS relaxations.
\end{remark}

\subsection{First-order relaxations of quaternionic QCQPs}\label{sec:qsos-vs-rsos}
For quaternionic QCQPs, this subsection establishes that working directly with quaternionic variables entails no loss of relaxation strength at the first order. After realifying a quaternionic QCQP, we construct an explicit positive, surjective map between the real and quaternionic lifted PSD cones and use it to show that the two first-order relaxations have the same optimal value.

Concretely, consider the quaternionic QCQP
\begin{equation}\label{eq:qpop1}
\begin{cases}
\inf\limits_{\q \in \H^n}
&
f(\q,\bar\q)=[\q]_1^*F[\q]_1
\\
\,\,\,\mathrm{s.t.}
&
g_i(\q,\bar\q)=[\q]_1^*G_i[\q]_1\ge0,
\quad i=1,\ldots,m,
\\
&
h_j(\q,\bar\q)=[\q]_1^*H_j[\q]_1=0,
\quad j=1,\ldots,\ell,
\end{cases}
\end{equation}
where $[\q]_1=[1,q_1,\dots,q_n]^\intercal$, and \(F,G_i,H_j\in\cH^{n+1}\).

\begin{theorem}
\label{prop:first-order-qcqp-realification}
The first-order quaternionic moment relaxation of \eqref{eq:qpop1},
constructed using the reduced monomial vector \( [\q]_1 \), has the same
optimal value as the first-order real moment relaxation of the QCQP obtained by writing
\[
\q=\x_1+\x_2\i+\x_3\j+\x_4\k,
\qquad
\x_s\in\R^n.
\]
\end{theorem}
\begin{proof}
For $\bv=
\begin{bmatrix}
t&\x_1^{\intercal}&\x_2^{\intercal}&
 \x_3^{\intercal}&\x_4^{\intercal}
\end{bmatrix}^{\intercal}
\in\R^{4n+1}$,
define
\begin{equation}\label{eq:tau-reduced-realification}
\tau(\bv)
\coloneqq
\begin{bmatrix}
 t\\
 \x_1+\x_2\i+\x_3\j+\x_4\k
\end{bmatrix}
\in\H^{n+1}.
\end{equation}
Let $P:\R^{4n+1}\longrightarrow\R^{4(n+1)}$
be the coordinate-insertion map determined by
\begin{equation}\label{eq:P-coordinate-insertion}
P\bv
= [t\quad \x_1^{\intercal}\quad 0\quad \x_2^{\intercal}\quad 0\quad \x_3^{\intercal}\quad 0\quad \x_4^{\intercal}]^{\intercal}.
\end{equation}
For \(A\in\cH^{n+1}\), define the compressed realification
\begin{equation}\label{eq:Gamma-compressed-realification}
\Gamma(A)
\coloneqq
P^{\intercal}\Lambda(A)P
\in\mathcal S^{4n+1},
\end{equation}
where \(\Lambda\) is the real representation from Lemma~\ref{lm1}.
By the defining property of \(\Lambda\),
\begin{equation}\label{eq:quadratic-form-compressed-realification}
\tau(\bv)^*A\tau(\bv)
=\bv^{\intercal}\Gamma(A)\bv
\qquad
\bigl(A\in\cH^{n+1},\ \bv\in\R^{4n+1}\bigr).
\end{equation}
In particular, for
\[
\widehat\x
\coloneqq
\begin{bmatrix}
1&\x_1^{\intercal}&\x_2^{\intercal}&
\x_3^{\intercal}&\x_4^{\intercal}
\end{bmatrix}^{\intercal},
\]
the real QCQP equivalent to \eqref{eq:qpop1} is
\begin{equation}\label{eq:realified-qcqp}
\begin{cases}
\inf\limits_{\widehat\x\in\R^{4n+1}} & \widehat\x^{\intercal}\Gamma(F)\widehat\x,\\
\quad\mathrm{s.t.}
&\widehat\x^{\intercal}\Gamma(G_i)\widehat\x\geq0,
\quad i=1,\ldots,m,\\
&\widehat\x^{\intercal}\Gamma(H_j)\widehat\x=0,
\quad j=1,\ldots,\ell.
\end{cases}
\end{equation}

On \(\cH^{n+1}\), use the real inner product
\begin{equation}\label{eq:real-Hermitian-pairing}
\langle A,Z\rangle_{\R}
\coloneqq
\RR\!\left(\Tr(A^*Z)\right).
\end{equation}
For every \(A\in\cH^{n+1}\) and \(\hat\q\in\H^{n+1}\), cyclic invariance of
the real part gives
\begin{equation}\label{eq:rank-one-quaternionic-lifting-pairing}
\langle A,\hat\q \hat\q^*\rangle_{\R}=\hat\q^*A\hat\q.
\end{equation}
Since all coefficient matrices in \eqref{eq:qpop1} are Hermitian, the
first-order quaternionic moment relaxation is
\begin{equation}\label{eq:first-order-quaternionic-Shor}
\rho_{\H}^{(1)}
\coloneqq
\begin{cases}
\inf\limits_{Z\in\cH^{n+1}}
&\langle F,Z\rangle_{\R}\\
\quad\mathrm{s.t.}
&\langle G_i,Z\rangle_{\R}\geq0,
\quad i=1,\ldots,m,\\
&\langle H_j,Z\rangle_{\R}=0,
\quad j=1,\ldots,\ell,\\
&Z_{00}=1,\qquad Z\succeq0.
\end{cases}
\end{equation}
Indeed, before relaxation the lifted matrix is
\(Z=[\q]_1[\q]_1^*\), and \eqref{eq:first-order-quaternionic-Shor}
results from dropping its rank-one constraint.  Similarly, the first-order real moment relaxation of \eqref{eq:realified-qcqp} is
\begin{equation}\label{eq:first-order-real-Shor}
\rho_{\R}^{(1)}
\coloneqq
\begin{cases}
\inf\limits_{Y\in\mathcal S^{4n+1}}
&\langle\Gamma(F),Y\rangle\\
\quad\mathrm{s.t.}
&\langle\Gamma(G_i),Y\rangle\geq0,
\quad i=1,\ldots,m,\\
&\langle\Gamma(H_j),Y\rangle=0,
\quad j=1,\ldots,\ell,\\
&Y_{00}=1,\qquad Y\succeq0,
\end{cases}
\end{equation}
where the pairing on real matrices is the Frobenius inner product.

Let
\[
\mathcal E:\mathcal S^{4n+1}\longrightarrow\cH^{n+1}
\]
be the real adjoint of \(\Gamma\), that is,
\begin{equation}\label{eq:E-adjoint-definition}
\langle A,\mathcal E(Y)\rangle_{\R}
=
\langle\Gamma(A),Y\rangle
\qquad
\bigl(A\in\cH^{n+1},\ Y\in\mathcal S^{4n+1}\bigr).
\end{equation}
We claim that
\begin{equation}\label{eq:E-rank-one-formula}
\mathcal E(\bv\bv^{\intercal})
=
\tau(\bv)\tau(\bv)^*
\qquad
\bigl(\bv\in\R^{4n+1}\bigr).
\end{equation}
Indeed, for every \(A\in\cH^{n+1}\),
\begin{equation*}
\langle A,\mathcal E(\bv\bv^{\intercal})\rangle_{\R}
=
\langle\Gamma(A),\bv\bv^{\intercal}\rangle=
\bv^{\intercal}\Gamma(A)\bv=
\tau(\bv)^*A\tau(\bv)=
\langle A,\tau(\bv)\tau(\bv)^*\rangle_{\R},
\end{equation*}
where the third equality is \eqref{eq:quadratic-form-compressed-realification}.
The nondegeneracy of \eqref{eq:real-Hermitian-pairing} proves
\eqref{eq:E-rank-one-formula}.

Equation~\eqref{eq:E-rank-one-formula} immediately implies that
\begin{equation}\label{eq:E-positive-direction}
Y\succeq0
\quad\Longrightarrow\quad
\mathcal E(Y)\succeq0.
\end{equation}
In fact, $\mathcal E$ maps the real PSD cone onto the quaternionic PSD cone.  To prove surjectivity, let
\(Z\succeq0\).  Choose a quaternionic Gram decomposition
\begin{equation}\label{eq:quaternionic-Gram-decomposition-Shor}
Z=\sum_{s=1}^N \hat\q_s\hat\q_s^*,
\qquad
\hat\q_s\in\H^{n+1}.
\end{equation}
For every \(s\), choose a unit quaternion \(u_s\) such that the zeroth
coordinate of \(\hat\q_su_s\) is real: if \((\hat\q_s)_0\neq0\), take
\[
u_s=\frac{\overline{(\hat\q_s)_0}}{|(\hat\q_s)_0|},
\]
and otherwise take \(u_s=1\).  There is then a vector
\(\bv_s\in\R^{4n+1}\) satisfying
$\tau(\bv_s)=\hat\q_su_s$.
Since \(|u_s|=1\),
$\tau(\bv_s)\tau(\bv_s)^*
=\hat\q_su_s\overline{u_s}\hat\q_s^*=\hat\q_s\hat\q_s^*$.
Consequently, with
$Y\coloneqq\sum_{s=1}^N\bv_s\bv_s^{\intercal}\succeq0$,
formula \eqref{eq:E-rank-one-formula} gives
$\mathcal E(Y)=Z$.
Thus
\begin{equation}\label{eq:E-cone-surjectivity}
\mathcal E\bigl(\{Y\in\mathcal S^{4n+1}:Y\succeq0\}\bigr)
=
\{Z\in\cH^{n+1}:Z\succeq0\}.
\end{equation}
The map also preserves the normalization constraint.  Let \(E_{00}\in\cH^{n+1}\) be the
matrix with a single \(1\) in its \((0,0)\) entry, and let
\(\be_0\in\R^{4n+1}\) be the first coordinate vector.  By construction,
$\Gamma(E_{00})=\be_0\be_0^{\intercal}$.
Hence, for every \(Y\in\mathcal S^{4n+1}\),
\begin{equation}\label{eq:E-preserves-normalization}
[\mathcal E(Y)]_{00}
=
\langle E_{00},\mathcal E(Y)\rangle_{\R}
=
\langle \be_0\be_0^{\intercal},Y\rangle
=
Y_{00}.
\end{equation}

Now let \(Y\) be feasible for \eqref{eq:first-order-real-Shor} and set
\(Z=\mathcal E(Y)\).  Equations
\eqref{eq:E-adjoint-definition}, \eqref{eq:E-positive-direction}, and
\eqref{eq:E-preserves-normalization} show that \(Z\) is feasible for
\eqref{eq:first-order-quaternionic-Shor}, and all constraint values and the
objective value are preserved.  Therefore
$\rho_{\H}^{(1)}\leq\rho_{\R}^{(1)}$.
Conversely, let \(Z\) be feasible for
\eqref{eq:first-order-quaternionic-Shor}.  By
\eqref{eq:E-cone-surjectivity}, there exists \(Y\succeq0\) such that
\(\mathcal E(Y)=Z\).  Equation \eqref{eq:E-preserves-normalization} gives
\(Y_{00}=Z_{00}=1\), while \eqref{eq:E-adjoint-definition} shows that
\(Y\) satisfies every realified constraint and has the same objective value.
Thus $\rho_{\R}^{(1)}\leq\rho_{\H}^{(1)}$.
We conclude that
$\rho_{\H}^{(1)}=\rho_{\R}^{(1)}$.
\end{proof}

\section{Numerical experiments}\label{sec:numerical-experiments}
In this section, we evaluate the proposed Moment--QSOS framework on randomly generated problems and application-motivated instances. For each problem, we compare the QSOS relaxation with the real SOS relaxation (denoted by RSOS) obtained from an equivalent real reformulation. Both QSOS and RSOS relaxations are generated using the Julia package \texttt{TSSOS}\footnote{{\tt TSSOS} is freely available at \href{https://github.com/wangjie212/TSSOS}{https://github.com/wangjie212/TSSOS}.}~\cite{magron2021tssos}.

All computations were performed on a laptop equipped with an Intel Core i7 processor and 16~GB of RAM\@. The resulting SDPs were solved with \texttt{MOSEK 11.0}. In the tables below, ``opt'' denotes the reported SDP optimal value, and ``time'' denotes the runtime in seconds. Moreover, ``-'' means that \texttt{MOSEK} runs out of memory.

For brevity, define the standard constraint sets
\[
\mathbf{B}_n
\coloneqq
\left\{\q\in\H^n:\sum_{i=1}^n |q_i|^2\leq1\right\},
\qquad
\mathbf{K}_n
\coloneqq
\left\{\q\in\H^n:|q_i|^2=1,\ i=1,\ldots,n\right\}.
\]

\subsection{Minimizing quadratic quaternionic polynomials}
We first consider two classes of quaternionic QCQPs. The first uses the quadratic form studied in Section~\ref{sec:qsos-vs-rsos}:
\begin{equation}\label{ne1}
\min\limits_{\q\in\mathbf{B}_n\text{ or }\mathbf{K}_n}[\q]_1^*Q[\q]_1,
\end{equation}
where $Q\in\mathcal{S}^{n+1}$ with entries drawn from a uniform distribution on \([0,1]\).
For each $n\in\{20,40,60\}$, we generate three random instances and solve the corresponding first-order relaxations. The results are reported in Table~\ref{table1}.

As the table shows, the two methods return identical optimal values at the displayed precision for every tested instance, consistent with the equivalence established in Section~\ref{sec:qsos-vs-rsos}. QSOS also has substantially shorter runtimes, and the difference increases with the problem size. The observed speedup ranges from approximately two orders of magnitude for $n=20$ to several thousandfold for $n=60$.

\begin{table}[htbp!]
\centering
\caption{Results for Problem~\eqref{ne1} when $Q$ is real.}
\label{table1}
\begin{tabular}{ccccccc}
\toprule
\multirow{2}{*}{$n$} & \multirow{2}{*}{constraints} & \multirow{2}{*}{trial} & \multicolumn{2}{c}{QSOS} & \multicolumn{2}{c}{RSOS} \\
\cmidrule(lr){4-5} \cmidrule(lr){6-7}
 & & & opt & time & opt & time \\
\midrule

\multirow{6}{*}{20} 
& \multirow{3}{*}{$\mathbf{B}_n$} 
& 1 & -3.82219 & 0.03 & -3.82219 & 3.65 \\
& & 2 & -2.54142 & 0.03 & -2.54142 & 3.71\\
& & 3 & -3.26249 & 0.03 & -3.26249 & 4.03 \\
\cmidrule(lr){2-7}

& \multirow{3}{*}{$\mathbf{K}_n$}
& 1 & -41.3916 & 0.03 & -41.3916 & 4.47 \\
& & 2 & -37.9063& 0.03& -37.9063 & 3.19 \\
& & 3 & -41.1027 & 0.04 & -41.1027 & 5.65 \\
\midrule

% 40
\multirow{6}{*}{40} 
& \multirow{3}{*}{$\mathbf{B}_n$} 
& 1 & -5.47739 & 0.17& -5.47739 & 168 \\
& & 2 & -4.48873 &  0.19 & -4.48873 & 158\\
& & 3 & -5.02283 & 0.17 & -5.02283 & 187 \\
\cmidrule(lr){2-7}

& \multirow{3}{*}{$\mathbf{K}_n$}
& 1 & -121.101 & 0.18 & -121.101 & 172 \\
& & 2 & -114.615& 0.19 & -114.615 & 221 \\
& & 3 & -125.411 & 0.17 & -125.411 & 154\\
\midrule
% 60
\multirow{6}{*}{60} 
& \multirow{3}{*}{$\mathbf{B}_n$} 
& 1 & -6.49712 & 0.92 & -6.49712 & 2788 \\
& & 2 & -5.34838 & 0.94& -5.34838 & 3299\\
& & 3 & -5.88301 & 0.96 & -5.88301 & 2893 \\
\cmidrule(lr){2-7}

& \multirow{3}{*}{$\mathbf{K}_n$}
& 1 & -228.302 & 0.96 & -228.302 & 3846 \\
& & 2 & -214.703& 0.99 & -214.703 & 3252 \\
& & 3 & -229.074 & 0.98  & -229.074 & 3778\\

\bottomrule
\end{tabular}
\end{table}
\FloatBarrier

We next let $Q$ be a random Hermitian quaternionic matrix. For each $n\in\{20,30,40\}$, we generate three random instances and solve the corresponding first-order relaxations. Table~\ref{table1q} shows that the two methods again return identical optimal values at the displayed precision, while QSOS has shorter runtimes.

\begin{table}[htbp!]
\centering
\caption{Results for Problem~\eqref{ne1} when $Q$ is quaternionic.}
\label{table1q}
\begin{tabular}{ccccccc}
\toprule
\multirow{2}{*}{$n$} & \multirow{2}{*}{constraints} & \multirow{2}{*}{trial} & \multicolumn{2}{c}{QSOS} & \multicolumn{2}{c}{RSOS} \\
\cmidrule(lr){4-5} \cmidrule(lr){6-7}
 & & & opt & time & opt & time \\
\midrule
\multirow{6}{*}{20} 
& \multirow{3}{*}{$\mathbf{B}_n$} 
& 1 & -5.09712 & 0.19& -5.09712 & 2.55 \\
& & 2 & -5.35170 &  0.23 & -5.35170 & 2.92\\
& & 3 & -4.97638 & 0.18 & -4.97638 & 2.71 \\
\cmidrule(lr){2-7}

& \multirow{3}{*}{$\mathbf{K}_n$}
& 1 & -62.0785 & 0.21 & -62.0785 & 3.08 \\
& & 2 & -64.0942& 0.22 & -64.0942 & 2.41 \\
& & 3 & -58.3504 & 0.24 & -58.3504 & 3.62\\
\midrule
% 30
\multirow{6}{*}{30} 
& \multirow{3}{*}{$\mathbf{B}_n$} 
& 1 & -6.39956 & 1.16 &-6.39956 & 36.1 \\
& & 2 & -6.90228 & 1.24& -6.90228 & 35.9\\
& & 3 & -6.23483 & 1.07 & -6.23483 & 30.4 \\
\cmidrule(lr){2-7}

& \multirow{3}{*}{$\mathbf{K}_n$}
& 1 &  -115.696 & 1.27 &  -115.696 & 40.0 \\
& & 2 & -116.395& 1.16 & -116.395 & 32.7 \\
& & 3 & -119.159 & 1.28  & -119.159 & 36.6\\
\midrule
% 40
\multirow{6}{*}{40} 
& \multirow{3}{*}{$\mathbf{B}_n$} 
& 1 & -7.37614 & 2.78 &-7.37614 & 154\\
& & 2 &  -7.53049 & 3.08& -7.53049 & 167\\
& & 3 & -7.23953 & 2.64 & -7.23953 & 162\\
\cmidrule(lr){2-7}

& \multirow{3}{*}{$\mathbf{K}_n$}
& 1 & -178.980 & 3.05 &  -178.980 & 178 \\
& & 2 & -179.488& 3.34 & -179.488 & 159 \\
& & 3 & -188.870 & 2.97  & -188.870 & 168\\

\bottomrule
\end{tabular}
\end{table}
\FloatBarrier

The second class consists of quaternionic QCQPs with a more general quadratic objective:
\begin{equation}\label{ne2}
\min\limits_{\q\in\mathbf{B}_n\text{ or }\mathbf{K}_n}\cW_1^*Q\cW_1,
\end{equation}
where $Q$ is a random symmetric real matrix with entries drawn from a uniform distribution on \([0,1]\).
For each $n\in\{20,40,60\}$, we generate three random instances and solve the corresponding first-order relaxations. The results are reported in Table~\ref{table2}.

The table shows that QSOS and RSOS return identical optimal values at the displayed precision for every instance, whereas QSOS has substantially shorter runtimes. The observed difference increases with the problem size.

\begin{table}[htbp!]
\centering
\caption{Results for Problem~\eqref{ne2}.}
\label{table2}
\begin{tabular}{ccccccc}
\toprule
\multirow{2}{*}{$n$} & \multirow{2}{*}{constraints} & \multirow{2}{*}{trial} & \multicolumn{2}{c}{QSOS} & \multicolumn{2}{c}{RSOS} \\
\cmidrule(lr){4-5} \cmidrule(lr){6-7}
 & & & opt & time & opt & time \\
\midrule

\multirow{6}{*}{20} 
& \multirow{3}{*}{$\mathbf{B}_n$} 
& 1 & -7.07036 & 0.25 & -7.07036 & 3.41 \\
& & 2 & -5.41837 & 0.24 & -5.41837 & 3.96\\
& & 3 & -5.94149 & 0.22 & -5.94149 & 3.50 \\
\cmidrule(lr){2-7}

& \multirow{3}{*}{$\mathbf{K}_n$}
& 1 & -86.1869 & 0.29 & -86.1869 & 5.13\\
& & 2 & -79.5948 & 0.31 & -79.5948 & 4.42 \\
& & 3 & -86.9110 & 0.27 & -86.9110 & 4.47 \\
\midrule

% 40
\multirow{6}{*}{40} 
& \multirow{3}{*}{$\mathbf{B}_n$} 
& 1 & -9.76842 & 1.49& -9.76842 & 65.9 \\
& & 2 & -7.71524 &  1.54 & -7.71524 & 60.2\\
& & 3 & -8.25736 & 1.51 & -8.25736 & 67.6 \\
\cmidrule(lr){2-7}

& \multirow{3}{*}{$\mathbf{K}_n$}
& 1 & -239.239 & 1.76 & -239.239 & 69.3 \\
& & 2 & -244.122 & 1.75& -244.122 & 64.6 \\
& & 3 & -253.876 & 1.80 &-253.876 &  71.2\\
\midrule
% 60
\multirow{6}{*}{60} 
& \multirow{3}{*}{$\mathbf{B}_n$} 
& 1 & -12.5970 & 8.85 & -12.5970 &  5055 \\
& & 2 & -10.6487 & 7.22 & -10.6487 & 4698\\
& & 3 & -9.61947 & 9.58 & -9.61947 & 4729 \\
\cmidrule(lr){2-7}

& \multirow{3}{*}{$\mathbf{K}_n$}
& 1 & -484.756 & 9.43 & -484.756 & 6135 \\
& & 2 & -488.419 & 8.10 & -488.419 & 3291\\
& & 3 &  -461.723 & 10.1 & -461.723 & 3799\\

\bottomrule
\end{tabular}
\end{table}
\FloatBarrier

\subsection{Minimizing quartic quaternionic polynomials}
We next minimize a random quartic quaternionic polynomial over $\mathbf{B}_n$ or \(\mathbf{K}_n\):
\begin{equation}\label{ne3}
\min\limits_{\q\in\mathbf{B}_n\text{ or }\mathbf{K}_n}[\q]_2^*Q[\q]_2,
\end{equation}
where $[\q]_2$ denotes the vector of quaternionic words generated by $\q$ of length at most $2$, and $Q$ is a random symmetric real matrix with entries drawn from a uniform distribution on \([0,1]\).
For each $n\in\{2,4,6\}$, we generate three random instances and solve the corresponding second-order relaxations. The results are reported in Table~\ref{table3}.

The table shows that QSOS and RSOS return identical optimal values at the displayed precision for all tested instances, while QSOS has substantially shorter runtimes. For $n=6$, for example, the QSOS runtimes are approximately two orders of magnitude smaller than the RSOS runtimes.

\begin{table}[htbp!]
\centering
\caption{Results for Problem~\eqref{ne3}.}
\label{table3}
\begin{tabular}{ccccccc}
\toprule
\multirow{2}{*}{$n$} & \multirow{2}{*}{constraints} & \multirow{2}{*}{trial} & \multicolumn{2}{c}{QSOS} & \multicolumn{2}{c}{RSOS} \\
\cmidrule(lr){4-5} \cmidrule(lr){6-7} 
 & & & opt & time & opt & time \\
\midrule
% 2
\multirow{6}{*}{2} 
& \multirow{3}{*}{$\mathbf{B}_n$} 
& 1 & -2.31395 & 0.02 & -2.31395 & 0.05 \\
& & 2 & -2.85354 & 0.03 & -2.85354 & 0.06\\
& & 3 & -1.45489 & 0.02 & -1.45489 & 0.06 \\
\cmidrule(lr){2-7}

& \multirow{3}{*}{$\mathbf{K}_n$}
& 1 & -1.75350 & 0.02 & -1.75350 & 0.06 \\
& & 2 & -6.86193 & 0.02  & -6.86193 & 0.07 \\
& & 3 & -4.71116 & 0.02  & -4.71116 & 0.06\\
\midrule

\multirow{6}{*}{4} 
& \multirow{3}{*}{$\mathbf{B}_n$} 
& 1 & -4.63091 & 0.42 & -4.63091 & 7.05\\
& & 2 & -3.20566 & 0.39 & -3.20566 & 6.61\\
& & 3 &  -4.61548 & 0.34  & -4.61548 & 6.11 \\
\cmidrule(lr){2-7}

& \multirow{3}{*}{$\mathbf{K}_n$}
& 1 &  -29.2331 & 0.45 &  -29.2331 & 8.61 \\
& & 2 & -24.3238 & 0.42 &  -24.3238 & 7.85\\
& & 3 & -23.5111 & 0.39 & -23.5111 & 6.58\\
\midrule

%6
\multirow{6}{*}{6} 
& \multirow{3}{*}{$\mathbf{B}_n$} 
& 1 & -6.29210 & 3.42 & -6.29210 & 323 \\
& & 2 & -6.10003 & 3.98 & -6.10003 & 356\\
& & 3 & -3.88157 & 3.24 & -3.88157 & 341 \\
\cmidrule(lr){2-7}

& \multirow{3}{*}{$\mathbf{K}_n$}
& 1 & -70.7152 & 3.76 & -70.7152 & 344 \\
& & 2 & -74.0446 & 2.89 & -74.0446 & 318 \\
& & 3 & -68.5696 & 2.81 & -68.5696 & 371\\
\bottomrule
\end{tabular}
\end{table}
\FloatBarrier

\subsection{Minimizing sparse quaternionic polynomials}
In this subsection, we test two classes of sparse problems.
The first class has the form
\begin{equation}\label{ne5}
\begin{cases}
\inf\limits_{\q\in\H^n} & \sum_{i=1}^k[\q_i]_1^*Q_i[\q_i]_1\\
\,\,\,\mathrm{s.t.}&\|\q_i\|_2^2=1,\quad i=1,\ldots,k,
\end{cases}
\end{equation}
where $\q_i=\{q_{3i-2},q_{3i-1},q_{3i},q_{3i+1}\}$, and $Q_i$ is a random real symmetric matrix with entries drawn from a uniform distribution on \([0,1]\), $i=1,\ldots,k$.
For each $n\in\{151,301,451\}$, we generate three random instances and solve the corresponding sparse first-order relaxations. The results are reported in Table~\ref{table5}.

The second class has the form
\begin{equation}\label{ne6}
\begin{cases}
\inf\limits_{\q\in\H^n} & \sum_{i=1}^k[\q_i]_2^*Q_i[\q_i]_2\\
\,\,\,\mathrm{s.t.}&\|\q_i\|_2^2=1,\quad i=1,\ldots,k,
\end{cases}
\end{equation}
where $\q_i=\{q_{3i-2},q_{3i-1},q_{3i},q_{3i+1}\}$, and $Q_i$ is a random real symmetric matrix with entries drawn from a uniform distribution on \([0,1]\), $i=1,\ldots,k$.
For each $n\in\{61,91,121\}$, we generate three random instances and solve the corresponding sparse second-order relaxations. The results are reported in Table~\ref{table6}.

Tables~\ref{table5} and~\ref{table6} show that QSOS and RSOS return identical optimal values at the displayed precision, while QSOS has substantially shorter runtimes.

\begin{table}[htbp!]
\centering
\caption{Results for Problem~\eqref{ne5}.}
\label{table5}
\begin{tabular}{cccccc}
\toprule
\multirow{2}{*}{$n$} &\multirow{2}{*}{trial} & \multicolumn{2}{c}{QSOS} & \multicolumn{2}{c}{RSOS} \\
\cmidrule(lr){3-4} \cmidrule(lr){5-6}
 & &opt & time & opt & time \\
\midrule

% 151
\multirow{3}{*}{151}
& 1 &  -41.2313 & 0.14 & -41.2313 & 0.44 \\
& 2 & -47.7061 & 0.13 & -47.7061 & 0.42 \\
& 3 & -47.1512 & 0.12 & -47.1512 & 0.51\\
\midrule

% 301
\multirow{3}{*}{301} 
& 1 & -82.0192 & 0.33  & -82.0192 & 0.95\\
& 2 & -93.0616 & 0.25 & -93.0616 & 1.01 \\
& 3 & -93.9201 & 0.24  & -93.9201 & 0.98\\
\midrule
% 451
\multirow{3}{*}{451} 
& 1 & -127.617 & 0.46  & -127.617 & 1.87\\
& 2 & -142.243 & 0.40 & -142.243 & 1.79 \\
& 3 & -142.823 & 0.41 & -142.823 & 1.85\\
\bottomrule
\end{tabular}
\end{table}

%results for test 6.5
\begin{table}[htbp!]
\centering
\caption{Results for Problem~\eqref{ne6}.}
\label{table6}
\begin{tabular}{cccccc}
\toprule
\multirow{2}{*}{$n$} & \multirow{2}{*}{trial} & \multicolumn{2}{c}{QSOS} & \multicolumn{2}{c}{RSOS} \\
\cmidrule(lr){3-4} \cmidrule(lr){5-6} 
 & & opt & time & opt & time \\
\midrule

% 31
\multirow{3}{*}{31}
& 1 & -23.3275 & 4.11 & -23.3275 & 112 \\
& 2 & -21.8058 & 4.50 & -21.8058 & 104 \\
& 3 & -20.5047 & 3.78  & -20.5047 & 102 \\
\midrule

% 61
\multirow{3}{*}{61} 
& 1 & -43.2442 & 7.06  & - & -  \\
& 2 & -44.8020 & 7.52 & - & -  \\
& 3 & -44.0988 & 8.01  & - & - \\
\midrule

% 91
\multirow{3}{*}{91} 
& 1 & -62.7327 & 12.45 & - & - \\
& 2 & -65.4687 & 12.43 & - & - \\
& 3 &  -64.3736 & 11.59 & - & -\\

\bottomrule
\end{tabular}
\end{table}
\FloatBarrier

\section{Two applications}\label{sec:app}
We illustrate the Moment--QSOS framework through two applications: quaternionic feature extraction for image recognition and quaternionic rotation synchronization.

\subsection{Application to the quaternion-based maximum margin criterion}
The maximum margin criterion (MMC) is a widely used feature-extraction method for image recognition. The quaternion-based MMC (QMMC)~\cite{liu2017quaternion} extends this formulation to quaternion-valued features by representing color images as quaternionic vectors.

Suppose that the training set contains $C$ classes, with class $c$ consisting of $N_c$ samples indexed by $\mathcal I_c$. Each color image is represented by a quaternionic feature vector $\x^{(i)}\in\H^n$, where $n$ is the feature dimension. Let $\boldsymbol\mu_c$ denote the mean of class $c$, and let $\boldsymbol\mu$ denote the global sample mean.
QMMC seeks a projection vector $\v\in\H^n$ that maximizes between-class scatter while penalizing within-class scatter:
\begin{equation}\label{mmc}
\begin{cases}
\sup\limits_{\v\in\H^n} & J(\v) = \v^* S_B \v - \gamma \v^* S_W \v\\
\,\,\,\mathrm{s.t.}&\|\v\| = 1,
\end{cases}
\end{equation}
where $\gamma > 0$ is a chosen trade-off parameter and
\(
S_B = \sum_{c=1}^{C} N_c (\boldsymbol{\mu}_c - \boldsymbol{\mu})(\boldsymbol{\mu}_c - \boldsymbol{\mu})^*, \,
S_W = \sum_{c=1}^{C} \sum_{i\in\mathcal{I}_c} (\x^{(i)} - \boldsymbol{\mu}_c)(\x^{(i)} - \boldsymbol{\mu}_c)^*.
\)
Define
$Q\coloneqq-(S_B-\gamma S_W)\in\cH^{n}$. Then \eqref{mmc} is equivalent to
\begin{equation}\label{mmcinf}
\begin{cases}
\inf\limits_{\v\in\H^n} & \v^*Q\v\\
\,\,\,\mathrm{s.t.}&\|\v\| = 1.\\
\end{cases}
\end{equation}

In the numerical experiments, we take $C=2$ and $N_1=N_2=5$. Each sample is represented by a randomly generated quaternionic feature vector. For each feature dimension $n\in\{20,30,40\}$, we generate three random instances and solve the first-order relaxations. The results are reported in Table~\ref{tablea1}.
The table shows that QSOS and RSOS return identical optimal values at the displayed precision for all instances, while QSOS consistently has shorter runtimes.

 \begin{table}[htbp!]
\centering
\caption{Results for Problem~\eqref{mmcinf}.}
\label{tablea1}
\begin{tabular}{cccccc}
\toprule
\multirow{2}{*}{$n$} & \multirow{2}{*}{trial} & \multicolumn{2}{c}{QSOS} & \multicolumn{2}{c}{RSOS} \\
\cmidrule(lr){3-4} \cmidrule(lr){5-6}
 & & opt & time & opt & time \\
\midrule

\multirow{3}{*}{20} 
& 1 & -36.6714 & 0.63  & -36.6714 & 2.27 \\
 & 2 & -47.8868 & 0.56 & -47.8868 & 2.28 \\
 & 3 & -47.5936 & 0.58 & -47.5936 & 2.45 \\
 \midrule
 
 \multirow{3}{*}{30}
& 1 & -77.0697 & 4.61  & -77.0697 & 23.2 \\
& 2 & -76.9280 & 5.93 & -76.9280 & 21.4  \\
& 3 & -81.6673 & 4.52  & -81.6673 & 21.4 \\
\midrule

 \multirow{3}{*}{40}
& 1 & -109.739 & 7.68 & -109.739 & 47.0 \\
& 2 & -99.2365 & 8.50 & -99.2365 & 58.4 \\
& 3 & -105.405 & 12.2 & -105.405 & 60.3 \\

\bottomrule
\end{tabular}
\end{table}
\FloatBarrier

\subsection{Application to quaternionic rotation synchronization}
Rotation synchronization estimates unknown absolute orientations from noisy pairwise relative measurements~\cite{hartley2013rotation}. Representing the orientations by unit quaternions reduces the task to recovering a vector $\q\in\mathbf{K}_n$.

Let $E$ be a measurement graph.
Given noisy relative-orientation measurements $Q_{ij}\approx q_i\bar q_j$ for $(i,j)\in E$, the rotation synchronization problem can be formulated as
\begin{equation}\label{qs}
\begin{cases}
\inf\limits_{\q\in\H^n} & \sum_{(i,j)\in E} \|q_i\bar q_j-Q_{ij}\|^2\\
\,\,\,\mathrm{s.t.}&\|q_i\| = 1,\quad i=1,\ldots,n.\\
\end{cases}
\end{equation}
Under the unit-norm constraints, the objective expands as
\[
\sum_{(i,j)\in E} \|q_i\bar q_j-Q_{ij}\|^2
=
\sum_{(i,j)\in E} \left(1+\|Q_{ij}\|^2\right)
-2\sum_{(i,j)\in E} \RR\!\left(\bar Q_{ij}q_i\bar q_j\right).
\]
Therefore, up to a positive affine transformation of the objective, \eqref{qs} is equivalent to the following quaternionic QCQP:
\begin{equation}\label{qs-qcqp}
\begin{cases}
\displaystyle \inf_{\q\in\H^n}
& \displaystyle -\frac{1}{2} \sum_{(i,j)\in E}
\left(\bar q_iQ_{ij}q_j+\bar q_j\bar Q_{ij}q_i\right), \\
\,\,\,\mathrm{s.t.}
& \|q_i\|=1,\quad i=1,\ldots,n.
\end{cases}
\end{equation}

In our experiments, the relative measurements are generated as $Q_{ij} = q_i \bar q_j + \varepsilon_{ij}$,
where $\varepsilon_{ij}$ is a random unit quaternion scaled by a fixed noise level $0.2$.
The measurement graph $E$ is sampled from the random graph model $G(n,p)$ with $p=0.2$.
For each $n\in\{20,40,60\}$, we generate three random instances and solve the first-order relaxations. The results are reported in Table~\ref{tablea2}. The column labeled ``GT value'' reports the objective value of \eqref{qs-qcqp} evaluated at the ground-truth quaternions.

The table shows that QSOS and RSOS return the same optimal values at the displayed precision for all instances, while QSOS is consistently faster. The relaxation values are also close to the objective values at the ground-truth quaternions. These experiments illustrate the computational advantage of QSOS on the tested synchronization instances.

\begin{table}[htbp!]
\centering
\caption{Results for Problem~\eqref{qs-qcqp}.}
\label{tablea2}
\begin{tabular}{ccccccc}
\toprule
\multirow{2}{*}{$n$} & \multirow{2}{*}{trial} & \multicolumn{2}{c}{QSOS} & \multicolumn{2}{c}{RSOS} & \multirow{2}{*}{GT value}\\
\cmidrule(lr){3-4} \cmidrule(lr){5-6}
 & & opt & time & opt & time \\
\midrule

\multirow{3}{*}{20}
% & \multirow{3}{*}{$\mathbf{K}_n$}
& 1 & -38.5388 & 0.11  & -38.5388 & 0.28 & -38.2176\\
& 2 & -37.4241 & 0.09 & -37.4241 & 0.24 & -37.1489 \\
& 3 & -55.3438 & 0.14 & -55.3438 & 0.31 & -55.0636\\
\midrule

\multirow{3}{*}{40} 
% & \multirow{3}{*}{$\mathbf{K}_n$}
& 1 & -134.807 & 3.72  & -134.807 & 21.2 & -134.265\\
 & 2 & -156.891 & 4.29 & -156.891 & 27.5 & -156.414 \\
 & 3 & -191.478 & 5.37 & -191.478& 22.3 & -190.857\\
 \midrule

 \multirow{3}{*}{60}
% & \multirow{3}{*}{$\mathbf{K}_n$}
& 1 & -341.507 & 13.3 & -341.507 & 140 & -340.621 \\
& 2 & -340.900 & 12.4 & -340.900 & 116 & -340.042 \\
& 3 & -367.784 & 12.2 & -367.784 & 102 & -366.823\\

\bottomrule
\end{tabular}
\end{table}

\section{Conclusions and future work}\label{sec:con}
We have developed a Moment--QSOS framework for quaternionic polynomial optimization over the standard quaternion algebra.  The central difficulty is that scalar quaternionic variables exhibit both noncommutative multiplication and hidden commutative relations among their real components.  To accommodate this structure, we introduced quaternionic sums of squares together with the scalar-identity ideals that distinguish formal free-algebra relations from identities valid under scalar quaternionic evaluation.  These ingredients provide a direct quaternionic counterpart of the Moment--SOS methodology without first expanding every quaternionic variable into four real variables.

On the algebraic side, we established Archimedean Positivstellens\"atze for two classes of quaternionic polynomials.  For real-coefficient polynomials, quotienting by the real scalar-identity ideal reduces the symmetric part of the algebra to a commutative setting, allowing the use of Jacobi's representation theorem.  For objectives with quaternionic middle coefficients and real-coefficient constraints, we developed an operator-valued lift and applied an Archimedean $*$-Positivstellensatz.  The univariate counterexample in Subsection~\ref{sec:counterexample-general-putinar} shows that the real-coefficient assumption on the constraints is essential for the present pure-word certificate cone, even in the presence of the scalar-identity quotient and a sphere constraint.

On the moment side, we characterized both real-valued and quaternion-valued full moment sequences by Riesz--Haviland-type positivity conditions and derived Archimedean moment--localizing matrix criteria.  For truncated quaternion-valued data, we proved a flat-extension theorem showing that positivity, flatness, and scalar consistency yield a finitely atomic scalar quaternionic representing measure.  These results lead to convergent Moment--QSOS hierarchies and finite-exactness criteria.  In addition, for quaternionic QCQPs we proved that the first-order quaternionic moment relaxation has the same optimal value as the first-order relaxation of the equivalent real QCQP.

The numerical experiments support the computational motivation for working directly with quaternionic structure.  Across the tested quadratic, quartic, and sparse instances, QSOS and the corresponding real SOS formulation return the same optimal values at the displayed precision, while the quaternionic formulation is substantially faster in the reported examples.  The experiments on the quaternion-based maximum margin criterion and rotation synchronization further illustrate how quaternionic models arising in applications can be handled directly within the proposed framework.

Several directions remain open.  First, the counterexample for quaternionic-coefficient constraints suggests that stronger certificate cones are needed beyond the present left-coefficient pure-word construction.  It would be useful to identify enlargements that recover an Archimedean Positivstellensatz while retaining tractable semidefinite representations.  Second, the finite-exactness theory can be sharpened by developing generic flatness conditions, more easily verifiable scalar-consistency criteria, and robust extraction procedures for higher-rank solutions.  Third, correlative and term sparsity, symmetry reduction, and chordal decomposition should be incorporated systematically into the Moment--QSOS hierarchy to extend the range of tractable problem sizes.  Finally, the equivalence proved here between quaternionic and real relaxations at the first order raises the question of when analogous equivalence or dominance relations hold at higher relaxation orders.  These questions are natural next steps toward a more complete theory and scalable implementation of quaternionic polynomial optimization.

\section*{Acknowledgements}
The authors acknowledge the use of AI to assist with brainstorming, mathematical development, and manuscript drafting. The final content, analysis, and conclusions remain the sole responsibility of the authors.

\section*{Funding}
This work was jointly funded by the National Key R\&D Program of China under grant No.~2023YFA1009401 and the Natural Science Foundation of China under grant Nos.~12201618 and 12171324.

\section*{Conflict of interest}
The authors declare that they have no conflicts of interest.

\section*{Data availability}
The authors confirm that all data generated or analyzed during this study are included in this article.

\bibliographystyle{siamplain}
\bibliography{refer}
\end{document}